%% file: main.tex
\documentclass[11pt]{article}

\input{ex_shared}

\begin{document}
	
\maketitle	

\begin{abstract} 
We consider the problem of estimating unknown parameters in stochastic differential equations driven by colored noise, which we model as a sequence of Gaussian stationary processes with decreasing correlation time. We aim to infer parameters in the limit equation, driven by white noise, given observations of the colored noise dynamics. We consider both the maximum likelihood and the stochastic gradient descent in continuous time estimators, and we propose to modify them by including filtered data. We provide a convergence analysis for our estimators showing their asymptotic unbiasedness in a general setting and asymptotic normality under a simplified scenario. 
\end{abstract}
 


\section{Introduction}

Estimating parameters from data in physical models is important in many applications. In recent years, model calibration has become an essential aspect of the overall mathematical modelling strategy~\cite{BrK22}. Often complex phenomena cannot be described by deterministic equations and some form of randomness needs to be taken into account, either due to model uncertainty/coarse-graining, parametric uncertainty or imprecise measurements. Often, noise in dynamical systems is modelled as white noise, i.e., as a mean-zero Gaussian stationary process that is delta-correlated in time, leading to Itô stochastic differential equations (SDEs). Inferring unknown parameters in diffusion models is therefore an essential problem which has been thoroughly investigated \cite{BaP80,Kut04,Bis08}. There are many applications, however, where modeling noise as an uncorrelated-in-time process is not accurate and where non-trivial (spatio-)temporal correlation structures need to be taken into account, leading to \emph{colored noise}. See., e.g., \cite{HoL84, HaJ94} and the references therein for applications of colored noise to physics, chemistry, and biology. Colored noise is modelled as a mean zero Gaussian stationary process with an exponential autocorrelation function, i.e., a stationary Ornstein-Uhlenbeck process~\cite[Chapter 8]{Pav14}. 

A natural question is whether the solution to an SDE driven by colored noise converges to the solution to the white noise-driven SDE, in the limit as the correlation time of the noise goes to zero. This is certainly the case for SDEs driven by \emph{additive colored noise}. In fact, it can be shown that, under standard dissipativity assumptions on the drift, that convergence is uniform in time, and error estimates can be obtained \cite{BlP78}. In particular, the stability and ergodic properties still hold in a vicinity of the  white noise regime, i.e., for SDEs driven by colored noise with a sufficiently small correlation time. The white noise limit becomes more complicated when the noise is multiplicative, in particular in the multidimensional case. In one dimension, the well known Wong--Zakai theorem~\cite{WoZ65, WoZ69} implies that, in the white noise limit, we obtain the Stratonovich SDE. However, this is not true in general in dimensions higher than one, except for the case of \emph{reducible} SDEs~\cite{Sus78, Sus91}. In general, this limiting procedure introduces an extra drift term, in addition to the stochastic integral, due to properties of the Lévy area of Brownian motion~\cite[Chapter 7]{IkW89}. This additional drift term, the \emph{Lévy area correction}, can have profound impact on qualitative properties of solutions to the SDE~\cite[Section 11.7.7]{PaS08}, \cite{DHP23}. In addition to the rigorous derivation of the Lévy area correction drift term~\cite[Chapter 7]{IkW89}, see also~\cite[Section 3.4]{FrH20} for a derivation using the theory of rough paths. The Lévy area correction drift term can also be derived, for multiplicative SDEs driven by colored noise, using multiscale analysis~\cite{VoW16, BoC13}, \cite[Section 5.1]{Pav14}. We also note that, for SDEs driven by colored multiplicative noise, and in the presence of an additional time scale, due to, e.g., delay effects or inertia, the white noise limit might lead to stochastic integrals that are neither Itô nor Stratonovich~\cite{KPS04, PaS05}. This is a phenomenon that has also been verified experimentally on a noisy electric circuit~\cite{VoW16, PMH13}.

The main goal of this paper is to develop inference methodologies for inferring parameters in the limiting, white noise SDE from data coming from the colored-noise driven equation, in the regime of small noise correlation time. This problem, and the approach adopted in this paper, is similar to the problem of estimating unknown parameters in homogenized models given observations of the slow variable from the corresponding multiscale dynamics \cite{PaS07, PPS09, ABJ13,GaS17}. Both in the problem studied in this paper and in the multiscale dynamics one, standard inference methodologies suffer from the problem of model misspecification: the data, obtained from observations of the full dynamics, is not compatible with the coarse-grained/homogenized/white noise model, but only at appropriate time scales, at which the limiting equation is valid. This leads to a systematic bias in the, e.g., maximum likelihood estimator (MLE). In this paper, in addition to the MLE we also consider stochastic gradient descent in continuous time (SGDCT), which allows for online learning from data. While the former is a well-established method for parameter estimation, the latter has been recently developed as an inference methodology for diffusion processes in \cite{SiS17}, and further analyzed in \cite{SiS20}, where a rate of convergence is obtained and a central limit theorem is proved. The SGDCT estimator was recently applied to McKean (mean field) SDEs in~\cite{SKP23}.  In contrast to stochastic gradient descent in discrete time, which has been studied in detail in, e.g., \cite{BMP90,BeT00,KuY03}, SGDCT consists in solving an SDE for the unknown parameter and thus performing online updates. In particular, it continuously follows a noisy descent direction along the path of the observations, yielding rapid convergence. 

Both the MLE and SGDCT, even if they perform well in problems where one is confronted with single-time-scale data, and for which model and data are compatible at all scales, fail in inferring parameters from observations driven by colored noise, due to systematic bias~\cite{PaS07}. Therefore, it is necessary to preprocess the data. Inspired by \cite{AGP21}, we propose to filter the data through an appropriate exponential kernel, and to then use the filtered data/process in the definition of the estimators. This approach was applied to the MLE for multiscale SDEs in \cite{AGP21}, where it is demonstrated that the filtering methodology outperforms and is more robust than classic subsampling techniques \cite{PaS07}. It was shown in this paper that inserting filtered data in the MLE allows one to correctly estimate the drift coefficient of the homogenized equation, when data are given from the slow variable of the multiscale system. In addition to the MLE, filtered data have then been used in combination with the continuous-time ensemble Kalman–Bucy filter \cite{Rei22}. The same methodology was also successfully applied to the case of discrete-time observations from multiscale dynamics \cite{APZ22}. In particular, based on \cite{KeS99} and the convergence of eigenvalues and eigenfunctions of the generator of multiscale dynamics to the corresponding eigenpairs of the generator of the homogenized process \cite{Zan23}, martingale estimating functions are first constructed and then modified taking into account filtered data. Moreover, a different filtering approach based on moving average is then presented in \cite{GaZ23}, still in the framework of multiscale diffusions.

The main contribution of this article is showing that coupling filtered data with either MLE or SGDCT is beneficial for their effectiveness when applied to stochastic models with colored noise. Our novel estimators based on filtered data are indeed able to learn linear parameters in the drift of the limit equation with white noise from the trajectories of the solution of the SDE driven by colored noise. We first consider the setting of additive noise and prove that both estimators are asymptotically unbiased in the limit of infinite data and when the correlation time of the colored noise vanishes. Our convergence analysis uses ideas from \cite{SiS20} and relies on the results in \cite{AGP21}. Moreover, in the linear one-dimensional setting we study the asymptotic normality of the estimators and derive formally their limit Gaussian distribution. In addition, we show that our methodology is not restricted to the case of additive noise, by analyzing a particular case of multiplicative noise which yields Lévy area correction. Despite the presence of this additional term, we show both theoretically and numerically that our estimators succeed in inferring the drift coefficient of the limit equation. This opens the possibility of extending these approaches to more general settings and more complex stochastic models driven by colored noise.

\paragraph{Outline.}

The rest of the paper is organized as follows. In \cref{sec:setting}, we introduce the general framework of stochastic processes driven by colored noise, and, in \cref{sec:estimation_additive}, we present the filtered data methodology, which is combined with MLE and SGDCT estimators, to infer parameters in SDEs driven by colored additive noise. Then, \cref{sec:proofs} is devoted to the convergence analysis of the proposed estimators, and, in \cref{sec:Levy}, we extend our methods to a particular case of diffusion with multiplicative colored noise, which yields Lévy area correction. In \cref{sec:numerical_experiments}, we demonstrate the effectiveness of our approach through numerical experiments and we present a preliminary analysis on the asymptotic normality of the estimators. Finally, in \cref{sec:conclusion}, we draw our conclusions and address possible future developments.

\section{Problem setting} \label{sec:setting}

We consider the framework of \cite[Section 5.1]{Pav14} and model colored noise as a Gaussian stationary diffusion process, i.e., the Ornstein--Uhlenbeck process. Consider the system of SDEs for the processes $X_t^\epl \in \R^d$, $Y_t^\epl \in \R^n$ in the time interval $[0,T]$
\begin{equation} \label{eq:system_multiD}
\begin{aligned}
\d X^\epl_t &= h(X^\epl_t) \dd t + g(X^\epl_t) \frac{Y^\epl_t}{\epl} \dd t, \\
\d Y^\epl_t &= - \frac{A}{\epl^2} Y^\epl_t \dd t + \frac{\sigma}{\epl} \dd W_t,
\end{aligned}
\end{equation}
with initial conditions $X_0^\epl \in \R^d$, $Y_0^\epl \in \R^n$, and where $h \colon \R^d \to \R^d$, $g \colon \R^d \to \R^{d \times n}$, $A \in \R^{n \times n}$, $\sigma \in \R^{n \times m}$, $W_t \in \R^m$ is a standard Brownian motion, and $0 < \epl \ll 1$ is the parameter which characterizes the colored noise. In the limit as $\epl \rightarrow 0$, the process $X_t^{\epl}$ converges to an SDE driven by white noise~\cite{PaS05}. Define $\Sigma = \sigma \sigma^\top \in \R^{n \times n}$ and assume that the eigenvalues of $A$ have positive real parts and that the matrix $\Sigma$ is positive definite. Then, the process $Y^\epl_t$ has a unique invariant measure which is Gaussian with zero mean and covariance matrix $\Sigma_\infty \in \R^{n \times n}$ which satisfies the steady-state variance equation 
\begin{equation}
A \Sigma_\infty + \Sigma_\infty A^\top = \Sigma.
\end{equation}
Moreover, define the quantities
\begin{equation}
\begin{aligned}
&B \colon \R^d \to \R^{d \times n}, \qquad &&B(x) = g(x) A^{-1}, \\
&R \colon \R^d \to \R^{d \times n}, \qquad &&R(x) = g(x) \Sigma_\infty, \\
&D \colon \R^d \to \R^{d \times d}, \qquad &&D(x) = R(x) B(x)^\top, \\
&b \colon \R^d \to \R^d, \qquad &&b(x) = \nabla \cdot D(x)^\top - B(x) \nabla \cdot R(x)^\top.
\end{aligned}
\end{equation}
In the limit as $\epl \to 0$ the process $X_t^\epl$ converges weakly in $\mathcal C^0([0,T];\R^d)$ to the solution $X_t$ of the Itô SDE (see, e.g., \cite{HMW19} and \cite[Section 5.1]{Pav14})
\begin{equation} \label{eq:limit_multiD}
\d X_t = (h(X_t) + b(X_t)) \dd t + \sqrt{2 D^S(X_t)} \dd W_t,
\end{equation}
with initial condition $X_0 = X_0^\epl$, where $W_t$ again denotes standard $d$-dimensional Brownian motion, $D^S$ denotes the symmetric part of the matrix $D$, i.e., $D^S(x) = (D(x) + D(x)^\top)/2$, and the additional drift term $b$ is called \emph{Lévy area correction}.

\begin{example} \label{ex:multiD}
Let $n=m=2$, $\alpha, \gamma, \eta$ be positive constants and consider system \eqref{eq:system_multiD} with 
\begin{equation}
A = \alpha I + \gamma J \qquad \text{and} \qquad \sigma = \sqrt{\eta} I,
\end{equation}
where
\begin{equation}
I = \begin{pmatrix} 1 & 0 \\ 0 & 1 \end{pmatrix} \qquad \text{and} \qquad
J = \begin{pmatrix} 0 & 1 \\ -1 & 0 \end{pmatrix}.
\end{equation}
Then we have
\begin{equation}
\Sigma = \eta I, \qquad \Sigma_\infty = \frac{\eta}{2\alpha} I, \qquad \text{and} \qquad A^{-1} = \frac1\rho A^\top \text{ where } \rho = \alpha^2 + \gamma^2,
\end{equation}
which give
\begin{equation}
B(x) = \frac{\alpha}{\rho} g(x) - \frac{\gamma}{\rho} g(x) J, \quad R(x) = \frac{\eta}{2\alpha} g(x), \quad D(x) = \frac{\eta}{2\rho} g(x) g(x)^\top + \frac{\eta \gamma}{2 \alpha \rho} g(x) J g(x)^\top,
\end{equation}
which in turn imply
\begin{equation}
D^S(x) = \frac{\eta}{2\rho} g(x) g(x)^\top.
\end{equation}
Therefore, we obtain
\begin{equation}
b(x) = \frac{\eta}{2\rho} \left( \nabla \cdot (g(x)g(x)^\top) - g(x) \nabla \cdot g(x)^\top \right) - \frac{\eta\gamma}{2\alpha\rho} \left( \nabla \cdot (g(x)Jg(x)^\top) - g(x)J \nabla \cdot g(x)^\top \right).
\end{equation}
\end{example}

\begin{remark} \label{rem:conversionItoStratonovich}
Since it will be employed in the following, we recall here the conversion from Stratonovich to Itô integrals and SDEs. Consider the $d$-dimensional Stratonovich SDE
\begin{equation} \label{eq:SDE_Stratonovich}
\d X_t = h(X_t) \dd t + g(X_t) \circ \d W_t,
\end{equation}
where $X_t \in \R^d$, $h \colon \R^d \to \R^d$, $g \colon \R^d \to \R^{d \times m}$ and $W_t$ is an $m$-dimensional Brownian motion. Then, equation \eqref{eq:SDE_Stratonovich} is equivalent to the Itô SDE
\begin{equation} \label{eq:SDE_Ito}
\d X_t = (h(X_t) + c(X_t)) \dd t + g(X_t) \dd W_t,
\end{equation}
with
\begin{equation} \label{eq:Ito2Stratonovich}
c(x) = \frac12 \left( \nabla \cdot (g(x)g(x)^\top) - g(x) \nabla \cdot g(x)^\top \right),
\end{equation}
and where the divergence of a matrix is computed per row. Moreover, we have the following conversion for integrals
\begin{equation}
\int_0^T X_t \circ \d Y_t = \int_0^T X_t \dd Y_t + \frac12 [X_t, Y_t]_T,
\end{equation}
where $[X_t, Y_t]_T$ denotes the quadratic covariation of two processes $X_t$ and $Y_t$.
\end{remark}

\subsection{One-dimensional case} \label{sec:one}

In the one-dimensional case, i.e., when $d=n=m=1$, a stronger result at the level of paths can be proven \cite[Result 5.1]{Pav14}. We notice that at stationarity the process $Y^\epl_t$ is a standard Gaussian process with auto-correlation function
\begin{equation}
\mathcal C(t,s) = \E \left[ Y^\epl_t Y^\epl_s \right] = \frac{\sigma^2}{2A} e^{-\frac{A}{\epl^2} \abs{t-s}},
\end{equation}
which implies 
\begin{equation}
\lim_{\epl \to 0} \E \left[ \frac{Y^\epl_t}{\epl} \frac{Y^\epl_s}{\epl} \right] = \frac{\sigma^2}{A^2} \delta(t),
\end{equation}
and therefore $Y^\epl_t/\epl$ converges to white noise $(\sigma/A) \dot W_t$ as $\epl$ vanishes. Moreover, by equation \eqref{eq:limit_multiD}, the process $X^\epl_t$ converges in the limit as $\epl \to 0$ to the solution $X_t$ of the Itô SDE
\begin{equation}
\d X_t = h(X_t) \dd t + \frac{\sigma^2}{2A^2} g(X_t) g'(X_t) \dd t + \frac{\sigma}{A} g(X_t) \dd W_t,
\end{equation}
which, due to \cref{rem:conversionItoStratonovich}, is equivalent to the Stratonovich SDE
\begin{equation}
\d X_t = h(X_t) \dd t + \frac{\sigma}{A} g(X_t) \circ \d W_t.
\end{equation}
Finally, through a repeated application of Itô's formula we can derive error estimates. In particular, it is possible to show that
\begin{equation} \label{eq:rate_weak_convergence}
\E \left[ (X^\epl_t - X_t)^2 \right]^{1/2} \le C\epl,
\end{equation}
where $C>0$ is a constant independent of $\epl$. We remark that for some particular examples pathwise convergence can also be proved in arbitrary dimensions \cite{KPS04,PaS05}.

\section{Parameter estimation for additive colored noise} \label{sec:estimation_additive}

In this section, we consider the problem of estimating parameters which appear linearly in the drift, given observations only from the fast process $(X_t^\epl)_{t\in[0,T]}$. In this section we focus on the (easier) case of additive noise case, where the Lévy area correction does not appear. Let the drift function $h$ depend linearly on an unknown matrix $\theta \in \R^{d \times \ell}$, i.e., $h(x) = \theta f(x)$ for some function $f \colon \R^d \to \R^\ell$, and let the diffusion function $g(x) = G$ be constant. Then, system \eqref{eq:system_multiD} reads
\begin{equation} \label{eq:system_multiD_additive}
\begin{aligned}
\d X^\epl_t &= \theta f(X^\epl_t) \dd t + G \frac{Y^\epl_t}{\epl} \dd t, \\
\d Y^\epl_t &= - \frac{A}{\epl^2} Y^\epl_t \dd t + \frac{\sigma}{\epl} \dd W_t,
\end{aligned}
\end{equation}
and the limit equation \eqref{eq:limit_multiD} becomes
\begin{equation} \label{eq:limit_multiD_additive}
\d X_t = \theta f(X_t) \dd t + \sqrt{2D^S} \dd W_t,
\end{equation}
where $D = G \Sigma_\infty A^{-\top} G^\top$. We assume the following conditions which guarantee ergodicity of the colored and white noise SDEs.

\begin{assumption} \label{as:ergodicity}
There exist constants $\mathfrak a, \mathfrak b, \mathfrak c > 0$, such that
\begin{equation}
\theta f(x) \cdot x \le \mathfrak a - \mathfrak b \norm{x}^2 \qquad \text{and} \qquad A y \cdot y \ge \mathfrak c \norm{y}^2,
\end{equation}
for all $x \in \R^d$ and $y \in \R^n$, and the dynamics $\dot x(t) = \theta f(x(t))$ has an asymptotically stable equilibrium point. Moreover, $f \in \mathcal C^2(\R^d)$ and the constants satisfy
\begin{equation}
2 \mathfrak b \mathfrak c - \norm{G}^2 > 0.
\end{equation}
\end{assumption}

\begin{example} \label{ex:OU}
Let us consider the simple case of the Ornstein--Uhlenbeck process in one dimension, i.e., set the drift function to $f(x) = - x$. Then the process $X^\epl_t$ can be computed analytically and is given by
\begin{equation}
X^\epl_t = X^\epl_0 e^{-\theta t} + \frac{G}{\epl} \int_0^t e^{-\theta(t-s)} Y^\epl_s \dd s.
\end{equation}
We remark that in this case the joint process $(X^\epl_t, Y^\epl_t)$ is a standard Gaussian process whose covariance matrix at stationarity is 
\begin{equation}
\mathcal C = \begin{pmatrix}
\dfrac{G^2 \sigma^2}{2\theta(A^2 - \epl^4\theta^2)} & \dfrac{\epl G \sigma^2}{2A(A + \epl^2 \theta)} \\
\dfrac{\epl G \sigma^2}{2A(A + \epl^2 \theta)} & \dfrac{\sigma^2}{2A}
\end{pmatrix},
\end{equation}
and the limit process is $X_t \sim \mathcal N(0, \frac{G^2\sigma^2}{2A^2\theta})$.
\end{example}

We recall that our goal is to infer the parameter $\theta$ given a realization $(X_t^\epl)_{t\in[0,T]}$ of the process defined by \eqref{eq:system_multiD_additive}. In the following subsections, we propose two different estimators: the first one is similar to the MLE, while the second one is based on the SGDCT.

\subsection{Maximum likelihood type estimator}

Let us first focus on equation \eqref{eq:limit_multiD_additive} and construct an estimator for $\theta$ given a trajectory $(X_t)_{t\in[0,T]}$ from the same limit equation. Inspired by the MLE, we propose the following estimator
\begin{equation} \label{eq:MLE_ok}
\widehat \theta(X,T) = \left( \int_0^T \d X_t \otimes f(X_t) \right) \left( \int_0^T f(X_t) \otimes f(X_t) \dd t \right)^{-1},
\end{equation}
where $\otimes$ denotes the outer product. 

\begin{remark}
The estimator in equation \eqref{eq:MLE_ok} is not directly obtained by maximizing the likelihood function, in fact the actual MLE would require the knowledge of the diffusion term. However, since the shape of the estimator is similar and it is obtained by replacing the diffusion term by the identity matrix, in the following we will always refer to it as the MLE estimator.
\end{remark}

In the following result, whose proof can be found in \cref{sec:proofs}, we show that the MLE estimator is asymptotically unbiased in the limit of infinite time. For clarity of the exposition, \cref{as:positive_definite}$(i)$, which is required in \cref{thm:MLE_ok,pro:MLE_ko} below for the well-posedness of the estimator, will be stated in the next section together with the corresponding \cref{as:positive_definite}$(ii)$ for filtered data.

\begin{theorem} \label{thm:MLE_ok}
Let \cref{as:ergodicity,as:positive_definite}$(i)$ hold and let $\widehat \theta(X,T)$ be defined in \eqref{eq:MLE_ok}. Then, it holds
\begin{equation}
\lim_{T\to\infty} \widehat \theta(X,T) = \theta, \qquad a.s.
\end{equation}
\end{theorem}

Let us now consider the more interesting problem of estimating the parameter $\theta$ given observations from the system with colored noise. Since the process $X_t^\epl$ is close in a weak sense to the process $X_t$, it is tempting to replace $X_t$ by $X_t^\epl$ in the definition of the previous estimator, which yields
\begin{equation} \label{eq:MLE_ko}
\widehat \theta(X^\epl,T) = \left( \int_0^T \d X_t^\epl \otimes f(X_t^\epl) \right) \left( \int_0^T f(X_t^\epl) \otimes f(X_t^\epl) \dd t \right)^{-1}.
\end{equation}
However, this estimator fails even when the process $X_t$ is one-dimensional. In particular, the estimator vanishes in the limit of infinite observation time, as stated in the following proposition, whose proof is presented in \cref{sec:proofs}.

\begin{proposition} \label{pro:MLE_ko}
Let \cref{as:ergodicity,as:positive_definite}$(i)$ hold and let $\widehat \theta(X^\epl,T)$ be defined in \eqref{eq:MLE_ko}. Then, if $d=1$  it holds
\begin{equation}
\lim_{T\to\infty} \widehat \theta(X^\epl,T) = 0, \qquad a.s.
\end{equation}
\end{proposition}

Therefore, in the next sections we propose two different approaches to infer the unknown coefficients in presence of colored noise, which rely on filtering the data with an appropriate kernel of the exponential family.

\subsection{The filtered data approach}

In \cite{AGP21} exponential filters were employed to remove the bias from the MLE when estimating parameters in the homogenized SDE given observations of the slow variable in the multiscale dynamics. Motivated by this paper, we introduce the kernel $k \colon [0,\infty) \to \R$ given by
\begin{equation} \label{eq:exp_kernel}
k(r) = \frac1\delta e^{-r/\delta},
\end{equation}
where $\delta > 0$ is a parameter measuring the filtering width. As it will become transparent later,  we need to assume the following condition on the filtering width in order to ensure the required ergodicity properties.

\begin{assumption} \label{as:assumption_delta}
The parameter $\delta$ in \eqref{eq:exp_kernel} satisfies
\begin{equation} 
\delta > \frac{\mathfrak c}{2 \mathfrak b \mathfrak c - \norm{G}^2},
\end{equation}
where $\mathfrak b, \mathfrak c, \norm{G}$ are given in \cref{as:ergodicity}.
\end{assumption}

\begin{remark}
Notice that the condition given in \cref{as:assumption_delta} on the filtering width $\delta$ is not very restrictive. In fact, if we assume, e.g., the function $f$ to be $f(x) = x^{2r+1}$ for an integer $r>0$, then the coefficient $\mathfrak b > 0$ can be taken arbitrarily large, and therefore the parameter $\delta$ can be chosen along the entire positive real axis.
\end{remark}

We then define filtered data by convolving the original data, driven by both colored and white noise, with the exponential kernel, and we obtain
\begin{equation} \label{eq:filtered_data}
Z^\epl_t = \int_0^t \frac1\delta e^{-(t-s)/\delta} X_s^\epl \dd s, \qquad Z_t = \int_0^t \frac1\delta e^{-(t-s)/\delta} X_s \dd s.
\end{equation}
The new estimators based on filtered data are obtained replacing one instance of the original data by the filtered data in both the terms in \eqref{eq:MLE_ok} for processes driven by white noise and \eqref{eq:MLE_ko} for processes driven by colored noise. Before defining the estimators and presenting the main theoretical results of this section, we introduce a technical assumption, which corresponds to the strong convexity of the objective function in~\cite{SiS20}. This assumption is related to the nondegeneracy of the Fisher information matrix studied in \cite{DeH23} for mean field SDEs, see also~\cite{PaZ22b} and, in our context, with the identifiability of the drift parameters $\theta$ from observations of $X_t^{\epl}$.

\begin{assumption} \label{as:positive_definite}
There exists a constant $K>0$ such that for all $v \in \R^\ell$
\begin{equation}
\begin{aligned}
(i) &\quad v^\top \E^\mu \left[ f(X) \otimes f(X) \right] v \ge K \norm{v}^2, \\
(ii) &\quad v^\top \E^\mu \left[ f(X) \otimes f(Z) \right] v \ge K \norm{v}^2,
\end{aligned}
\end{equation}
where $\mu$ is the invariant measure of the joint process $(X_t,Z_t)$, which will be rigorously defined in \cref{sec:ergodic}. Moreover, the coefficient $a$ in \eqref{eq:learning_rate} is such that $aK > 1$.
\end{assumption}

\begin{remark} \label{rem:positive_definite}
Notice that \cref{as:positive_definite} implies that the same properties hold true also for the processes $X_t^\epl$ and $Z_t^\epl$ driven by colored noise if $\epl$ is sufficiently small. In fact, we have for all $v \in \R^\ell$
\begin{equation} \label{eq:difference_E}
\begin{aligned}
v^\top \E^{\mu^\epl} \left[ f(X^\epl) \otimes f(Z^\epl) \right] v &= v^\top \E^\mu \left[ f(X) \otimes f(Z) \right] v \\
&\quad + v^\top \left( \E^{\mu^\epl} \left[ f(X^\epl) \otimes f(Z^\epl) \right] - \E^\mu \left[ f(X) \otimes f(Z) \right] \right) v.
\end{aligned}
\end{equation}
Then, letting $0 < \mathfrak d < K$, by the convergence in law of the process $(X_t^\epl,Z_t^\epl)$ to $(X_t,Z_t)$, there exists $\epl_0 > 0$ such that for all $\epl < \epl_0$ we get
\begin{equation}
\begin{aligned}
&\abs{v^\top \left( \E^{\mu^\epl} \left[ f(X^\epl) \otimes f(Z^\epl) \right] - \E^\mu \left[ f(X) \otimes f(Z) \right] \right) v} \\
&\hspace{2cm} \le \norm{\E^{\mu^\epl} \left[ f(X^\epl) \otimes f(Z^\epl) \right] - \E^\mu \left[ f(X) \otimes f(Z) \right]} \norm{v}^2 \le \mathfrak d \norm{v^2},
\end{aligned}
\end{equation}
which together with equation \eqref{eq:difference_E} and \cref{as:positive_definite} gives
\begin{equation}
v^\top \E^{\mu^\epl} \left[ f(X^\epl) \otimes f(Z^\epl) \right] v \ge (K - \mathfrak d) \norm{v}^2 \eqdef K_{\mathfrak d} \norm{v}^2,
\end{equation}
where we notice that $K_{\mathfrak d}$ can be chosen arbitrarily close to $K$. Similarly, we also obtain 
\begin{equation}
v^\top \E^{\mu^\epl} \left[ f(X^\epl) \otimes f(X^\epl) \right] v \ge K_{\mathfrak d} \norm{v}^2.
\end{equation}
Moreover, we remark that even if \cref{as:positive_definite} is stated for vectors $v \in \R^\ell$, then it also implies that for any matrix $V \in \R^{\ell \times \ell}$ we have 
\begin{equation}
\begin{aligned}
V \E^\mu \left[ f(X) \otimes f(X) \right] \colon V &\ge K \norm{V}^2, \\
V \E^\mu \left[ f(X) \otimes f(Z) \right] \colon V &\ge K \norm{V}^2, \\
V \E^{\mu^\epl} \left[ f(X^\epl) \otimes f(X^\epl) \right] \colon V &\ge K_{\mathfrak d} \norm{V}^2, \\
V \E^{\mu^\epl} \left[ f(X^\epl) \otimes f(Z^\epl) \right] \colon V &\ge K_{\mathfrak d} \norm{V}^2,
\end{aligned}
\end{equation}
where $\colon$ and $\norm{\cdot}$ stand for the Frobenius scalar product and norm, respectively. In fact, e.g., for the first inequality, but similarly for the others, we have
\begin{equation}
V \E^\mu \left[ f(X) \otimes f(X) \right] \colon V = \sum_{i=1}^{\ell} v_i^\top \E^\mu \left[ f(X) \otimes f(X) \right] v_i \ge K \sum_{i=1}^{\ell} \norm{v_i}^2 = K \norm{V}^2,
\end{equation}
where $v_1^\top, \dots, v_\ell^\top$ are the rows of the matrix $V$, and the inequality follows from \cref{as:positive_definite}.
\end{remark}

\begin{remark}
The hypotheses about the coercivity of the matrices $\E^\mu \left[ f(X) \otimes f(X) \right]$ and $\E^\mu \left[ f(X) \otimes f(Z) \right]$ in \cref{as:positive_definite} are not a strong limitations of the scope of the next theorems. In fact, the matrix $\E^\mu \left[ f(X) \otimes f(X) \right]$ is already symmetric positive semi-definite, so we are only requiring the matrix to have full rank. Moreover, the process $Z_t$ is a filtered version of the process $X_t$ and therefore in practice we expect $f(X_t)$ and $f(Z_t)$ to behave similarly for the majority of time, in particular when the filtering width $\delta$ is sufficiently small, implying the property for $\E^\mu \left[ f(X) \otimes f(Z) \right]$. Hence, we expect these assumptions to hold true in all concrete examples, as we also observed in our numerical experiments.
\end{remark}

We now study the performance of the MLE estimator \eqref{eq:MLE_ok} in the presence of observations from the limit equation. In particular, we define the estimator 
\begin{equation} \label{eq:exp_ok}
\widehat \theta_{\mathrm{exp}}^\delta(X,T) = \left( \int_0^T \d X_t \otimes f(Z_t) \right) \left( \int_0^T f(X_t) \otimes f(Z_t) \dd t \right)^{-1}
\end{equation}
and prove that it is asymptotically unbiased in the limit of infinite data. The following result is therefore analogous to \cref{thm:MLE_ok}, and shows that our modification does not affect the unbiasedness of the MLE.

\begin{theorem} \label{thm:exp_ok}
Let \cref{as:ergodicity,as:assumption_delta,as:positive_definite}$(ii)$ hold and let $\widehat \theta_{\mathrm{exp}}^\delta(X,T)$ be defined in \eqref{eq:exp_ok}. Then, it holds
\begin{equation}
\lim_{T\to\infty} \widehat \theta_{\mathrm{exp}}^\delta(X,T) = \theta, \qquad a.s.
\end{equation}
\end{theorem}

We now focus on our main problem of interest, i.e., the case when the observations are generated by the system with colored noise, for which the corresponding estimator is given by
\begin{equation} \label{eq:exp_ok_eps}
\widehat \theta_{\mathrm{exp}}^\delta(X^\epl,T) = \left( \int_0^T \d X_t^\epl \otimes f(Z_t^\epl) \right) \left( \int_0^T f(X_t^\epl) \otimes f(Z_t^\epl) \dd t \right)^{-1}.
\end{equation}
The following theorem, which is the main result of this section and its proof can be found in \cref{sec:proofs}, shows that, in contrast with the estimator in \eqref{eq:MLE_ko}, the proposed estimator is asymptotically unbiased in the limit as $\epl\to0$ and $T\to\infty$.

\begin{theorem} \label{thm:exp_ok_eps}
Let \cref{as:ergodicity,as:assumption_delta,as:positive_definite}$(ii)$ hold and let $\widehat \theta_{\mathrm{exp}}^\delta(X^\epl,T)$ be defined in \eqref{eq:exp_ok_eps}. Then, it holds
\begin{equation}
\lim_{\epl\to0} \lim_{T\to\infty} \widehat \theta_{\mathrm{exp}}^\delta(X^\epl,T) = \theta, \qquad a.s.
\end{equation}
\end{theorem}

In particular, due to \cref{thm:exp_ok_eps}, the estimator in \eqref{eq:exp_ok_eps} gives a straightforward methodology for inferring the unknown drift coefficient given observations from the system \eqref{eq:system_multiD_additive} with colored noise. Moreover, in the next section we show that this methodology based on filtered data can be applied to a different estimator, i.e., the SGDCT.

\begin{remark}
\cref{as:positive_definite} and \cref{rem:positive_definite} are important for the well-posedness of the MLE estimators \eqref{eq:MLE_ok}, \eqref{eq:MLE_ko}, \eqref{eq:exp_ok}, \eqref{eq:exp_ok_eps}. In fact, by the ergodic theorem we have, e.g., for the estimator \eqref{eq:exp_ok_eps}, and similarly for the others, that
\begin{equation}
\lim_{T \to \infty} \frac1T \int_0^T f(X_t^\epl) \otimes f(Z_t^\epl) \dd t = \E^{\mu^\epl} [f(X^\epl) \otimes f(Z^\epl)].
\end{equation}
Therefore, the positive-definiteness of the expectations in \cref{as:positive_definite} guarantees the invertibility of the matrices in the definition of the estimators, for sufficiently large time $T$.
\end{remark}

\begin{remark}
The exponential kernel in \eqref{eq:exp_kernel} is not the only possible choice for the filtering procedure. A different technique based on moving averages is presented in \cite{GaZ23}, where it is applied to multiscale diffusions. We expect that the analysis presented in this paper applies to the moving average-based filtering methodology, when applied to systems driven by colored noise. We will leave this for future work.
\end{remark}
 
\subsection{Coupling filtered data with SGDCT} \label{sec:SGDCT}

In this section, we present a different approach for inferring parameters in SDEs driven by colored noise. In particular, we employ the SGDCT method introduced in \cite{SiS17,SiS20}. We first consider the SGDCT as an inference methodology for estimating the parameter $\theta$ in the model \eqref{eq:limit_multiD_additive} given observations $(X_t)_{t\in[0,T]}$ from the same equation. The SGDCT consists of the following system of SDEs for the unknown parameter
\begin{equation} \label{eq:SDE_base_ok}
\begin{aligned}
\d \widetilde \theta_t &= \xi_t \d \mathcal I_t \otimes f(X_t), \\
\d \mathcal I_t &= \d X_t - \widetilde \theta_t f(X_t) \dd t,
\end{aligned}
\end{equation}
where $\mathcal I_t$ is called the innovation and $\xi_t$ is the learning rate, which has the form
\begin{equation} \label{eq:learning_rate}
\xi_t = \frac{a}{b+t},
\end{equation}
for some constants $a,b>0$, and with initial condition $\widetilde \theta_0 = \theta_0 \in \R^{d \times \ell}$. We note that the SGDCT estimator~\eqref{eq:SDE_base_ok} is in fact quite similar to the MLE estimator~\eqref{eq:MLE_ok}, with the additional feature of having introduced the learning rate $\xi_t$. Proceeding similarly to the previous section, we define analogous estimators for the reduced model using first data from the multiscale process driven by colored noise, and then employing filtered data. It turns out that for the latter, it is important to not modify the innovation term in order to keep the estimator asymptotically unbiased, as we will see in the theoretical analysis. The three considered estimators correspond the following system of SDEs:
\begin{align}
\d \widetilde \theta_t^\epl &= \xi_t \d \mathcal I_t^\epl \otimes f(X_t^\epl), \hspace{1.6cm} \d \mathcal I_{t}^\epl = \d X_t^\epl - \widetilde \theta_t^\epl f(X_t^\epl) \dd t, \label{eq:SDE_base_ko} \\
\d \widetilde \theta_{\mathrm{exp},t}^\delta &= \xi_t \d \mathcal I_{\mathrm{exp},t}^\delta \otimes f(Z_t), \qquad \d \mathcal I_{\mathrm{exp},t}^\delta = \d X_t - \widetilde \theta_{\mathrm{exp},t}^\delta f(X_t) \dd t, \label{eq:SDE_exp_ok} \\
\d \widetilde \theta_{\mathrm{exp},t}^{\delta,\epl} &= \xi_t \d \mathcal I_{\mathrm{exp},t}^{\delta,\epl} \otimes f(Z_t^\epl), \qquad \d \mathcal I_{\mathrm{exp},t}^{\delta,\epl} = \d X_t^\epl - \widetilde \theta_{\mathrm{exp},t}^{\delta,\epl} f(X_t^\epl) \dd t, \label{eq:SDE_exp_ok_eps}
\end{align}
with initial conditions $\widetilde \theta_0^\epl = \widetilde \theta_{\mathrm{exp},0}^\delta = \widetilde \theta_{\mathrm{exp},0}^{\delta,\epl} = \theta_0 \in \R^{d \times \ell}$. Consider now the one-dimensional case that was introduced in~\cref{sec:one}, i.e., when $d=\ell=n=m=1$. The solutions of the SDEs \eqref{eq:SDE_base_ok}, \eqref{eq:SDE_base_ko}, \eqref{eq:SDE_exp_ok}, \eqref{eq:SDE_exp_ok_eps} have a closed form expression, and therefore the estimators can be computed analytically and are given by
\begin{equation} \label{eq:SGDCT_1D_explicit}
\begin{aligned}
\widetilde \theta_t &= \theta + (\theta_0 - \theta) e^{-\int_0^t \xi_r f(X_r)^2 \dd r} + \sqrt{2D^S} \int_0^t \xi_s e^{-\int_s^t \xi_r f(X_r)^2 \dd r} f(X_s) \dd W_s, \\
\widetilde \theta_t^\epl &= \theta + (\theta_0 - \theta) e^{-\int_0^t \xi_r f(X_r^\epl)^2 \dd r} + G \int_0^t \xi_s e^{-\int_s^t \xi_r f(X_r^\epl)^2 \dd r} f(X_s^\epl) \frac{Y_s^\epl}{\epl} \dd s, \\
\widetilde \theta_{\mathrm{exp},t}^\delta &= \theta + (\theta_0 - \theta) e^{-\int_0^t \xi_r f(X_r) f(Z_r) \dd r} + \sqrt{2D^S} \int_0^t \xi_s e^{-\int_s^t \xi_r f(X_r) f(Z_r) \dd r} f(Z_s) \dd W_s, \\
\widetilde \theta_{\mathrm{exp},t}^{\delta,\epl} &= \theta + (\theta_0 - \theta) e^{-\int_0^t \xi_r f(X_r^\epl) f(Z_r^\epl) \dd r} + G \int_0^t \xi_s e^{-\int_s^t \xi_r f(X_r^\epl) f(Z_r^\epl) \dd r} f(Z_s^\epl) \frac{Y_s^\epl}{\epl} \dd s.
\end{aligned}
\end{equation}
In the next section we show that MLE and SGDCT behave similarly in the limit at $t \rightarrow +\infty$, and therefore also the estimators based on SGDCT are asymptotically unbiased. In particular, we have the following convergence results.

\begin{theorem} \label{thm:SGDCT_ok}
Let $\widetilde \theta_t$ be defined in \eqref{eq:SDE_base_ok}. Under \cref{as:ergodicity,as:positive_definite}$(i)$, it holds
\begin{equation}
\lim_{t\to\infty} \widetilde \theta_t = \theta, \qquad \text{in } L^2.
\end{equation}
\end{theorem}

\begin{proposition} \label{pro:SGDCT_ko}
Let $\widetilde \theta_t^\epl$ be defined in \eqref{eq:SDE_base_ko}. Under \cref{as:ergodicity,as:positive_definite}$(i)$, it holds
\begin{equation}
\lim_{t\to\infty} \widetilde \theta_t^\epl = 0, \qquad \text{in } L^2.
\end{equation}
\end{proposition}

\begin{theorem} \label{thm:SGDCT_exp_ok}
Let $\widetilde \theta_{\mathrm{exp},t}^\delta$ be defined in \eqref{eq:SDE_exp_ok}. Under \cref{as:ergodicity,as:assumption_delta,as:positive_definite}$(ii)$, it holds
\begin{equation}
\lim_{t\to\infty} \widetilde \theta_{\mathrm{exp},t}^\delta = \theta, \qquad \text{in } L^2.
\end{equation}
\end{theorem}

\begin{theorem} \label{thm:SGDCT_exp_ok_eps}
Let $\widetilde \theta_{\mathrm{exp},t}^{\delta,\epl}$ be defined in \eqref{eq:SDE_exp_ok_eps}. Under \cref{as:ergodicity,as:assumption_delta,as:positive_definite}$(ii)$, it holds
\begin{equation}
\lim_{\epl\to0} \lim_{t\to\infty} \widetilde \theta_{\mathrm{exp},t}^{\delta,\epl} = \theta, \qquad \text{in } L^2.
\end{equation}
\end{theorem}

\begin{remark} \label{rem:compact_state_space}
The proofs of these theorems, which are outlined in \cref{sec:proofs}, are based on the additional assumption that the state space for equations \eqref{eq:system_multiD_additive} and \eqref{eq:limit_multiD_additive} is a compact phase space, namely the $d$-dimensional torus $\mathbb T^d$. In particular, we consider the wrapping of the stochastic processes on the torus \cite{GSM19}. Indeed the theoretical analysis, and especially \cref{lem:Poisson_expansion}, are based on the results for the Poisson problem presented in \cite[Section 4]{MST10}, in which the Poisson equation on the torus is considered. We believe that it should be possible, at the expense of introducing additional technical difficulties, to modify our proof so that it applies to the case where the state space is the whole $\R^d$. In this case the solution of the Poisson PDE and its derivatives are not expected to be bounded; we will need to control the time spent by the process outside a compact subset of the phase space and to use the results in \cite{PaV01,PaV03,PaV05}, see also~\cite{Wu_09}. The numerical experiments in \cref{sec:numerical_experiments} indeed show that our estimators work when the state space is $\R^d$. Since the focus of this paper is on parameter estimation, we chose to dispense with all these technical difficulties by considering the case where our processes are defined on the torus.
\end{remark}

The estimators introduced here for SDEs driven by colored additive noise will then be successfully applied to numerical examples in \cref{sec:num_additive}, where we will observe that the exact unknown parameter can be accurately approximated.

\section{Convergence analysis} \label{sec:proofs}

This section is devoted to the proofs of the main results presented in the previous section, i.e., \cref{thm:MLE_ok,pro:MLE_ko,thm:exp_ok,thm:exp_ok_eps,thm:SGDCT_ok,pro:SGDCT_ko,thm:SGDCT_exp_ok,thm:SGDCT_exp_ok_eps}, which show the asymptotic (un)biasedness of the proposed estimators. The convergence analysis is divided in three parts. We first study the ergodic properties of the stochastic processes under investigation together with the filtered data, then we study the infinite time limit of our SGDCT estimators, and, finally, we focus on the proofs of the main results of this work. We remark that $\norm{A}$ denotes the Frobenius norm of a matrix $A$ throughout this section. 

\subsection{Ergodic properties} \label{sec:ergodic}

We consider system \eqref{eq:system_multiD_additive} and the limit equation \eqref{eq:limit_multiD_additive} together with the additional equations given by the filtered data \eqref{eq:filtered_data}, i.e.,
\begin{equation} \label{eq:system_filtered_eps}
\begin{aligned}
\d X^\epl_t &= \theta f(X^\epl_t) \dd t + G \frac{Y^\epl_t}{\epl} \dd t, \\
\d Y^\epl_t &= - \frac{A}{\epl^2} Y^\epl_t \dd t + \frac{\sigma}{\epl} \dd W_t, \\
\d Z^\epl_t &= \frac1\delta (X^\epl_t - Z^\epl_t) \dd t,
\end{aligned}
\end{equation}
and
\begin{equation} \label{eq:system_filtered}
\begin{aligned}
\d X_t &= \theta f(X_t) \dd t + \sqrt{2D^S} \dd W_t, \\
\d Z_t &= \frac1\delta (X_t - Z_t) \dd t,
\end{aligned}
\end{equation}
respectively.
We first verify that the measures induced by the stochastic processes admit
smooth densities with respect to the Lebesgue measure. Since white noise is present only in one component, this is a consequence of the theory of hypoellipticity, as shown in the next lemma.

\begin{lemma} \label{lem:Lebesgue}
Let $\mu_t^\epl$ and $\mu_t$ be the measures at time $t$ induced by the joint processes $(X_t^\epl,Y_t^\epl,Z_t^\epl)$ and $(X_t,Z_t)$ given by equations \eqref{eq:system_filtered_eps} and \eqref{eq:system_filtered}, respectively. Then, the measures $\mu_t^\epl$ and $\mu_t$ admit smooth densities $\rho^\epl_t$ and $\rho_t$ with respect to the Lebesgue measure.
\end{lemma}
\begin{proof}
Let us first consider the system driven by colored noise. The generator of the joint process $(X_t^\epl, Y_t^\epl, Z_t^\epl)$ is
\begin{equation}
\begin{aligned}
\mathcal L^\epl &= \theta f(x) \cdot \nabla_x + \frac1\epl G y \cdot \nabla_x - \frac1{\epl^2} A y \cdot \nabla_y + \frac1\delta (x-z) \cdot \nabla_z + \frac1{2\epl^2} \sigma \sigma^\top \colon \nabla^2_{yy} \\
&\eqdef \mathcal X_0 + \frac1{2\epl^2} \sum_{i=1}^n \mathcal X_i^2,
\end{aligned}
\end{equation}
where
\begin{equation}
\begin{aligned}
\mathcal X_0 &= \theta f(x) \cdot \nabla_x + \frac1\epl G y \cdot \nabla_x - \frac1{\epl^2} A y \cdot \nabla_y + \frac1\delta (x-z) \cdot \nabla_z, \\
\mathcal X_i &= \sigma_i \cdot \nabla_y, \qquad i = 1, \dots, n,
\end{aligned}
\end{equation}
and where $\sigma_i$, $i = 1, \dots, n$, are the columns of the matrix $\sigma$. The commutator $[\mathcal X_0, \mathcal X_i]$ is
\begin{equation}
[\mathcal X_0, \mathcal X_i] = - \frac1\epl \sigma_i \cdot G^\top \nabla_x + \frac1{\epl^2} \sigma_i \cdot A^\top \nabla_y,
\end{equation}
and the commutator $[\mathcal X_0, [\mathcal X_0, \mathcal X_i]]$ is
\begin{equation}
[\mathcal X_0, [\mathcal X_0, \mathcal X_i]] = \frac1\epl \sigma_i \cdot G^\top \theta \nabla_x f(x) \nabla_x - \frac1{\epl^3} \sigma_i \cdot A^\top G^\top \nabla_x + \frac1{\epl\delta} \sigma_i \cdot G^\top \nabla_z + \frac1{\epl^4} \sigma_i \cdot (A^\top)^2 \nabla_y.
\end{equation}
Therefore, for any point $(x,y,z) \in \R^{2d+n}$, the set
\begin{equation}
\mathcal H = \mathrm{Lie}\left( \mathcal X_i, [\mathcal X_0, \mathcal X_i], [\mathcal X_0, [\mathcal X_0, \mathcal X_i]]; i = 1, \dots, n \right)
\end{equation}
spans the tangent space of $\R^{2d+n}$ at $(x,y,z)$. The result then follows from Hörmander’s theorem (see, e.g., \cite[Chapter 6]{Pav14}). The proof of hypoellipticity for the limit system $(X_t,Z_t)$ is similar and we omit the details.
\end{proof}

We are now interested in the limiting properties of the measures $\mu_t^\epl$ and $\mu_t$, and, in particular, in the stationary Fokker--Planck equations of the systems of SDEs. The next lemma guarantees that the joint processes $(X_t^\epl,Y_t^\epl,Z_t^\epl)$ and $(X_t,Z_t)$ are ergodic. 

\begin{lemma} \label{lem:FP}
Under \cref{as:ergodicity,as:assumption_delta}, the processes $(X_t^\epl,Y_t^\epl,Z_t^\epl)$ and $(X_t,Z_t)$ given by equations \eqref{eq:system_filtered_eps} and \eqref{eq:system_filtered}, are ergodic with unique invariant measures $\mu^\epl$ and $\mu$, whose densities $\rho^\epl$ and $\rho$ with respect to the Lebesgue measure solve the stationary Fokker--Planck equations
\begin{equation} \label{eq:FP_colored}
\begin{aligned}
- \nabla_x \cdot \left( \theta f(x) \rho^\epl(x,y,z) \right) - \frac1\epl \nabla_x \cdot \left( G y \rho^\epl(x,y,z) \right) - \frac1\delta \nabla_z \cdot \left( (x-z) \rho^\epl(x,y,z) \right) & \\
+ \frac1{\epl^2} \nabla_y \cdot \left( A y \rho^\epl(x,y,z) \right) + \frac{1}{2 \epl^2} \sigma \sigma^\top \colon \nabla_y^2 \rho^\epl(x,y,z) &= 0,
\end{aligned}
\end{equation}
and
\begin{equation} \label{eq:FP_limit}
- \nabla_x \cdot \left( \theta f(x) \rho(x,z) \right) - \frac1\delta \nabla_z \cdot \left( (x-z) \rho(x,z) \right) + D^S \colon \nabla_x^2 \rho(x,z) = 0,
\end{equation}
where $\colon$ denotes the Frobenius inner product between two matrices. Moreover, the measure $\mu^\epl$ converges weakly to $\mu$ as $\epl$ goes to zero.
\end{lemma}
\begin{proof}
Let us first consider the system driven by colored noise. \cref{lem:Lebesgue} guarantees that the Fokker--Planck equation can be written directly from the system \eqref{eq:system_filtered_eps}. In order to prove ergodicity, consider the function
\begin{equation}
\mathcal S(x,y,z) = \begin{pmatrix} \theta f(x) + \frac1\epl Gy \\ -\frac1{\epl^2} Ay \\ \frac1\delta (x-z) \end{pmatrix} \cdot \begin{pmatrix} x \\ y \\ z \end{pmatrix} = \theta f(x) \cdot x + \frac1\epl G y \cdot x - \frac1{\epl^2} A y \cdot y + \frac1\delta x \cdot z - \frac1\delta \norm{z}^2.
\end{equation}
Due to \cref{as:ergodicity} and by Young's inequality we then have for all $\gamma_1,\gamma_2>0$
\begin{equation}
\mathcal S(x,y,z) \le \mathfrak a - \left( \mathfrak b - \frac{\norm{G}^2}{2\gamma_1} - \frac{1}{2\delta\gamma_2} \right) \norm{x}^2 - \frac1{\epl^2} \left( \mathfrak c - \frac{\gamma_1}2 \right) \norm{y}^2 - \frac1\delta \left( 1 - \frac{\gamma_2}2 \right) \norm{z}^2.
\end{equation}
Choosing $\gamma_1 = \mathfrak c$ and $\gamma_2 = 1$ we get
\begin{equation}
\begin{aligned}
\mathcal S(x,y,z) &\le \mathfrak a - \left( \mathfrak b - \frac{\norm{G}^2}{2 \mathfrak c} - \frac1{2\delta} \right) \norm{x}^2 - \frac{\mathfrak c}{2 \epl^2} \norm{y}^2 - \frac1{2\delta} \norm{z}^2 \\
&\le \mathfrak a - \min \left\{ \mathfrak b - \frac{\norm{G}^2}{2 \mathfrak c} - \frac1{2\delta}, \frac{\mathfrak c}{2 \epl^2}, \frac1{2\delta} \right\} \left( \norm{x}^2 + \norm{y}^2 + \norm{z}^2 \right),
\end{aligned}
\end{equation}
where the coefficient in front of the norm of $x$ is positive due to condition \cref{as:assumption_delta}, and which shows that the dissipativity assumption is satisfied. It remains to prove the irreducibility condition \cite[Condition 4.3]{MSH02}. We remark that a similar argument to the one in the example at the end of \cite[Page 199]{MSH02} can be carried out, and therefore \cite[Condition 4.3]{MSH02} is satisfied. Ergodicity then follows from \cite[Theorem 4.4]{MSH02}. The proof for the limit system $(X_t,Z_t)$ is analogous, and therefore we omit the details. Finally, the weak convergence of the measure $\mu^\epl$ to $\mu$ is given by, e.g., \cite[Section 4]{BlP78} or \cite[Section 5.1]{Pav14}.
\end{proof}

We are now ready to state important formulas that link expectations of different quantities with respect to the invariant measure, and which will be employed in the proof of the main theorems.

\begin{lemma} \label{lem:magic_formula_d1}
Let $d=1$, then it holds
\begin{equation}
\frac1\epl G \E^{\mu^\epl} \left[ Y^\epl \otimes f(X^\epl) \right] = - \theta \E^{\mu^\epl} \left[ f(X^\epl) \otimes f(X^\epl) \right].
\end{equation}
\end{lemma}
\begin{proof}
Let $F \colon \R \to \R^\ell$ be a primitive of $f$ defined by
\begin{equation}
F(x) = \int_0^x f(t) \dd t.
\end{equation}
Due to \cref{lem:FP}, multiplying equation \eqref{eq:FP_colored} by $1 \otimes F(x)$, integrating over $\R^{2d+n}$ and then by parts, and noting that
\begin{equation}
\begin{aligned}
\int_{\R^{2d+n}} 1 \otimes F(x) \nabla_x \cdot \left( \theta f(x) \rho^\epl(x,y,z) \right) &= - \theta \int_{\R^{2d+n}} f(x) \otimes f(x) \rho^\epl(x,y,z), \\
\frac1\epl \int_{\R^{2d+n}} 1 \otimes F(x) \nabla_x \cdot \left( G y \rho^\epl(x,y,z) \right) &= - \frac1\epl G \int_{\R^{2d+n}} y \otimes f(x) \rho^\epl(x,y,z),
\end{aligned}
\end{equation}
we get the desired result.
\end{proof}

\begin{lemma} \label{lem:magic_formulas}
The following equalities hold true
\begin{equation}
\begin{aligned}
(i)& \quad \frac1\epl G \E^{\mu^\epl} \left[ Y^\epl \otimes f(Z^\epl) \right] = - \theta \E^{\mu^\epl} \left[ f(X^\epl) \otimes f(Z^\epl) \right] - \frac1\delta \E^{\mu^\epl} \left[ X^\epl \otimes \nabla f(Z^\epl) (X^\epl - Z^\epl) \right], \\
(ii)& \quad \frac1\delta \E^{\mu} \left[ X \otimes \nabla f(Z) (X - Z) \right] = - \theta \E^{\mu} \left[ f(X) \otimes f(Z) \right].
\end{aligned}
\end{equation}
\end{lemma}
\begin{proof}
Let us first consider point $(i)$. Due to \cref{lem:FP}, multiplying equation \eqref{eq:FP_colored} by $x \otimes f(z)$, integrating over $\R^{2d+n}$ and then by parts, and noting that
\begin{equation}
\begin{aligned}
\int_{\R^{2d+n}} x \otimes f(z) \nabla_x \cdot \left( \theta f(x) \rho^\epl(x,y,z) \right) &= - \theta \int_{\R^{2d+n}} f(x) \otimes f(z) \rho^\epl(x,y,z), \\
\frac1\epl \int_{\R^{2d+n}} x \otimes f(z) \nabla_x \cdot \left( G y \rho^\epl(x,y,z) \right) &= - \frac1\epl G \int_{\R^{2d+n}} y \otimes f(z) \rho^\epl(x,y,z), \\
\frac1\delta \int_{\R^{2d+n}} x \otimes f(z) \nabla_z \cdot \left( (x-z) \rho^\epl(x,y,z) \right) &= - \frac1\delta \int_{\R^{2d+n}} x \otimes \nabla f(z) (x - z) \rho^\epl(x,y,z),
\end{aligned}
\end{equation}
we get the desired result. Analogously, multiplying equation \eqref{eq:FP_limit} by $x \otimes f(z)$, integrating over $\R^{2d}$ and then by parts, we obtain point $(ii)$, which concludes the proof.
\end{proof}

\subsection{Infinite time limit of SGDCT}

In this section we show that the SGDCT estimator behaves like the MLE estimator in the infinite time limit. The first result is a technical lemma which will be required later. It is based on \cite[Theorem 4.1]{SiS20} and we restate it in our setting for clarity of exposition. 

\begin{lemma} \label{lem:Poisson_expansion}
Let $\mathcal L^\epl$ and $\mathcal L$ be the generators of the processes given by the SDEs \eqref{eq:system_filtered_eps} and \eqref{eq:system_filtered}, respectively. Let $h^\epl \colon \mathbb T^{2d+n} \to \R^P$ and $h \colon \mathbb T^{2d} \to \R^P$ be functions in $W^{k,\infty}$ with $k,P \in \N$, $k \ge 2$, such that
\begin{equation}
\E^{\mu^\epl} \left[ h(X^\epl, Y^\epl, Z^\epl) \right] = 0 \qquad \text{and} \qquad \E^\mu \left[ h(X, Z) \right] = 0.
\end{equation}
Then, there exist unique solutions $\psi^\epl \colon \T^{2d+n} \to \R^P$ and $\psi \colon \T^{2d} \to \R^P$ in $W^{k,\infty}$ of the Poisson problems
\begin{equation}
\mathcal L^\epl \psi^\epl(x,y,z) = h^\epl(x,y,z), \qquad \mathcal L \psi(x,z) = h(x,z).
\end{equation}
\end{lemma}
\begin{proof}
The result follows from \cite[Theorem 4.1]{SiS20} noting that the hypoellipticity condition is guaranteed by \cref{lem:Lebesgue}.
\end{proof}

\begin{remark} \label{rem:compact_state_space_2}
We recall that we are working under the assumption that the state space is compact, as stated in \cref{rem:compact_state_space}. Replacing the torus with the whole space would involve different technicalities, especially for proving that the functions in \cref{lem:Poisson_expansion} and their derivatives are bounded. We believe that the needed results follow from the results presented in \cite{PaV01,PaV03,PaV05}; but we do not study this extension here.
\end{remark}

Before moving to the limiting properties, we first need to control the moments of the SGDCT estimators. We show in the next lemma that the moments are bounded uniformly in time.

\begin{lemma} \label{lem:bounded_moments}
Let the estimators $\widetilde \theta_t, \widetilde \theta_t^\epl, \widetilde \theta_{\mathrm{exp},t}^\delta, \widetilde \theta_{\mathrm{exp},t}^{\delta,\epl}$ be defined by the SDEs \eqref{eq:SDE_base_ok}, \eqref{eq:SDE_base_ko}, \eqref{eq:SDE_exp_ok}, \eqref{eq:SDE_exp_ok_eps}. Under \cref{as:ergodicity,as:assumption_delta,as:positive_definite}, for all $p\ge1$ there exists a constant $C>0$ independent of time, such that following bounds hold
\begin{equation}
\begin{aligned}
&(i) \quad \E \left[ \norm{\widetilde \theta_t}^p \right] \le C, \qquad &&(ii) \quad \E \left[ \norm{\widetilde \theta_t^\epl}^p \right] \le C, \\
&(iii) \quad \E \left[ \norm{\widetilde \theta_{\mathrm{exp},t}^\delta}^p \right] \le C, \qquad &&(iv) \quad \E \left[ \norm{\widetilde \theta_{\mathrm{exp},t}^{\delta,\epl}}^p \right] \le C.
\end{aligned}
\end{equation}
\end{lemma}
\begin{proof}
We only give full details for $(iv)$, and then we outline the differences with respect to $(ii)$. Let us first show $(iv)$. We consider the SDE for the estimator \eqref{eq:SDE_exp_ok_eps} and we rewrite it as
\begin{equation}
\begin{aligned}
\d \widetilde \theta_{\mathrm{exp},t}^{\delta,\epl} &= - \xi_t \widetilde \theta_{\mathrm{exp},t}^{\delta,\epl} \E^{\mu^\epl} \left[ f(X^\epl) \otimes f(Z^\epl) \right] \dd t + \xi_t \theta f(X_t^\epl) \otimes f(Z_t^\epl) \dd t + \frac1\epl \xi_t G Y_t^\epl \otimes f(Z_t^\epl) \dd t\\
&\quad - \xi_t \widetilde \theta_{\mathrm{exp},t}^{\delta,\epl} \left( f(X_t^\epl) \otimes f(Z_t^\epl) - \E^{\mu^\epl} \left[ f(X^\epl) \otimes f(Z^\epl) \right] \right) \dd t,
\end{aligned}
\end{equation}
then by Itô's lemma we obtain
\begin{equation}
\begin{aligned}
\d \norm{\widetilde \theta_{\mathrm{exp},t}^{\delta,\epl}} &= - \xi_t \widetilde \theta_{\mathrm{exp},t}^{\delta,\epl} \E^{\mu^\epl} \left[ f(X^\epl) \otimes f(Z^\epl) \right] \colon \frac{\widetilde \theta_{\mathrm{exp},t}^{\delta,\epl}}{\norm{\widetilde \theta_{\mathrm{exp},t}^{\delta,\epl}}} \dd t + \xi_t \theta f(X_t^\epl) \otimes f(Z_t^\epl) \colon \frac{\widetilde \theta_{\mathrm{exp},t}^{\delta,\epl}}{\norm{\widetilde \theta_{\mathrm{exp},t}^{\delta,\epl}}} \dd t \\
&\quad - \xi_t \widetilde \theta_{\mathrm{exp},t}^{\delta,\epl} \left( f(X_t^\epl) \otimes f(Z_t^\epl) - \E^{\mu^\epl} \left[ f(X^\epl) \otimes f(Z^\epl) \right] \right) \colon \frac{\widetilde \theta_{\mathrm{exp},t}^{\delta,\epl}}{\norm{\widetilde \theta_{\mathrm{exp},t}^{\delta,\epl}}} \dd t \\
&\quad + \frac1\epl \xi_t G Y_t^\epl \otimes f(Z_t^\epl) \colon \frac{\widetilde \theta_{\mathrm{exp},t}^{\delta,\epl}}{\norm{\widetilde \theta_{\mathrm{exp},t}^{\delta,\epl}}} \dd t, 
\end{aligned}
\end{equation}
where we recall that $\norm{\cdot}$ denotes the Frobenius norm of a matrix. Let us also consider the process $\widetilde \Theta_{\mathrm{exp},t}^{\delta,\epl}$ which satisfies the SDE
\begin{equation}
\begin{aligned}
\d \widetilde \Theta_{\mathrm{exp},t}^{\delta,\epl} &= - \xi_t K \widetilde \Theta_{\mathrm{exp},t}^{\delta,\epl} \dd t + \xi_t \theta f(X_t^\epl) \otimes f(Z_t^\epl) \dd t + \frac1\epl \xi_t G Y_t^\epl \otimes f(Z_t^\epl) \dd t \\
&\quad - \xi_t \widetilde \Theta_{\mathrm{exp},t}^{\delta,\epl} \left( f(X_t^\epl) \otimes f(Z_t^\epl) - \E^{\mu^\epl} \left[ f(X^\epl) \otimes f(Z^\epl) \right] \right) \dd t,
\end{aligned}
\end{equation}
where $K$ is given by \cref{as:positive_definite}, and note that by Itô's lemma we get
\begin{equation}
\begin{aligned}
\d \norm{\widetilde \Theta_{\mathrm{exp},t}^{\delta,\epl}} &= - \xi_t K \norm{\widetilde \Theta_{\mathrm{exp},t}^{\delta,\epl}} \dd t + \xi_t \theta f(X_t^\epl) \otimes f(Z_t^\epl) \colon \frac{\widetilde \Theta_{\mathrm{exp},t}^{\delta,\epl}}{\norm{\widetilde \Theta_{\mathrm{exp},t}^{\delta,\epl}}} \dd t \\
&\quad - \xi_t \widetilde \Theta_{\mathrm{exp},t}^{\delta,\epl} \left( f(X_t^\epl) \otimes f(Z_t^\epl) - \E^{\mu^\epl} \left[ f(X^\epl) \otimes f(Z^\epl) \right] \right) \colon \frac{\widetilde \Theta_{\mathrm{exp},t}^{\delta,\epl}}{\norm{\widetilde \Theta_{\mathrm{exp},t}^{\delta,\epl}}} \dd t \\
&\quad + \frac1\epl \xi_t G Y_t^\epl \otimes f(Z_t^\epl) \colon \frac{\widetilde \Theta_{\mathrm{exp},t}^{\delta,\epl}}{\norm{\widetilde \Theta_{\mathrm{exp},t}^{\delta,\epl}}} \dd t.
\end{aligned}
\end{equation}
Due to \cref{as:positive_definite} and since the drift and diffusion functions in the previous SDEs are continuous if $\norm{\widetilde \theta_{\mathrm{exp},t}^{\delta,\epl}}, \norm{\widetilde \Theta_{\mathrm{exp},t}^{\delta,\epl}} > R$ for any $R>0$, applying the comparison theorem \cite[Theorem 1.1]{IkW77} we deduce that
\begin{equation} \label{eq:comparison}
\Pr \left( \norm{\widetilde \theta_{\mathrm{exp},t}^{\delta,\epl}} \le \norm{\widetilde \Theta_{\mathrm{exp},t}^{\delta,\epl}}, t \ge 0 \right) = 1.
\end{equation}
Moreover, the process $\widetilde \Theta_{\mathrm{exp},t}^{\delta,\epl}$ can be written as the solution of the integral equation
\begin{equation} \label{eq:moments_intergral_equation}
\begin{aligned}
\widetilde \Theta_{\mathrm{exp},t}^{\delta,\epl} &= e^{-K \int_0^t \xi_r \dd r} \widetilde \Theta_{\mathrm{exp},0}^{\delta,\epl} + \int_0^t \xi_s e^{-K \int_s^t \xi_r \dd r} \theta f(X_s^\epl) \otimes f(Z_s^\epl) \dd s \\
&\quad - \int_0^t \xi_s e^{-K \int_s^t \xi_r \dd r} \widetilde \Theta_{\mathrm{exp},s}^{\delta,\epl} \left( f(X_s^\epl) \otimes f(Z_s^\epl) - \E^{\mu^\epl} \left[ f(X^\epl) \otimes f(Z^\epl) \right] \right) \dd s \\
&\quad + \frac1\epl \int_0^t \xi_s e^{-K \int_s^t \xi_r \dd r} G Y_s^\epl \otimes f(Z_s^\epl) \dd s,
\end{aligned}
\end{equation}
and notice that, from the definition of the learning rate in equation \eqref{eq:learning_rate}, we have
\begin{equation}\label{eq:learn_rate_integral}
e^{-K \int_s^t \xi_r \dd r} = \left( \frac{b+s}{b+t} \right)^{aK}.
\end{equation}
Using now \cref{lem:Poisson_expansion} for the function
\begin{equation}
h^\epl(x,y,z) = f(x) \otimes f(z) - \E^{\mu^\epl} \left[ f(X^\epl) \otimes f(Z^\epl) \right],
\end{equation}
we deduce that the solution to the Poisson equation $-\mathcal{L}^{\epl} \psi^\epl = h^\epl(x,y,z) $, with $\psi^\epl = \psi^\epl(x,y,z)$, is bounded, together with its derivatives. Next, we apply Itô's formula to the function
\begin{equation}
\phi^\epl(t,x,y,z,\vartheta) = (b+t)^{aK-1} \vartheta \psi^\epl(x,y,z),
\end{equation}
to deduce 
\begin{equation}
\begin{aligned}
&(b+t)^{aK-1} \widetilde \Theta_{\mathrm{exp},t}^{\delta,\epl} \psi^\epl(X_t^\epl, Y_t^\epl, Z_t^\epl) - b^{aK-1} \widetilde \Theta_{\mathrm{exp},0}^{\delta,\epl} \psi^\epl(X_0^\epl, Y_0^\epl, Z_0^\epl) = \\
&(aK-1) \int_0^t (b+s)^{aK-2} \Theta_{\mathrm{exp},s}^{\delta,\epl} \psi^\epl(X_s^\epl, Y_s^\epl, Z_s^\epl) \dd s + \int_0^t (b+s)^{aK-1} \widetilde \Theta_{\mathrm{exp},s}^{\delta,\epl} h^\epl(X_s^\epl, Y_s^\epl, Z_s^\epl) \dd s \\
&+ \frac1\epl \int_0^t (b+s)^{aK-1} \widetilde \Theta_{\mathrm{exp},s}^{\delta,\epl} \nabla_y \psi^\epl(X_s^\epl, Y_s^\epl, Z_s^\epl) \cdot \sigma \d W_s + \int_0^t (b+s)^{aK-1} \d \widetilde \Theta_{\mathrm{exp},s}^{\delta,\epl} \psi^\epl(X_s^\epl, Y_s^\epl, Z_s^\epl). \\
\end{aligned}
\end{equation}
This, together with \eqref{eq:moments_intergral_equation}, implies
\begin{equation}
\begin{aligned}
\widetilde \Theta_{\mathrm{exp},t}^{\delta,\epl} &= \widetilde \Theta_{\mathrm{exp},0}^{\delta,\epl} \left( \frac{b}{b+t} \right)^{aK} - \frac{a}{b+t} \widetilde \Theta_{\mathrm{exp},t}^{\delta,\epl} \psi^\epl(X_t^\epl, Y_t^\epl, Z_t^\epl) + \frac{a b^{aK-1}}{(b+t)^{aK}} \widetilde \Theta_{\mathrm{exp},0}^{\delta,\epl} \psi^\epl(X_0^\epl, Y_0^\epl, Z_0^\epl) \\
&\quad + a \int_0^t \frac{(b+s)^{aK-1}}{(b+t)^{aK}} \theta f(X_s^\epl) \otimes f(Z_s^\epl) \dd s + \frac{a}\epl \int_0^t \frac{(b+s)^{aK-1}}{(b+t)^{aK}} G Y_s^\epl \otimes f(Z_s^\epl) \dd s \\
&\quad + a(aK-1) \int_0^t \frac{(b+s)^{aK-2}}{(b+t)^{aK}} \Theta_{\mathrm{exp},s}^{\delta,\epl} \psi^\epl(X_s^\epl, Y_s^\epl, Z_s^\epl) \dd s \\
&\quad + \frac{a}\epl \int_0^t \frac{(b+s)^{aK-1}}{(b+t)^{aK}} \widetilde \Theta_{\mathrm{exp},s}^{\delta,\epl} \nabla_y \psi^\epl(X_s^\epl, Y_s^\epl, Z_s^\epl) \cdot \sigma \d W_s \\
&\quad - a^2 \int_0^t \frac{(b+s)^{aK-2}}{(b+t)^{aK}} \widetilde \Theta_{\mathrm{exp},s}^{\delta,\epl} \left( KI + f(X_s^\epl) \otimes f(Z_s^\epl) \right) \psi^\epl(X_s^\epl, Y_s^\epl, Z_s^\epl) \dd s \\
&\quad + a^2 \int_0^t \frac{(b+s)^{aK-2}}{(b+t)^{aK}} \widetilde \Theta_{\mathrm{exp},s}^{\delta,\epl} \E^{\mu^\epl} \left[ f(X^\epl) \otimes f(Z^\epl) \right] \psi^\epl(X_s^\epl, Y_s^\epl, Z_s^\epl) \dd s \\
&\quad + a^2 \int_0^t \frac{(b+s)^{aK-2}}{(b+t)^{aK}} \left( \theta f(X_s^\epl) \otimes f(Z_s^\epl) + \frac1\epl G Y_s^\epl \otimes f(Z_s^\epl) \right) \psi^\epl(X_s^\epl, Y_s^\epl, Z_s^\epl) \dd s.
\end{aligned}
\end{equation}
Due to the boundedness of the function $\psi^\epl$ and the processes $X_t^\epl,Y_t^\epl,Z_t^\epl$, inequality \cite[Theorem 1]{Zak67}, and Jensen's inequality we obtain
\begin{equation}
\begin{aligned}
\E \left[ \norm{\widetilde \Theta_{\mathrm{exp},t}^{\delta,\epl}}^{2q} \right] &\le C + \frac{C}{b+t} \E \left[ \norm{\widetilde \Theta_{\mathrm{exp},t}^{\delta,\epl}}^{2q} \right] + C \int_0^t \frac{(b+s)^{aK-2}}{(b+t)^{aK}} \E \left[ \norm{\widetilde \Theta_{\mathrm{exp},s}^{\delta,\epl}}^{2q} \right] \dd s \\
&\quad + C \int_0^t \frac{(b+s)^{2aK-2}}{(b+t)^{2aK}} \E \left[ \norm{\widetilde \Theta_{\mathrm{exp},s}^{\delta,\epl}}^{2q} \right] \dd s,
\end{aligned}
\end{equation}
which for $t$ sufficiently large implies
\begin{equation}
(b+t)^{aK} \E \left[ \norm{\widetilde \Theta_{\mathrm{exp},t}^{\delta,\epl}}^{2q} \right] \le C(b+t)^{aK} + C \int_0^t \frac1{(b+s)^2} (b+s)^{aK} \E \left[ \norm{\widetilde \Theta_{\mathrm{exp},s}^{\delta,\epl}}^{2q} \right] \dd s,
\end{equation}
and by Grönwall's inequality we get
\begin{equation}
\E \left[ \norm{\widetilde \Theta_{\mathrm{exp},t}^{\delta,\epl}}^{2q} \right] \le C e^{\int_0^t \frac1{(b+s)^2} \dd s} \le C e^{1/b}.
\end{equation}
We remark that if $t$ is small, i.e., in a compact interval $[0,T]$, then we could directly apply Grönwall's inequality from equation \eqref{eq:moments_intergral_equation}. In fact, we would have
\begin{equation}
\E \left[ \norm{\widetilde \Theta_{\mathrm{exp},t}^{\delta,\epl}}^{2q} \right] \le C + C \int_0^t \frac{(b+s)^{aK-1}}{(b+t)^{aK}} \E \left[ \norm{\widetilde \Theta_{\mathrm{exp},s}^{\delta,\epl}}^{2q} \right] \dd s,
\end{equation}
which similarly as above yields to
\begin{equation}
\E \left[ \norm{\widetilde \Theta_{\mathrm{exp},t}^{\delta,\epl}}^{2q} \right] \le C e^{\int_0^t \frac1{b+s} \dd s} \le \frac{C}{b} (T + b).
\end{equation}
This shows that all the even moments of order $p = 2q$ are bounded. If $p$ is odd, then by Hölder's inequality we have
\begin{equation}
\E \left[ \norm{\widetilde \Theta_{\mathrm{exp},t}^{\delta,\epl}}^{p} \right] \le \left( \E \left[ \norm{\widetilde \Theta_{\mathrm{exp},t}^{\delta,\epl}}^{p+1} \right] \right)^{\frac{p}{p+1}},
\end{equation}
and therefore any odd moment can be bounded by the consecutive even moment. Finally, the desired result follows from \eqref{eq:comparison}. 

Let us now consider $(ii)$. The main difference in comparison to $(iv)$ is that we do not have the colored noise process $Y_t^\epl$; this simplifies the analysis. On the other hand, we have the diffusion term in the SDE for the estimator \eqref{eq:SDE_exp_ok}, which we rewrite as
\begin{equation}
\begin{aligned}
\d \widetilde \theta_{\mathrm{exp},t}^\delta &= - \xi_t \widetilde \theta_{\mathrm{exp},t}^\delta \E^\mu \left[ f(X) \otimes f(Z) \right] \dd t + \xi_t \theta f(X_t) \otimes f(Z_t) \dd t \\
&\quad - \xi_t \widetilde \theta_{\mathrm{exp},t}^\delta \left( f(X_t) \otimes f(Z_t) - \E^\mu \left[ f(X) \otimes f(Z) \right] \right) \dd t \\
&\quad + \xi_t \sqrt{2D^S} \dd W_t \otimes f(Z_t).
\end{aligned}
\end{equation}
Similarly as above we work with the process $\widetilde \Theta_{\mathrm{exp},t}^\delta$ defined by the SDE
\begin{equation}
\begin{aligned}
\d \widetilde \Theta_{\mathrm{exp},t}^\delta &= - \xi_t K \widetilde \Theta_{\mathrm{exp},t}^\delta \dd t + \xi_t \theta f(X_t) \otimes f(Z_t) \dd t \\
&\quad - \xi_t \widetilde \Theta_{\mathrm{exp},t}^\delta \left( f(X_t) \otimes f(Z_t) - \E^\mu \left[ f(X) \otimes f(Z) \right] \right) \dd t \\
&\quad + \xi_t \sqrt{2D^S} \dd W_t \otimes f(Z_t),
\end{aligned}
\end{equation}
which can be written as
\begin{equation}
\begin{aligned}
\widetilde \Theta_{\mathrm{exp},t}^\delta &= e^{-K \int_0^t \xi_r \dd r} \widetilde \Theta_{\mathrm{exp},0}^\delta + \int_0^t \xi_s e^{-K \int_s^t \xi_r \dd r} \theta f(X_s) \otimes f(Z_s) \dd s \\
&\quad - \int_0^t \xi_s e^{-K \int_s^t \xi_r \dd r} \widetilde \Theta_{\mathrm{exp},s}^\delta \left( f(X_s) \otimes f(Z_s) - \E^\mu \left[ f(X) \otimes f(Z) \right] \right) \dd s \\
&\quad + \int_0^t \xi_s e^{-K \int_s^t \xi_r \dd r} \sqrt{2D^S} \d W_s \otimes f(Z_s),
\end{aligned}
\end{equation}
and whose moments give an upper bound to the moments of $\widetilde \theta_{\mathrm{exp},t}^\delta$ due to the comparison theorem~\cite[Theorem 1.1]{IkW77}. Proceeding analogously to $(iv)$ we obtain for a function $\psi= \psi(x,z)$ which is bounded along with its derivatives
\begin{equation}
\begin{aligned}
\widetilde \Theta_{\mathrm{exp},t}^\delta &= \widetilde \Theta_{\mathrm{exp},0}^\delta \left( \frac{b}{b+t} \right)^{aK} - \frac{a}{b+t} \widetilde \Theta_{\mathrm{exp},t}^\delta \psi(X_t, Z_t) + \frac{a b^{aK-1}}{(b+t)^{aK}} \widetilde \Theta_{\mathrm{exp},0}^\delta \psi(X_0, Z_0) \\
&\quad + a \int_0^t \frac{(b+s)^{aK-1}}{(b+t)^{aK}} \theta f(X_s) \otimes f(Z_s) \dd s + a \int_0^t \frac{(b+s)^{aK-1}}{(b+t)^{aK}} \sqrt{2D^S} \d W_s \otimes f(Z_s^\epl) \\
&\quad + a(aK-1) \int_0^t \frac{(b+s)^{aK-2}}{(b+t)^{aK}} \Theta_{\mathrm{exp},s}^\delta \psi(X_s, Z_s) \dd s \\
&\quad + a \int_0^t \frac{(b+s)^{aK-1}}{(b+t)^{aK}} \widetilde \Theta_{\mathrm{exp},s}^\delta \nabla_x \psi(X_s, Z_s) \cdot \sqrt{2D^S} \d W_s \\
&\quad - a^2 \int_0^t \frac{(b+s)^{aK-2}}{(b+t)^{aK}} \widetilde \Theta_{\mathrm{exp},s}^\delta \left( KI + f(X_s) \otimes f(Z_s) - \E^\mu \left[ f(X) \otimes f(Z) \right] \right) \psi(X_s, Z_s) \dd s \\
&\quad + a^2 \int_0^t \frac{(b+s)^{aK-2}}{(b+t)^{aK}} \theta f(X_s^\epl) \otimes f(Z_s^\epl) \psi(X_s, Z_s) \dd s \\
&\quad + a^2 \int_0^t \frac{(b+s)^{aK-2}}{(b+t)^{aK}} \sqrt{2D^S} \d W_s \otimes f(Z_s^\epl) \psi(X_s, Z_s),
\end{aligned}
\end{equation}
and the desired result then follows from Grönwall's inequality. Finally, the proofs of $(i)$ and $(iii)$ are almost the same as the proofs of $(ii)$ and $(iv)$, respectively. We omit the details.
\end{proof}

The next two propositions are crucial for proving the (un)biasedness of the SGDCT estimators for processes driven by colored noise. Since the proofs are similar, we only provide details in the case with filtered data.

\begin{proposition} \label{pro:limit_SGDCT_t_NOfilter}
Let $X_t^\epl, Y_t^\epl$ be solutions of system \eqref{eq:system_multiD_additive}. Under \cref{as:positive_definite}$(i)$, it holds
\begin{equation}
\lim_{t\to\infty} \widetilde \theta_t^\epl = \theta + \frac1\epl G \E^{\mu^\epl} \left[ Y^\epl \otimes f(X^\epl) \right] \E^{\mu^\epl} \left[ f(X^\epl) \otimes f(X^\epl) \right]^{-1}, \qquad \text{in } L^2.
\end{equation}
\end{proposition}

\begin{proposition} \label{pro:limit_SGDCT_t_YESfilter}
Let $X_t^\epl, Y_t^\epl, Z_t^\epl$ be solutions of system \eqref{eq:system_filtered_eps}. Under \cref{as:positive_definite}$(ii)$, it holds
\begin{equation}
\lim_{t\to\infty} \widetilde \theta_{\mathrm{exp},t}^{\delta,\epl} = \theta + \frac1\epl G \E^{\mu^\epl} \left[ Y^\epl \otimes f(Z^\epl) \right] \E^{\mu^\epl} \left[ f(X^\epl) \otimes f(Z^\epl) \right]^{-1}, \qquad \text{in } L^2.
\end{equation}
\end{proposition}
\begin{proof}
Let us consider the SDE for the estimator \eqref{eq:SDE_exp_ok_eps} and rewrite it as
\begin{equation}
\begin{aligned}
\d & \left( \widetilde \theta_{\mathrm{exp},t}^{\delta,\epl} - \theta - \frac1\epl G \E^{\mu^\epl} \left[ Y^\epl \otimes f(Z^\epl) \right] \E^{\mu^\epl} \left[ f(X^\epl) \otimes f(Z^\epl) \right]^{-1} \right) \\
&= - \xi_t \left( \widetilde \theta_{\mathrm{exp},t}^{\delta,\epl} - \theta - \frac1\epl G \E^{\mu^\epl} \left[ Y^\epl \otimes f(Z^\epl) \right] \E^{\mu^\epl} \left[ f(X^\epl) \otimes f(Z^\epl) \right]^{-1} \right) \E^{\mu^\epl} \left[ f(X^\epl) \otimes f(Z^\epl) \right] \dd t \\
&\quad - \xi_t \left( \widetilde \theta_{\mathrm{exp},t}^{\delta,\epl} - \theta \right) \left( f(X_t^\epl) \otimes f(Z_t^\epl) - \E^{\mu^\epl} \left[ f(X^\epl) \otimes f(Z^\epl) \right] \right) \dd t \\
&\quad + \frac1\epl \xi_t G \left( Y_t^\epl \otimes f(Z_t^\epl) - \E^{\mu^\epl} \left[ Y^\epl \otimes f(Z^\epl)  \right] \right) \dd t.
\end{aligned}
\end{equation}
Letting
\begin{equation}
\Delta^\epl_t = \widetilde \theta_{\mathrm{exp},t}^{\delta,\epl} - \theta - \frac1\epl G \E^{\mu^\epl} \left[ Y^\epl \otimes f(Z^\epl) \right] \E^{\mu^\epl} \left[ f(X^\epl) \otimes f(Z^\epl) \right]^{-1} \quad \text{and} \quad \Gamma^\epl_t = \norm{\Delta_t^\epl}^2,
\end{equation}
by Itô's lemma we have
\begin{equation}
\begin{aligned}
\d \Gamma_t^\epl &= - 2 \xi_t \Delta_t^\epl \E^{\mu^\epl} \left[ f(X^\epl) \otimes f(Z^\epl) \right] : \Delta_t^\epl \dd t \\
&\quad - 2 \xi_t \left( \widetilde \theta_{\mathrm{exp},t}^{\delta,\epl} - \theta \right) \left( f(X_t^\epl) \otimes f(Z_t^\epl) - \E^{\mu^\epl} \left[ f(X^\epl) \otimes f(Z^\epl) \right] \right) : \Delta_t^\epl \dd t \\
&\quad + \frac2\epl \xi_t G \left( Y_t^\epl \otimes f(Z_t^\epl) - \E^{\mu^\epl} \left[ Y^\epl \otimes f(Z^\epl)  \right] \right) : \Delta_t^\epl \dd t,
\end{aligned}
\end{equation}
where $\colon$ stands for the Frobenius inner product, and which due to \cref{as:positive_definite,rem:positive_definite} implies
\begin{equation}
\begin{aligned}
\d \Gamma_t^\epl &\le - 2 K \xi_t \Gamma_t^\epl \dd t \\
&\quad - 2 \xi_t \left( \widetilde \theta_{\mathrm{exp},t}^{\delta,\epl} - \theta \right) \left( f(X_t^\epl) \otimes f(Z_t^\epl) - \E^{\mu^\epl} \left[ f(X^\epl) \otimes f(Z^\epl) \right] \right) : \Delta_t^\epl \dd t \\
&\quad + \frac2\epl \xi_t G \left( Y_t^\epl \otimes f(Z_t^\epl) - \E^{\mu^\epl} \left[ Y^\epl \otimes f(Z^\epl)  \right] \right) : \Delta_t^\epl \dd t.
\end{aligned}
\end{equation}
By the comparison principle we obtain
\begin{equation} \label{eq:decomposition_Gamma}
\begin{aligned}
\Gamma_t^\epl &\le \norm{\Delta_0^\epl}^2 e^{-2K \int_0^t \xi_r \dd r} \\
&\quad -2 \int_0^t \xi_s e^{-2K \int_s^t \xi_r \dd r} \left( \widetilde \theta_{\mathrm{exp},s}^{\delta,\epl} - \theta \right) \left( f(X_s^\epl) \otimes f(Z_s^\epl) - \E^{\mu^\epl} \left[ f(X^\epl) \otimes f(Z^\epl) \right] \right) : \Delta_s^\epl \dd s \\
&\quad + \frac2\epl \int_0^t \xi_s e^{-2K \int_s^t \xi_r \dd r} G \left( Y_s^\epl \otimes f(Z_s^\epl) - \E^{\mu^\epl} \left[ Y^\epl \otimes f(Z^\epl)  \right] \right) : \Delta_s^\epl \dd s \\
&\eqdef I_t^1 + I_t^2 + I_t^3,
\end{aligned}
\end{equation}
and we now study the three terms in the right-hand side separately. First, we use equation \eqref{eq:learn_rate_integral} to deduce that
\begin{equation} \label{eq:lim_It1}
\lim_{t\to\infty} I_t^1 = \lim_{t\to\infty} \norm{\Delta_0^\epl}^2 \left( \frac{b}{b+t} \right)^{2aK} = 0.
\end{equation}
Then, applying \cref{lem:Poisson_expansion} for the right hand side of the Poisson equation being the function
\begin{equation}
h_1^\epl(x,y,z) = f(x) \otimes f(z) - \E^{\mu^\epl} \left[ f(X^\epl) \otimes f(Z^\epl) \right],
\end{equation}
we deduce that the solution to the Poisson equation $\psi_1^\epl = \psi_1^\epl(x,y,z)$ is bounded, together with all its derivatives. Letting $\zeta_t = (b+t)^{2aK - 1}$ and applying Itô's lemma to the function
\begin{equation}
\phi_1^\epl(t,x,y,z,\vartheta) = \zeta_t (\vartheta - \theta) \psi_1^\epl(x,y,z) : \left( \vartheta - \theta - \frac1\epl G \E^{\mu^\epl} \left[ Y^\epl \otimes f(Z^\epl) \right] \E^{\mu^\epl} \left[ f(X^\epl) \otimes f(Z^\epl) \right]^{-1} \right),
\end{equation}
we obtain
\begin{equation}
\begin{aligned}
&\zeta_t \left( \widetilde \theta_{\mathrm{exp},t}^{\delta,\epl} - \theta \right) \psi_1^\epl(X_t^\epl, Y_t^\epl, Z_t^\epl) \colon \Delta_t^\epl - \zeta_0 \left( \widetilde \theta_{\mathrm{exp},0}^{\delta,\epl} - \theta \right) \psi_1^\epl(X_0^\epl, Y_0^\epl, Z_0^\epl) \colon \Delta_0^\epl = \\
&\hspace{2cm}\quad \int_0^t \zeta_s' \left( \widetilde \theta_{\mathrm{exp},s}^{\delta,\epl} - \theta \right) \psi_1^\epl(X_s^\epl, Y_s^\epl, Z_s^\epl) \colon \Delta_s^\epl \dd s \\
&\hspace{2cm} + \int_0^t \zeta_t \left( \widetilde \theta_{\mathrm{exp},s}^{\delta,\epl} - \theta \right) \mathcal L^\epl \psi_1^\epl(X_s^\epl, Y_s^\epl, Z_s^\epl) \colon \Delta_s^\epl \dd s \\
&\hspace{2cm} + \frac1\epl \int_0^t \zeta_s \left( \widetilde \theta_{\mathrm{exp},s}^{\delta,\epl} - \theta \right) \nabla_y \psi_1^\epl(X_s^\epl, Y_s^\epl, Z_s^\epl) \colon \Delta_s^\epl \cdot \sigma \dd W_s \\
&\hspace{2cm} + \int_0^t \zeta_s \left[ \left( \widetilde \theta_{\mathrm{exp},s}^{\delta,\epl} - \theta \right) \psi_1^\epl(X_s^\epl, Y_s^\epl, Z_s^\epl) + \Delta_s^\epl \psi_1^\epl(X_s^\epl, Y_s^\epl, Z_s^\epl)^\top \right] \colon \dd \widetilde \theta_{\mathrm{exp},s}^{\delta,\epl},
\end{aligned}
\end{equation}
which implies
\begin{equation} \label{eq:twice}
\begin{aligned}
I_t^2 &=\quad \frac{2a}{b+t} \left( \widetilde \theta_{\mathrm{exp},t}^{\delta,\epl} - \theta \right) \psi_1^\epl(X_t^\epl, Y_t^\epl, Z_t^\epl) \colon \Delta_t^\epl - \frac{2ab^{2aK-1}}{(b+t)^{2aK}} \left( \widetilde \theta_{\mathrm{exp},0}^{\delta,\epl} - \theta \right) \psi_1^\epl(X_0^\epl, Y_0^\epl, Z_0^\epl) \colon \Delta_0^\epl \\
&\quad - \frac{2a(2aK-1)}{(b+t)^{2aK}} \int_0^t (b+s)^{2aK-2} \left( \widetilde \theta_{\mathrm{exp},s}^{\delta,\epl} - \theta \right) \psi_1^\epl(X_s^\epl, Y_s^\epl, Z_s^\epl) \colon \Delta_s^\epl \dd s \\
&\quad - \frac{2a}{\epl(b+t)^{2aK}} \int_0^t (b+s)^{2aK-1} \left( \widetilde \theta_{\mathrm{exp},s}^{\delta,\epl} - \theta \right) \nabla_y \psi_1^\epl(X_s^\epl, Y_s^\epl, Z_s^\epl) \colon \Delta_s^\epl \cdot \sigma \dd W_s \\
&\quad - \frac{2a^2}{(b+t)^{2aK}} \int_0^t (b+s)^{2aK-2} \left[ \left( \widetilde \theta_{\mathrm{exp},s}^{\delta,\epl} - \theta \right) \psi_1^\epl(X_s^\epl, Y_s^\epl, Z_s^\epl) + \Delta_s^\epl \psi_1^\epl(X_s^\epl, Y_s^\epl, Z_s^\epl)^\top \right] \colon \\
&\hspace{5.25cm} \left[ \left( -\left(\widetilde \theta_{\mathrm{exp},s}^{\delta,\epl} - \theta \right) f(X_s^\epl) + \frac1\epl G Y_s^\epl \right) \otimes f(Z_s^\epl) \right] \dd s.
\end{aligned}
\end{equation}
By the boundedness of the function $\psi_1^\epl$ and the estimator $\widetilde \theta_{\mathrm{exp},s}^{\delta,\epl}$ by \cref{lem:Poisson_expansion} and \cref{lem:bounded_moments}, respectively, and due to the Itô isometry, we obtain
\begin{equation} \label{eq:lim_It2}
\E \left[ I_t^2 \right] \le C \left( \frac1{b+t} + \frac1{(b+t)^{2aK}} \right).
\end{equation}
Repeating a similar argument for $I_t^3$, but now applying \cref{lem:Poisson_expansion} to the function 
\begin{equation}
h_2^\epl(x,y,z) = y \otimes f(z) - \E^{\mu^\epl} \left[ Y^\epl \otimes f(Z^\epl)  \right],
\end{equation}
and Itô's lemma to the function
\begin{equation}
\phi_2^\epl(t,x,y,z,\vartheta) = \zeta_t \psi_2^\epl(x,y,z) \colon \left( \vartheta - \theta - \frac1\epl G \E^{\mu^\epl} \left[ Y^\epl \otimes f(Z^\epl) \right] \E^{\mu^\epl} \left[ f(X^\epl) \otimes f(Z^\epl) \right]^{-1} \right),
\end{equation}
where $\psi_2^\epl$ is the solution of the Poisson problem, we also get that
\begin{equation} \label{eq:lim_It3}
\E \left[ I_t^3 \right] \le C \left( \frac1{b+t} + \frac1{(b+t)^{2aK}} \right).
\end{equation}
Finally, decomposition \eqref{eq:decomposition_Gamma} together with equations \eqref{eq:lim_It1}, \eqref{eq:lim_It2}, \eqref{eq:lim_It3} yield that $\Delta_t^\epl \to 0$ in $L^2$, which in turn gives desired results.
\end{proof}

\subsection{Proof of the main results}

We are now ready to prove the theoretical results presented in \cref{sec:estimation_additive}.

\begin{proof}[Proof of \cref{thm:MLE_ok}]
Replacing equation \eqref{eq:limit_multiD_additive} in the definition of the estimator \eqref{eq:MLE_ok} we have
\begin{equation}
\widehat \theta(X,T) = \theta + \sqrt{2D^S} \left( \int_0^T \d W_t \otimes f(X_t) \right) \left( \int_0^T f(X_t) \otimes f(X_t) \dd t \right)^{-1}.
\end{equation}
Then, by the ergodic theorem for additive functionals of Markov processes and the martingale central limit theorem \cite[Theorem 3.33]{PaS08}, \cite[Thm 2.1]{KLO12} we obtain
\begin{equation}
\lim_{T\to\infty} \frac1T \int_0^T f(X_t) \otimes f(X_t) \dd t = \E^{\mu} \left[ f(X) \otimes f(X) \right], \qquad a.s.,
\end{equation}
and
\begin{equation}
\lim_{T\to\infty} \frac1T \int_0^T \d W_t \otimes f(X_t) = 0, \qquad a.s., 
\end{equation}
which imply the desired result.
\end{proof}

\begin{proof}[Proof of \cref{pro:MLE_ko}]
Replacing the first equation from \eqref{eq:system_multiD_additive} in the definition of the estimator \eqref{eq:MLE_ko} we have
\begin{equation}
\widehat \theta(X^\epl,T) = \theta + \frac1\epl G \left( \int_0^T Y_t^\epl \otimes f(X^\epl_t) \dd t \right) \left( \int_0^T f(X^\epl_t) \otimes f(X^\epl_t) \dd t \right)^{-1}, \qquad a.s.,
\end{equation}
and by the ergodic theorem we obtain
\begin{equation}
\lim_{T\to\infty} \widehat \theta(X^\epl,T) = \theta + \frac1\epl G \E^{\mu^\epl} \left[ Y^\epl \otimes f(X^\epl) \right] \E^{\mu^\epl} \left[ f(X^\epl) \otimes f(X^\epl) \right]^{-1}, \qquad a.s.
\end{equation}
Finally, employing \cref{lem:magic_formula_d1} we get the desired result.
\end{proof}

\begin{proof}[Proof of \cref{thm:exp_ok}]
Replacing equation \eqref{eq:limit_multiD_additive} in the definition of the estimator \eqref{eq:exp_ok} we have
\begin{equation}
\widehat \theta_{\mathrm{exp}}^\delta(X,T) = \theta + \sqrt{2D^S} \left( \int_0^T \d W_t \otimes f(Z_t) \right) \left( \int_0^T f(X_t) \otimes f(Z_t) \dd t \right)^{-1}.
\end{equation}
Then, by the ergodic theorem and the martingale central limit theorem \cite[Theorem 3.33]{PaS08} we obtain
\begin{equation}
\lim_{T\to\infty} \frac1T \int_0^T f(X_t) \otimes f(Z_t) \dd t = \E^{\mu} \left[ f(X) \otimes f(Z) \right], \qquad a.s.,
\end{equation}
and
\begin{equation}
\lim_{T\to\infty} \frac1T \int_0^T \d W_t \otimes f(Z_t) = 0, \qquad a.s., 
\end{equation}
which imply the desired result.
\end{proof}

\begin{proof}[Proof of \cref{thm:exp_ok_eps}]
Replacing the first equation from \eqref{eq:system_multiD_additive} in the definition of the estimator \eqref{eq:exp_ok_eps} we have
\begin{equation}
\widehat \theta_{\mathrm{exp}}^\delta(X^\epl,T) = \theta + \frac1\epl G \left( \int_0^T Y_t^\epl \otimes f(Z^\epl_t) \dd t \right) \left( \int_0^T f(X^\epl_t) \otimes f(Z^\epl_t) \dd t \right)^{-1},
\end{equation}
and by the ergodic theorem we obtain
\begin{equation}
\lim_{T\to\infty} \widehat \theta_{\mathrm{exp}}^\delta(X^\epl,T) = \theta + \frac1\epl G \E^{\mu^\epl} \left[ Y^\epl \otimes f(Z^\epl) \right] \E^{\mu^\epl} \left[ f(X^\epl) \otimes f(Z^\epl) \right]^{-1}, \qquad a.s.
\end{equation}
Employing formula $(i)$ in \cref{lem:magic_formulas} we have
\begin{equation}
\lim_{T\to\infty} \widehat \theta_{\mathrm{exp}}^\delta(X^\epl,T) = - \frac1\delta \E^{\mu^\epl} \left[ X^\epl \otimes \nabla f(Z^\epl) (X^\epl - Z^\epl) \right] \E^{\mu^\epl} \left[ f(X^\epl) \otimes f(Z^\epl) \right]^{-1}, \qquad a.s.
\end{equation}
which due to the weak convergence of the joint process $(X_t^\epl,Z_t^\epl)$ to $(X_t,Z_t)$ gives
\begin{equation}
\lim_{\epl\to0} \lim_{T\to\infty} \widehat \theta_{\mathrm{exp}}^\delta(X^\epl,T) = - \frac1\delta \E^\mu \left[ X \otimes \nabla f(Z) (X - Z) \right] \E^\mu \left[ f(X) \otimes f(Z) \right]^{-1}, \qquad a.s.
\end{equation}
Finally, formula $(ii)$ in \cref{lem:magic_formulas} yields the desired result.
\end{proof}

\begin{proof}[Proof of \cref{thm:SGDCT_ok}]
The proof is analogous to the proof of \cref{thm:SGDCT_exp_ok}, so we omit the details here. Moreover, it can also be seen as a particular case of \cite[Theorem 1]{SiS20}.
\end{proof}

\begin{proof}[Proof of \cref{pro:SGDCT_ko}]
The desired result is obtained upon combining \cref{pro:limit_SGDCT_t_NOfilter} and \cref{lem:magic_formula_d1}.
\end{proof}

\begin{proof}[Proof of \cref{thm:SGDCT_exp_ok}]
Let us consider the SDE for the estimator  \eqref{eq:SDE_exp_ok} and rewrite it as
\begin{equation}
\begin{aligned}
\d \left( \widehat \theta^\delta_{\mathrm{exp},t} - \theta \right) &= - \xi_t \left( \widehat \theta^\delta_{\mathrm{exp},t} - \theta \right) \E \left[ f(X) \otimes f(Z) \right] \dd t \\
&\quad - \xi_t \left( \widehat \theta^\delta_{\mathrm{exp},t} - \theta \right) \left( f(X_t) \otimes f(Z_t) - \E \left[ f(X) \otimes f(Z) \right] \right) \dd t \\
&\quad + \xi_t \sqrt{2D^S} \dd W_t \otimes f(Z_t).
\end{aligned}
\end{equation}
Letting
\begin{equation}
\Delta_t = \widetilde \theta_{\mathrm{exp},t}^\delta - \theta \qquad \text{and} \qquad \Gamma_t = \norm{\Delta_t}^2,
\end{equation}
by Itô's lemma we have
\begin{equation}
\begin{aligned}
\d \Gamma_t &= - 2 \xi_t \Delta_t \E^\mu \left[ f(X) \otimes f(Z) \right] : \Delta_t \dd t \\
&\quad - 2 \xi_t \Delta_t \left( f(X_t) \otimes f(Z_t) - \E^\mu \left[ f(X) \otimes f(Z) \right] \right) : \Delta_t \dd t \\
&\quad + \xi_t^2 \norm{\sqrt{2D^S}}^2 \norm{f(Z_t)}^2 \dd t \\
&\quad + 2 \xi_t \sqrt{2D^S} \dd W_t \otimes f(Z_t) \colon \Delta_t,
\end{aligned}
\end{equation}
which due to \cref{as:positive_definite,rem:positive_definite} implies
\begin{equation}
\begin{aligned}
\d \Gamma_t &\le - 2 K\xi_t \Gamma_t \dd t \\
&\quad - 2 \xi_t \Delta_t \left( f(X_t) \otimes f(Z_t) - \E^\mu \left[ f(X) \otimes f(Z) \right] \right) : \Delta_t \dd t \\
&\quad + \xi_t^2 \norm{\sqrt{2D^S}}^2 \norm{f(Z_t)}^2 \dd t \\
&\quad + 2 \xi_t \sqrt{2D^S} \dd W_t \otimes f(Z_t) \colon \Delta_t.
\end{aligned}
\end{equation}
By the comparison principle we obtain
\begin{equation}
\begin{aligned}
\Gamma_t &\le \norm{\Delta_0}^2 e^{-2K \int_0^t \xi_r \dd r} \\
&\quad -2 \int_0^t \xi_s e^{-2K \int_s^t \xi_r \dd r} \Delta_s \left( f(X_s) \otimes f(Z_s) - \E^\mu \left[ f(X) \otimes f(Z) \right] \right) : \Delta_s \dd s \\
&\quad + \norm{\sqrt{2D^S}}^2 \int_0^t \xi_s^2 e^{-2K \int_s^t \xi_r \dd r} \norm{f(Z_s)}^2 \dd s \\
&\quad + 2 \int_0^t \xi_s e^{-2K \int_s^t \xi_r \dd r} \sqrt{2D^S} \dd W_s \otimes f(Z_s) \colon \Delta_s \\
&\eqdef I_t^1 + I_t^2 + I_t^3 + I_t^4,
\end{aligned}
\end{equation}
and we now study the four terms in the right-hand side separately. The first two terms $I_t^1$ and $I_t^2$ appear also in the proof of \cref{pro:limit_SGDCT_t_YESfilter} with colored noise, where we show that their expectations vanish as $t\to\infty$. Regarding $I_t^2$, the main difference, which does not affect the final result, is in equation \eqref{eq:twice}, where in this case we would have
\begin{equation}
\begin{aligned}
I_t^2 &=\quad \frac{2a}{b+t} \left( \widetilde \theta_{\mathrm{exp},t}^\delta - \theta \right) \psi_1(X_t, Z_t) \colon \Delta_t - \frac{2ab^{2aK-1}}{(b+t)^{2aK}} \left( \widetilde \theta_{\mathrm{exp},0}^\delta - \theta \right) \psi_1(X_0, Z_0) \colon \Delta_0 \\
&\quad - \frac{2a(2aK-1)}{(b+t)^{2aK}} \int_0^t (b+s)^{2aK-2} \left( \widetilde \theta_{\mathrm{exp},s}^\delta - \theta \right) \psi_1(X_s, Z_s) \colon \Delta_s \dd s \\
&\quad + \frac{2a^2}{(b+t)^{2aK}} \int_0^t (b+s)^{2aK-2} \left[ \left( \widetilde \theta_{\mathrm{exp},s}^\delta - \theta \right) \psi_1(X_s, Z_s) + \Delta_s \psi_1(X_s, Z_s)^\top \right] \colon \\
&\hspace{8.5cm} \left[ \left(\widetilde \theta_{\mathrm{exp},s}^\delta - \theta \right) f(X_s) \otimes f(Z_s) \right] \dd s, \\
&\quad - \frac{2a^2}{(b+t)^{2aK}} \int_0^t (b+s)^{2aK-2} \left[ \left( \widetilde \theta_{\mathrm{exp},s}^\delta - \theta \right) \psi_1(X_s, Z_s) + \Delta_s \psi_1(X_s, Z_s)^\top \right] \colon \\
&\hspace{8.5cm} \left[ \sqrt{2D^S} \dd W_s \otimes f(Z_s) \right], \\
&\quad - \frac{2a^3 \norm{\sqrt{D^S}}^2}{(b+t)^{2aK}} \int_0^t (b + s)^{2aK-3} \left( \psi_1(X_s, Z_s) + \psi_1(X_s, Z_s)^\top \right) \colon f(Z_s) \otimes f(Z_s) \dd s.
\end{aligned}
\end{equation}
Then, we have
\begin{equation}
\E \left[ I_t^3 \right] \le \frac{C}{(b+t)^{2aK}} \int_0^t (b+s)^{2aK-2} \dd s \le C \left( \frac1{b+t} + \frac1{(b+t)^{2aK}} \right),
\end{equation}
and $\E[I_t^4] = 0$ since by the boundedness of the estimator $\widetilde \theta_{\mathrm{exp},t}^\delta $ due to \cref{lem:bounded_moments} we get
\begin{equation}
\E \left[ (I_t^4)^2 \right] = 4 \int_0^t \xi_s^2 e^{-4K \int_s^t \xi_r \dd r} \norm{\sqrt{2D^S} \Delta_s f(Z_s)}^2 \dd s \le C \left( \frac1{b+t} + \frac1{(b+t)^{4aK}} \right).
\end{equation}
Therefore, we deduce that $\Delta_t \to 0$ in $L^2$, which in turn gives desired results.
\end{proof}

\begin{proof}[Proof of \cref{thm:SGDCT_exp_ok_eps}]
The desired result is obtained applying \cref{pro:limit_SGDCT_t_YESfilter}, and, similarly to the proof of \cref{thm:exp_ok_eps}, \cref{lem:magic_formulas}.
\end{proof}

\section{Lévy area correction} \label{sec:Levy}

In this section, we demonstrate that the MLE and SGDCT estimators, that were introduced and analyzed in the previous sections, can be extended to SDEs driven by colored multiplicative noise. We are primarily interested in the case where the limiting SDE contains an additional drift term that is due to the Lévy area correction introduced in \cref{sec:setting}; compare the Stratonovich stochastic integral in the Wong--Zakai theorem. Consider, in particular, the setting of \cref{ex:multiD}, let $d=2$ and set the drift function $h \colon \R^2 \to \R^2$ and the diffusion function $g \colon \R^2 \to \R^{2 \times 2}$ to 
\begin{equation}
h(x) = - \theta x \qquad \text{and} \qquad g(x) = \sqrt{\kappa + \beta \norm{x}^2} I,
\end{equation}
where $\theta,\kappa,\beta$ are positive constants. Then the system \eqref{eq:system_multiD} reads
\begin{equation} \label{eq:systemLevy2D}
\begin{aligned}
\d X^\epl_t &= - \theta X^\epl_t \dd t + \sqrt{\kappa + \beta \norm{X^\epl_t}^2} \frac{Y^\epl_t}{\epl} \dd t, \\
\d Y^\epl_t &= - \frac{A}{\epl^2} Y^\epl_t \dd t + \frac{\sqrt{\eta}}{\epl} \dd W_t,
\end{aligned}
\end{equation}
and the limit equation~\eqref{eq:limit_multiD} becomes
\begin{equation} \label{eq:limitLevy2D}
\d X_t = - L X_t \dd t + \sqrt{\kappa_0 + \beta_0 \norm{X_t}^2} \dd W_t,
\end{equation}
where $\kappa_0 = \kappa \eta/\rho$, $\beta_0 = \beta \eta/\rho$ and
\begin{equation} \label{eq:L_def}
L = \left( \theta - \frac{\beta_0}{2} \right) I + \frac{\gamma \beta_0}{2 \alpha} J. 
\end{equation}
We can easily verify that \cref{as:ergodicity} is satisfied in this framework because $f(x) = -x$ and
\begin{equation}
- \theta x \cdot x = - \theta \norm{x}^2 \qquad \text{and} \qquad A y \cdot y = \alpha \norm{y}^2.
\end{equation}
Moreover, we assume that the parameters are chosen so that the processes are ergodic with a unique invariant measure, and in particular that
\begin{equation} \label{eq:condition_L}
\theta > \beta_0/2,
\end{equation}
which is necessary from the definition of $L$ in \eqref{eq:L_def}. We are interested inferring the parameters in the drift function, i.e., the four components of the matrix $L \in \R^{2\times2}$,  of the limit SDE \eqref{eq:limitLevy2D} from observations of the process $X_{\epl}^t$ in~\eqref{eq:systemLevy2D}. Similarly to the case of additive noise, we propose the MLE estimator
\begin{equation} \label{eq:exp_Levy}
\widehat L_{\mathrm{exp}}^\delta(X^\epl,T) = - \left( \int_0^T \d X_t^\epl \otimes Z_t^\epl \right) \left( \int_0^T X_t^\epl \otimes Z_t^\epl \dd t \right)^{-1},
\end{equation}
and the SGDCT estimator which solves the SDEs
\begin{equation}
\d \widetilde L_{\mathrm{exp},t}^{\delta,\epl} = - \xi_t \d \mathcal I_{\mathrm{exp},t}^{\delta,\epl} \otimes Z_t^\epl, \qquad \d \mathcal I_{\mathrm{exp},t}^{\delta,\epl} = \d X_t^\epl + \widetilde L_{\mathrm{exp},t}^{\delta,\epl} X_t^\epl \dd t,
\end{equation}
or equivalently
\begin{equation} \label{eq:SGDCT_exp_Levy}
\d \widetilde L^{\delta,\epl}_{\mathrm{exp},t} = - \xi_t \widetilde L_{\mathrm{exp},t}^{\delta,\epl} X_t^\epl \otimes Z_t^\epl \dd t - \xi_t \d X_t^\epl \otimes Z_t^\epl.
\end{equation}
In the next sections we study the asymptotic unbiasedness of the MLE estimator $\widehat L_{\mathrm{exp}}^\delta(X^\epl,T)$ and the SGDCT estimator $\widetilde L^{\delta,\epl}_{\mathrm{exp},t}$ in the limit as time goes to infinity and $\epl\to0$. Since the main ideas in the convergence analysis are similar to the case of additive noise, we only give a sketch of the proofs. Then, in the numerical experiments in \cref{sec:num_Levy} we observe that both the estimators are able to accurately infer the unknown matrix $L$ in the limit SDE. 

\subsection{Asymptotic unbiasedness for MLE estimator}

Similarly to the analysis in \cref{sec:proofs}, we first consider the following system of SDEs for the stochastic processes $X_t^\epl$ and $Y_t^\epl$ together with the additional equation \eqref{eq:filtered_data} for the filtered data $Z_t^\epl$
\begin{equation} \label{eq:system_Levy}
\begin{aligned}
\d X^\epl_t &= - \theta X^\epl_t \dd t + \sqrt{\kappa + \beta \norm{X^\epl_t}^2} \frac{Y^\epl_t}{\epl} \dd t, \\
\d Y^\epl_t &= - \frac{A}{\epl^2} Y^\epl_t \dd t + \frac{\sqrt{\eta}}{\epl} \dd W_t, \\
\d Z^\epl_t &= - \frac1\delta (Z^\epl_t - X^\epl_t) \dd t,
\end{aligned}
\end{equation}
and the corresponding system obtained in the limit as $\epl\to0$
\begin{equation}
\begin{aligned} \label{eq:system_limit_Levy}
\d X_t &= - L X_t \dd t + \sqrt{\kappa_0 + \beta_0 \norm{X_t}^2} \dd W_t, \\
\d Z_t &= - \frac1\delta (Z_t - X_t) \dd t.
\end{aligned}
\end{equation}
We verify that the system of SDEs are hypoelliptic, and therefore the measures induced by the stochastic processes admit smooth densities with respect to the Lebesgue measure. 

\begin{lemma} \label{lem:Lebesgue_Levy}
Let $\nu_t^\epl$ and $\nu_t$ be the measures at time $t$ induced by the joint processes $(X_t^\epl,Y_t^\epl,Z_t^\epl)$ and $(X_t,Z_t)$ given by equations \eqref{eq:system_Levy} and \eqref{eq:system_limit_Levy}, respectively. Then, the measures $\nu_t^\epl$ and $\nu_t$ admit smooth densities $\varphi^\epl_t$ and $\varphi_t$ with respect to the Lebesgue measure.
\end{lemma}
\begin{proof}
Let us first consider the system driven by colored noise. The generator of the joint process $(X_t^\epl, Y_t^\epl, Z_t^\epl)$ is
\begin{equation}
\begin{aligned}
\mathcal L^\epl &= - \theta x \cdot \nabla_x + \frac1\epl \sqrt{\kappa + \beta \norm{x}^2} y \cdot \nabla_x - \frac1{\epl^2} A y \cdot \nabla_y - \frac1\delta (z-x) \cdot \nabla_z + \frac{\eta}{2\epl^2} \Delta_y \\
&\eqdef \mathcal X_0 + \frac{\eta}{2\epl^2} \sum_{i=1}^n \mathcal X_i^2,
\end{aligned}
\end{equation}
where
\begin{equation}
\begin{aligned}
\mathcal X_0 &= - \theta x \cdot \nabla_x + \frac1\epl \sqrt{\kappa + \beta \norm{x}^2} y \cdot \nabla_x - \frac1{\epl^2} A y \cdot \nabla_y - \frac1\delta (z-x) \cdot \nabla_z \\
\mathcal X_i &= \frac{\partial}{\partial y_i}, \qquad i = 1,2.
\end{aligned}
\end{equation}
The commutator $[\mathcal X_0, \mathcal X_i]$ is
\begin{equation}
[\mathcal X_0, \mathcal X_i] = - \frac1\epl \sqrt{\kappa + \beta \norm{x}^2} \frac{\partial}{\partial x_i} + \frac1{\epl^2} a_i \cdot \nabla_y,
\end{equation}
where $a_i$, $i = 1,2$, are the columns of the matrix $A$, and the commutator $[\mathcal X_0, [\mathcal X_0, \mathcal X_i]]$ is
\begin{equation}
\begin{aligned}
[\mathcal X_0, [\mathcal X_0, \mathcal X_i]] &= - \frac{\theta\kappa}{\epl \sqrt{\kappa + \beta \norm{x}^2}} \frac{\partial}{\partial x_i} + \frac{\beta}{\epl^2} \left( x_i y \cdot \nabla_x - y \cdot x \frac{\partial}{\partial x_i} \right) + \frac1{\epl\delta} \sqrt{\kappa + \beta \norm{x}^2} \frac{\partial}{\partial z_i} \\
&\quad - \frac1{\epl^3} \sqrt{\kappa + \beta \norm{x}^2} a_i \cdot \nabla_x + \frac1{\epl^4} a_i \cdot A^\top \nabla_y.
\end{aligned}
\end{equation}
Therefore, for any point $(x,y,z) \in \R^6$, the set
\begin{equation}
\mathcal H = \mathrm{Lie}\left( \mathcal X_i, [\mathcal X_0, \mathcal X_i], [\mathcal X_0, [\mathcal X_0, \mathcal X_i]]; i = 1, \dots, 2 \right)
\end{equation}
spans the tangent space of $\R^6$ at $(x,y,z)$. The desired result then follows from Hörmander’s theorem (see, e.g., \cite[Chapter 6]{Pav14}). Finally, the proof for the limit system $(X_t,Z_t)$ is analogous, and therefore we omit the details.
\end{proof}

We can now write the stationary Fokker--Planck equations for the processes $(X_t^\epl,Y_t^\epl,Z_t^\epl)$ and $(X_t,Z_t)$ given by equations \eqref{eq:system_Levy} and \eqref{eq:system_limit_Levy}, i.e.,
\begin{equation} \label{eq:FP_colored_Levy}
\begin{aligned}
\nabla_x \cdot \left( \theta x \rho^\epl(x,y,z) \right) - \frac1\epl \nabla_x \cdot \left( \sqrt{\kappa + \beta \norm{x}^2} y \rho^\epl(x,y,z) \right) + \frac1\delta \nabla_z \cdot \left( (z-x) \rho^\epl(x,y,z) \right) \\
+ \frac{1}{\epl^2} \nabla_y \cdot \left( Ay\rho^\epl(x,y,z) \right) + \frac{\eta}{2 \epl^2} \Delta_y \rho^\epl(x,y,z) &= 0,
\end{aligned}
\end{equation}
and
\begin{equation} \label{eq:FP_limit_Levy}
\nabla_x \cdot \left( Lx \rho(x,z) \right) + \frac1\delta \nabla_z \cdot \left( (z-x) \rho(x,z) \right) + \frac12 \Delta_x \left( (\kappa_0 + \beta_0 \norm{x}^2) \rho(x,z) \right) = 0,
\end{equation}
respectively.

Employing these Fokker--Planck equations, we show the following technical result, whose identities are analogous to the ones in \cref{lem:magic_formulas}.

\begin{lemma} \label{lem:magic_formulas_Levy}
The following equalities hold true
\begin{equation}
\begin{aligned}
(i)& \quad \frac1\epl \E^{\mu^\epl} \left[ \sqrt{\kappa + \beta \norm{X^\epl}^2} Y^\epl \otimes Z^\epl \right] = \theta \E^{\mu^\epl} \left[ X^\epl \otimes Z^\epl \right] + \frac1\delta \E^{\mu^\epl} \left[ X^\epl \otimes (Z^\epl - X^\epl) \right], \\
(ii)& \quad \frac1\delta \E^\mu \left[ X \otimes (Z - X) \right] = - L \E^\mu \left[ X \otimes Z \right].
\end{aligned}
\end{equation}
\end{lemma}
\begin{proof}
Let us first consider point $(i)$. Multiplying equation \eqref{eq:FP_colored_Levy} by $x \otimes z$, integrating over $\R^6$ and then by parts, and noting that
\begin{equation}
\begin{aligned}
\int_{\R^6} x \otimes z \nabla_x \cdot \left( \theta x \rho^\epl(x,y,z) \right) &= - \theta \int_{\R^6} x \otimes z \rho^\epl(x,y,z), \\
\frac1\epl \int_{\R^6} x \otimes z \nabla_x \cdot \left( \sqrt{\kappa + \beta \norm{x}^2} y \rho^\epl(x,y,z) \right) &= - \frac1\epl \int_{\R^6} \sqrt{\kappa + \beta \norm{x}^2} y \otimes z \rho^\epl(x,y,z), \\
\frac1\delta \int_{\R^6} x \otimes z \nabla_z \cdot \left( (z - x) \rho^\epl(x,y,z) \right) &= - \frac1\delta \int_{\R^6} x \otimes (z - x) \rho^\epl(x,y,z),
\end{aligned}
\end{equation}
we get the desired result. Analogously, multiplying equation \eqref{eq:FP_limit_Levy} by $x \otimes z$, integrating over $\R^4$ and then by parts, we obtain point $(ii)$, which concludes the proof.
\end{proof}

The asymptotic unbiasedness of the estimator $\widehat L_{\mathrm{exp}}^\delta(X^\epl,T)$ defined in \eqref{eq:exp_Levy} is finally given by the following theorem.

\begin{theorem} \label{thm:exp_ok_eps_Levy}
Let $\widehat L_{\mathrm{exp}}^\delta(X^\epl,T)$ be defined in \eqref{eq:exp_Levy}. Then it holds
\begin{equation}
\lim_{\epl\to0} \lim_{T\to\infty} \widehat L_{\mathrm{exp}}^\delta(X^\epl,T) = L, \qquad a.s.
\end{equation}
\end{theorem}
\begin{proof}
Replacing the first equation from \eqref{eq:systemLevy2D} in the definition of the estimator \eqref{eq:exp_Levy} we have
\begin{equation}
\widehat L_{\mathrm{exp}}^\delta(X^\epl,T) = \theta - \frac1\epl \left( \int_0^T \sqrt{\kappa + \beta \norm{X^\epl_t}^2} Y^\epl_t \otimes Z_t^\epl \dd t \right) \left( \int_0^T X_t^\epl \otimes Z_t^\epl \dd t \right)^{-1},
\end{equation}
and by the ergodic theorem we obtain 
\begin{equation}
\lim_{T\to\infty} \widehat L_{\mathrm{exp}}^\delta(X^\epl,T) = \theta - \frac1\epl \E^{\mu^\epl} \left[ \sqrt{\kappa + \beta \norm{X^\epl}^2} Y^\epl \otimes Z^\epl \right] \E^{\mu^\epl} \left[ X^\epl \otimes Z^\epl \right]^{-1}, \qquad a.s.
\end{equation}
Employing formula $(i)$ in \cref{lem:magic_formulas_Levy} we have
\begin{equation}
\lim_{T\to\infty} \widehat L_{\mathrm{exp}}^\delta(X^\epl,T) = - \frac1\delta \E^{\mu^\epl} \left[ X^\epl \otimes (Z^\epl - X^\epl) \right] \E^{\mu^\epl} \left[ X^\epl \otimes Z^\epl \right]^{-1}, \qquad a.s.,
\end{equation}
which due to the weak convergence of the joint process $(X_t^\epl,Z_t^\epl)$ to $(X_t,Z_t)$ gives
\begin{equation}
\lim_{\epl\to0} \lim_{T\to\infty} \widehat L_{\mathrm{exp}}^\delta(X^\epl,T) = - \frac1\delta \E^\mu \left[ X \otimes (Z - X) \right] \E^\mu \left[ X \otimes Z \right]^{-1}, \qquad a.s.
\end{equation}
Finally, formula $(ii)$ in \cref{lem:magic_formulas_Levy} yields the desired result.
\end{proof}

\subsection{Asymptotic unbiasedness for SGDCT estimator}

We proceed similarly to the analysis in \cref{sec:proofs}. We recall that, for the sake of the proof of convergence for the SGDCT estimator, we wrap all the processes in the $d$-dimensional torus $\mathbb T^d$, and therefore we work under the additional assumption that the state space is compact, as explained in \cref{rem:compact_state_space,rem:compact_state_space_2}. The first step consists in showing that \cref{as:positive_definite} is satisfied. 

\begin{lemma} \label{lem:positive_definite_Levy}
There exists a constant $K>0$ such that for all $V \in \R^{2\times2}$
\begin{equation}
V \E^{\mu^\epl} \left[ X^\epl \otimes Z^\epl \right] \colon V \ge K \norm{V}^2.
\end{equation}
\end{lemma}
\begin{proof}
We first prove that the statement holds true for vectors $v \in \R^2$ and for the limit equation, i.e.,
\begin{equation}
v^\top \E^\mu \left[ X \otimes Z \right] v \ge K \norm{v}^2.
\end{equation}
The result then follows from \cref{rem:positive_definite}. Multiplying equation \eqref{eq:FP_limit_Levy} by $x \otimes x$, integrating over $\R^4$ and then by parts, and noting that
\begin{equation}
\begin{aligned}
\int_{\R^4} x \otimes x \nabla_x \cdot \left( L x \rho(x,z) \right) &= - L \int_{\R^6} x \otimes x \rho^\epl(x,y,z) - \int_{\R^6} x \otimes x \rho^\epl(x,y,z) L^\top, \\
\frac12 \int_{\R^4} x \otimes x \Delta_x \left( (\kappa_0 + \beta_0 \norm{x}^2) \rho(x,z) \right) &= \int_{\R^4} (\kappa_0 + \beta_0 \norm{x}^2) \rho(x,z) I,
\end{aligned}
\end{equation}
we get
\begin{equation}
L \E^\mu \left[ X \otimes X \right] + \E^\mu \left[ X \otimes X \right] L^\top = \E^\mu \left[ \kappa_0 + \beta_0 \norm{X}^2 \right] I.
\end{equation}
By definition of $L$ in \eqref{eq:L_def}, the solution of the continuous Lyapunov equation is given by
\begin{equation}
\E^\mu \left[ X \otimes X \right] = \frac{\E^\mu \left[ \kappa_0 + \beta_0 \norm{X}^2 \right]}{2 \theta - \beta_0} I.
\end{equation}
We consider the equation in \cref{lem:magic_formulas_Levy}$(ii)$, which implies
\begin{equation}
\E^\mu \left[ X \otimes Z \right] = (I + \delta L)^{-1} \E \left[ X \otimes X \right] = \frac{\E^\mu \left[ \kappa_0 + \beta_0 \norm{X}^2 \right]}{2 \theta - \beta_0} (I + \delta L)^{-1}.
\end{equation}
Noting that
\begin{equation}
\E^\mu \left[ \kappa_0 + \beta_0 \norm{X}^2 \right] \ge \kappa_0,
\end{equation}
and that 
\begin{equation}
\begin{aligned}
(I + \delta L)^{-1} &= \left[ \left( 1 + \delta \theta - \frac{\delta \beta_0}2\right) I + \frac{\delta \gamma \beta_0}{2 \alpha} J \right]^{-1} \\
&= \left[ \left( 1 + \delta \theta - \frac{\delta \beta_0}2\right)^2 + \frac{\delta^2 \gamma^2 \beta_0^2}{4 \alpha^2} \right]^{-1} \left[ \left( 1 + \delta \theta - \frac{\delta \beta_0}2\right) I - \frac{\delta \gamma \beta_0}{2 \alpha} J \right],
\end{aligned}
\end{equation}
we obtain
\begin{equation}
v^\top \E^\mu \left[ X \otimes Z \right] v \ge \kappa_0 \left[ \left( 1 + \delta \theta - \frac{\delta \beta_0}2\right)^2 + \frac{\delta^2 \gamma^2 \beta_0^2}{4 \alpha^2} \right]^{-1} \frac{2 + 2\delta\theta - \delta\beta_0}{4\theta - 2\beta_0} \norm{v}^2 \eqdef K \norm{v}^2, 
\end{equation}
where $K>0$ due to \eqref{eq:condition_L}, and which concludes the proof.
\end{proof}

The second ingredient necessary to prove the convergence of the SGDCT estimator $\widetilde L_{\mathrm{exp},t}^{\delta,\epl}$ is bounding its moments uniformly in time.

\begin{lemma} \label{lem:bounded_moments_Levy}
Let the estimators $L_{\mathrm{exp},t}^{\delta,\epl}$ be defined by the SDE \eqref{eq:SGDCT_exp_Levy}. For all $p\ge1$ there exists a constant $C>0$ independent of time such that following bound holds
\begin{equation}
\E \left[ \norm{\widetilde L_{\mathrm{exp},t}^{\delta,\epl}}^p \right] \le C.
\end{equation}
\end{lemma}
\begin{proof}
We only give a sketch of the proof since it follows the same steps of the proof of \cref{lem:bounded_moments}. Let us consider the SDE for the estimator and rewrite it as
\begin{equation}
\begin{aligned}
\d \widetilde L^{\delta,\epl}_{\mathrm{exp},t} &= - \xi_t \widetilde L_{\mathrm{exp},t}^{\delta,\epl} \E^{\mu^\epl} \left[ X^\epl \otimes Z^\epl \right] \dd t + \xi_t \theta X_t^\epl \otimes Z_t^\epl \dd t - \frac1\epl \xi_t \sqrt{\kappa + \beta \norm{X_t^\epl}^2} Y_t^\epl \otimes Z_t^\epl \dd t \\
&\quad - \xi_t \widetilde L^{\delta,\epl}_{\mathrm{exp},t} \left( X_t^\epl \otimes Z_t^\epl - \E^{\mu^\epl} \left[ X^\epl \otimes Z^\epl \right] \right) \dd t.
\end{aligned}
\end{equation}
We then introduce the auxiliary process $\widetilde {\mathfrak L}^{\delta,\epl}_{\mathrm{exp},t}$ defined by the SDE
\begin{equation}
\begin{aligned}
\d \widetilde{\mathfrak L}^{\delta,\epl}_{\mathrm{exp},t} &= - \xi_t K \widetilde{\mathfrak L}_{\mathrm{exp},t}^{\delta,\epl} \dd t + \xi_t \theta X_t^\epl \otimes Z_t^\epl \dd t - \frac1\epl \xi_t \sqrt{\kappa + \beta \norm{X_t^\epl}^2} Y_t^\epl \otimes Z_t^\epl \dd t \\
&\quad - \xi_t \widetilde{\mathfrak L}^{\delta,\epl}_{\mathrm{exp},t} \left( X_t^\epl \otimes Z_t^\epl - \E^{\mu^\epl} \left[ X^\epl \otimes Z^\epl \right] \right) \dd t,
\end{aligned}
\end{equation}
which can be written as
\begin{equation}
\begin{aligned}
\widetilde{\mathfrak L}^{\delta,\epl}_{\mathrm{exp},t} &= e^{-K \int_0^t \xi_r \dd r} \widetilde{\mathfrak L}^{\delta,\epl}_{\mathrm{exp},0} + \int_0^t \xi_s e^{-K \int_s^t \xi_r \dd r} \theta X_s^\epl \otimes Z_s^\epl \dd s \\
&\quad - \frac1\epl \int_0^t \xi_s e^{-K \int_s^t \xi_r \dd r} \sqrt{\kappa + \beta \norm{X_s^\epl}^2} Y_s^\epl \otimes Z_s^\epl \dd s \\
&\quad - \int_0^t \xi_s e^{-K \int_s^t \xi_r \dd r} \xi_t \widetilde{\mathfrak L}^{\delta,\epl}_{\mathrm{exp},s} \left( X_s^\epl \otimes Z_s^\epl - \E^{\mu^\epl} \left[ X^\epl \otimes Z^\epl \right] \right) \dd s,
\end{aligned}
\end{equation}
and which, by \cref{lem:positive_definite_Levy} and due to the comparison theorem \cite[Theorem 1.1]{IkW77}, is such that
\begin{equation} \label{eq:comparison_Levy}
\Pr \left( \norm{\widetilde L_{\mathrm{exp},t}^{\delta,\epl}} \le \norm{\widetilde {\mathfrak L}_{\mathrm{exp},t}^{\delta,\epl}}, t \ge 0 \right) = 1.
\end{equation}
We now proceed analogously to the proof of \cref{lem:bounded_moments}, we consider the solution $\psi^\epl$ of the Poisson problem for the generator and we apply Itô's lemma. We remark that \cref{lem:Poisson_expansion} holds true also for the generators of the processes \eqref{eq:system_Levy} and \eqref{eq:system_limit_Levy}, since the hypoelliptic setting is guaranteed by \cref{lem:Lebesgue_Levy}. After some computation we get
\begin{equation}
\begin{aligned}
\widetilde{\mathfrak L}_{\mathrm{exp},t}^{\delta,\epl} &= \widetilde{\mathfrak L}_{\mathrm{exp},0}^{\delta,\epl} \left( \frac{b}{b+t} \right)^{aK} - \frac{a}{b+t} \widetilde{\mathfrak L}_{\mathrm{exp},t}^{\delta,\epl} \psi^\epl(X_t^\epl, Y_t^\epl, Z_t^\epl) + \frac{a b^{aK-1}}{(b+t)^{aK}} \widetilde{\mathfrak L}_{\mathrm{exp},0}^{\delta,\epl} \psi^\epl(X_0^\epl, Y_0^\epl, Z_0^\epl) \\
&\quad + a \int_0^t \frac{(b+s)^{aK-1}}{(b+t)^{aK}} \theta X_s^\epl \otimes Z_s^\epl \dd s - \frac{a}\epl \int_0^t \frac{(b+s)^{aK-1}}{(b+t)^{aK}} \sqrt{\kappa + \beta \norm{X_s^\epl}^2} Y_s^\epl \otimes Z_s^\epl \dd s \\
&\quad + a(aK-1) \int_0^t \frac{(b+s)^{aK-2}}{(b+t)^{aK}} \widetilde{\mathfrak L}_{\mathrm{exp},s}^{\delta,\epl} \psi^\epl(X_s^\epl, Y_s^\epl, Z_s^\epl) \dd s \\
&\quad + \frac{a}\epl \int_0^t \frac{(b+s)^{aK-1}}{(b+t)^{aK}} \widetilde{\mathfrak L}_{\mathrm{exp},s}^{\delta,\epl} \nabla_y \psi^\epl(X_s^\epl, Y_s^\epl, Z_s^\epl) \cdot \sigma \d W_s \\
&\quad - a^2 \int_0^t \frac{(b+s)^{aK-2}}{(b+t)^{aK}} \widetilde{\mathfrak L}_{\mathrm{exp},s}^{\delta,\epl} \left( KI + X_s^\epl \otimes Z_s^\epl - \E^{\mu^\epl} \left[ X^\epl \otimes Z^\epl \right] \right) \psi^\epl(X_s^\epl, Y_s^\epl, Z_s^\epl) \dd s \\
&\quad + a^2 \int_0^t \frac{(b+s)^{aK-2}}{(b+t)^{aK}} \left( \theta X_s^\epl \otimes Z_s^\epl + \frac1\epl \sqrt{\kappa + \beta \norm{X_s^\epl}^2} Y_s^\epl \otimes Z_s^\epl \right) \psi^\epl(X_s^\epl, Y_s^\epl, Z_s^\epl) \dd s,
\end{aligned}
\end{equation}
which implies
\begin{equation}
\E \left[ \norm{\widetilde{\mathfrak L}_{\mathrm{exp},t}^{\delta,\epl}}^{2q} \right] \le C + C \int_0^t \frac{(b+s)^{aK-2}}{(b+t)^{aK}} \E \left[ \norm{\widetilde{\mathfrak L}_{\mathrm{exp},s}^{\delta,\epl}}^{2q} \right] \dd s.
\end{equation}
Finally, the desired result follows from Grönwall's inequality and equation \eqref{eq:comparison_Levy}.
\end{proof}

We now compute the limit of the SGDCT estimator as time tends to infinity. The proof of next lemma is similar to the proof of \cref{pro:limit_SGDCT_t_YESfilter}.

\begin{proposition} \label{pro:limit_SGDCT_t_YESfilter_levy}
Let $X_t^\epl, Y_t^\epl, Z_t^\epl$ be solutions of system \eqref{eq:system_Levy}. Then it holds
\begin{equation}
\lim_{t\to\infty} \widetilde L_{\mathrm{exp},t}^{\delta,\epl} = \theta - \frac1\epl \E^{\mu^\epl} \left[ \sqrt{\kappa + \beta \norm{X^\epl}^2} Y^\epl \otimes Z^\epl \right] \E^{\mu^\epl} \left[ X^\epl \otimes Z^\epl \right]^{-1}, \qquad \text{in } L^2.
\end{equation}
\end{proposition}
\begin{proof}
Let us consider the SDE for the estimator \eqref{eq:SGDCT_exp_Levy} and rewrite it as
\begin{equation}
\begin{aligned}
\d & \left( \widetilde L_{\mathrm{exp},t}^{\delta,\epl} - \theta + \frac1\epl \E^{\mu^\epl} \left[ \sqrt{\kappa + \beta \norm{X^\epl}^2} Y^\epl \otimes Z^\epl \right] \E^{\mu^\epl} \left[ X^\epl \otimes Z^\epl \right]^{-1} \right) \\
&= - \xi_t \left( \widetilde L_{\mathrm{exp},t}^{\delta,\epl} - \theta + \frac1\epl \E^{\mu^\epl} \left[ \sqrt{\kappa + \beta \norm{X^\epl}^2} Y^\epl \otimes Z^\epl \right] \E^{\mu^\epl} \left[ X^\epl \otimes Z^\epl \right]^{-1} \right) \E^{\mu^\epl} \left[ X^\epl \otimes Z^\epl \right] \dd t \\
&\quad - \xi_t \left( \widetilde L_{\mathrm{exp},t}^{\delta,\epl} - \theta \right) \left( X_t^\epl \otimes Z_t^\epl - \E^{\mu^\epl} \left[ X^\epl \otimes Z^\epl \right] \right) \dd t \\
&\quad - \frac1\epl \xi_t \left( \sqrt{\kappa + \beta \norm{X_t^\epl}^2} Y_t^\epl \otimes Z_t^\epl - \E^{\mu^\epl} \left[ \sqrt{\kappa + \beta \norm{X^\epl}^2} Y^\epl \otimes Z^\epl \right] \right) \dd t.
\end{aligned}
\end{equation}
Letting
\begin{equation}
\Delta_t^\epl = \widetilde L_{\mathrm{exp},t}^{\delta,\epl} - \theta + \frac1\epl \E^{\mu^\epl} \left[ \sqrt{\kappa + \beta \norm{X^\epl}^2} Y^\epl \otimes Z^\epl \right] \E^{\mu^\epl} \left[ X^\epl \otimes Z^\epl \right]^{-1} \quad \text{and} \quad \Gamma_t^\epl = \norm{\Delta_t^\epl}^2,
\end{equation}
by Itô's lemma we have
\begin{equation}
\begin{aligned}
\d \Gamma_t^\epl &= - 2 \xi_t \Delta_t^\epl \E^{\mu^\epl} \left[ X^\epl \otimes Z^\epl \right] \colon \Delta_t^\epl \dd t \\
&\quad - 2 \xi_t \left( \widetilde L_{\mathrm{exp},t}^{\delta,\epl} - \theta \right) \left( X_t^\epl \otimes Z_t^\epl - \E^{\mu^\epl} \left[ X^\epl \otimes Z^\epl \right] \right) \colon \Delta_t^\epl \dd t \\
&\quad - \frac2\epl \xi_t \left( \sqrt{\kappa + \beta \norm{X_t^\epl}^2} Y_t^\epl \otimes Z_t^\epl - \E^{\mu^\epl} \left[ \sqrt{\kappa + \beta \norm{X^\epl}^2} Y^\epl \otimes Z^\epl \right] \right) \colon \Delta_t^\epl \dd t,
\end{aligned}
\end{equation}
which due to \cref{lem:positive_definite_Levy} implies
\begin{equation}
\begin{aligned}
\d \Gamma_t^\epl &\le - 2 K \xi_t \Gamma_t^\epl \dd t \\
&\quad - 2 \xi_t \left( \widetilde L_{\mathrm{exp},t}^{\delta,\epl} - \theta \right) \left( X_t^\epl \otimes Z_t^\epl - \E^{\mu^\epl} \left[ X^\epl \otimes Z^\epl \right] \right) \colon \Delta_t^\epl \dd t \\
&\quad - \frac2\epl \xi_t \left( \sqrt{\kappa + \beta \norm{X_t^\epl}^2} Y_t^\epl \otimes Z_t^\epl - \E^{\mu^\epl} \left[ \sqrt{\kappa + \beta \norm{X^\epl}^2} Y^\epl \otimes Z^\epl \right] \right) \colon \Delta_t^\epl \dd t.
\end{aligned}
\end{equation}
Then, by the comparison principle we obtain
\begin{equation}
\begin{aligned}
\Gamma_t^\epl &\le \norm{\Delta_0^\epl}^2 e^{-2K \int_0^t \xi_r \dd r} \\
&\quad -2 \int_0^t \xi_s e^{-2K \int_s^t \xi_r \dd r} \left( \widetilde \theta_{\mathrm{exp},s}^{\delta,\epl} - \theta \right) \left( f(X_s^\epl) \otimes f(Z_s^\epl) - \E^{\mu^\epl} \left[ f(X^\epl) \otimes f(Z^\epl) \right] \right) : \Delta_s^\epl \dd s \\
&\quad - \frac2\epl \int_0^t \xi_s e^{-2K \int_s^t \xi_r \dd r} \left( Y_s^\epl \otimes f(Z_s^\epl) - \E^{\mu^\epl} \left[ Y^\epl \otimes f(Z^\epl)  \right] \right) : \Delta_s^\epl \dd s \\
&\eqdef I_t^1 + I_t^2 + I_t^3.
\end{aligned}
\end{equation}
The last steps needed to bound these quantities and to obtain the desired result are similar to what is done at the end of the proof of \cref{pro:limit_SGDCT_t_YESfilter}, and we omit the details. We only remark that \cref{lem:Poisson_expansion} holds true also for the generators of the processes \eqref{eq:system_Levy} and \eqref{eq:system_limit_Levy}, since the hypoellipticity of these generators is guaranteed by \cref{lem:Lebesgue_Levy}, and that the moments of the estimator $\widetilde L_{\mathrm{exp},t}^{\delta,\epl}$ are bounded uniformly in time due to \cref{lem:bounded_moments_Levy}.
\end{proof}

We are now ready to show the asymptotic unbiasedness of the estimator $\widetilde L^{\delta,\epl}_{\mathrm{exp},t}$.

\begin{theorem} \label{thm:SGDCT_exp_ok_eps_Levy}
Let $\widetilde L_{\mathrm{exp},t}^{\delta,\epl}$ be defined in \eqref{eq:SGDCT_exp_Levy}. Then it holds
\begin{equation}
\lim_{\epl\to0} \lim_{t\to\infty} \widetilde L_{\mathrm{exp},t}^{\delta,\epl} = L, \qquad \text{in } L^2.
\end{equation}
\end{theorem}
\begin{proof}
The desired result is obtained applying \cref{pro:limit_SGDCT_t_YESfilter_levy} and, following the proof of \cref{thm:exp_ok_eps_Levy}, due to \cref{lem:magic_formulas_Levy}.
\end{proof}

\begin{remark}
Notice that in \cref{thm:SGDCT_exp_ok_eps_Levy}, as well as in the previous results in \cref{thm:exp_ok_eps,thm:SGDCT_exp_ok_eps,thm:exp_ok_eps_Levy}, the order of the limits is important and they cannot commute. In fact, the convergence of the stochastic processes with respect to the parameter $\epl$ is in law. Hence, as it is shown in the proofs, we first need to reach the expectations with respect to the invariant measures through the ergodic theorem, and therefore the infinite limit in time, and then let the correlation time vanish.
\end{remark}

\section{Numerical experiments} \label{sec:numerical_experiments}

In this section, we present a series of numerical experiments which confirm our theoretical results. Synthetic data are generated employing the Euler--Maruyama method with a fine time step $h = \epl^3$, where we set the scale parameter to $\epl = 0.05$ and $\epl= 0.1$, respectively. All the experiments are performed for $M = 100$ times.  The red/yellow lines in the displayed plots represent the average estimated values and the blue/green shades correspond to the standard deviations, while the dashed black lines are the exact values. The results are shown as functions of time. We finally remark that filtered data are generated setting the filtering width $\delta = 1$, and the learning rate for the SGDCT estimator is chosen to be $\xi_t = a/(b+t)$ where $a$ and $b$ will be specified in the next sections. Moreover, all the initial conditions for both the stochastic processes and the SGDCT estimators are set to be zeros, and the final time of integration is chosen such that the mean values of the estimators reach convergence, i.e., they stabilize and do not oscillate.

\subsection{Additive noise} \label{sec:num_additive}

\begin{figure}
\centering
\includegraphics[]{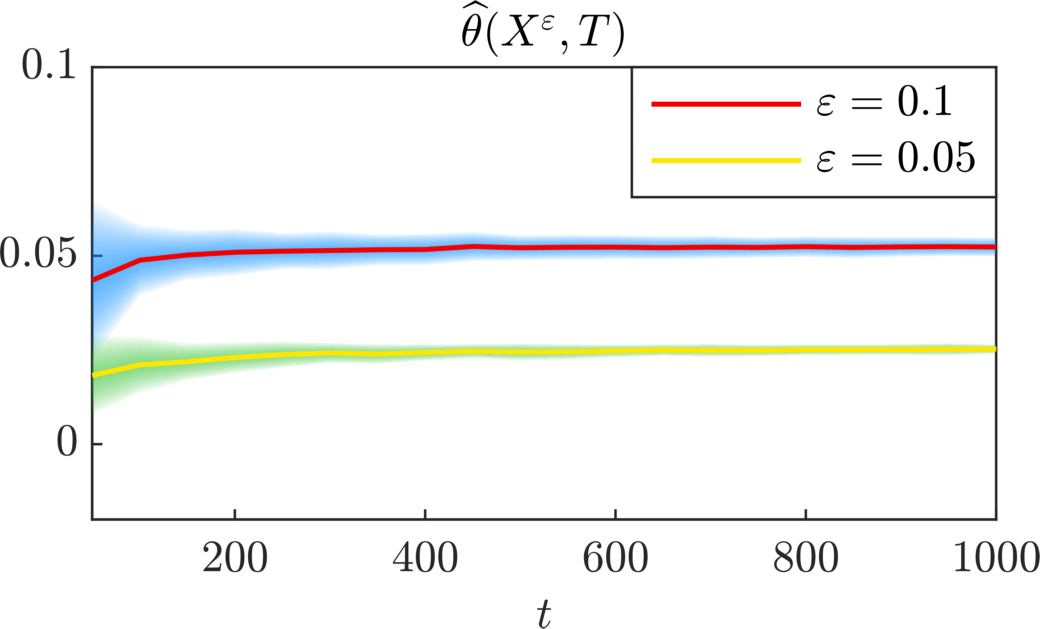}
\includegraphics[]{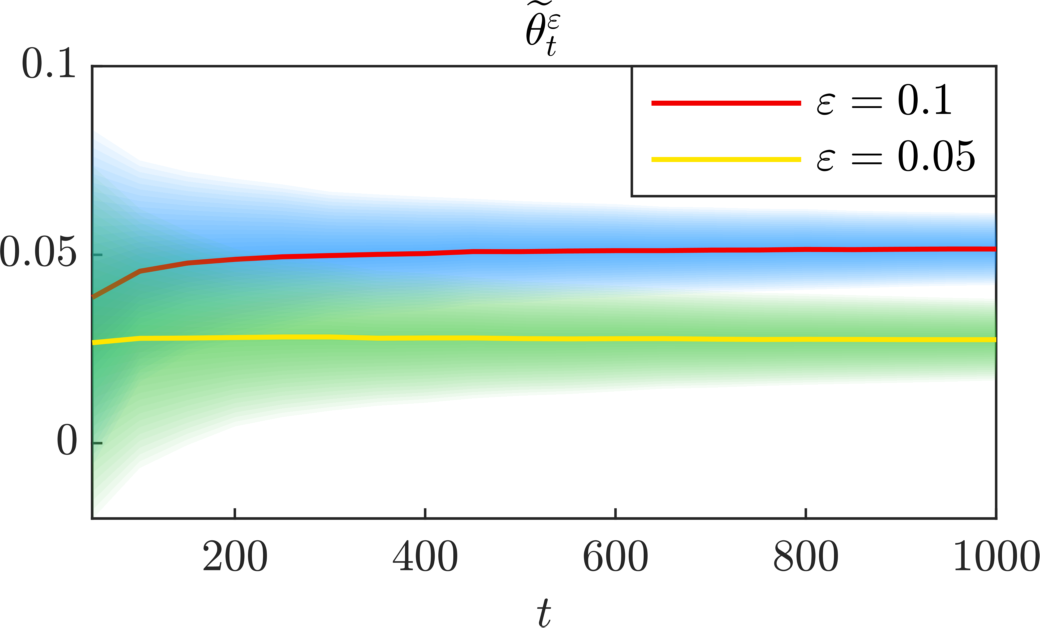}
\caption{MLE estimator $\widehat \theta(X^\epl,T)$ (left) and SGDCT estimator $\widetilde \theta^\epl_t$ (right) in the one dimensional case with additive colored noise.}
\label{fig:ko}
\end{figure}

Consider the system of SDEs \eqref{eq:system_multiD_additive} with additive colored noise and its limiting equation \eqref{eq:limit_multiD_additive} with white noise. We first verify that the estimators $\widehat \theta(X^\epl,T)$ and $\widetilde \theta_t^\epl$, which are only based on the original data $X_t^\epl$ driven by colored noise, are not able to correctly estimate the unknown drift coefficient $\theta$, as predicted by \cref{pro:MLE_ko,pro:SGDCT_ko}. We focus on the one-dimensional case, i.e., we let $d=\ell=n=m=1$, we set the parameters $\theta = G = A = \sigma = 1$, and we fix the final time $T = 1000$ and $\epl = 0.1$. Moreover, the parameters in the learning rate $\xi_t$ for the SGDCT estimator are chosen to be $a=1$ and $b=0.1$. In \cref{fig:ko}, we display the approximations provided by the estimators $\widehat \theta(X^\epl,T)$ and $\widetilde \theta_t^\epl$ as a function of time. We notice that the estimated values are close to zero independently of time. In particular, we observe that the smaller value of the scale parameter $\epl$ yields estimates closer to zero. This shows that the drift coefficient $\theta$ cannot be inferred from the data which originates from the SDE driven by colored noise, and, consequently, confirms that the original data must be preprocessed.

\begin{figure}
\centering
\includegraphics[]{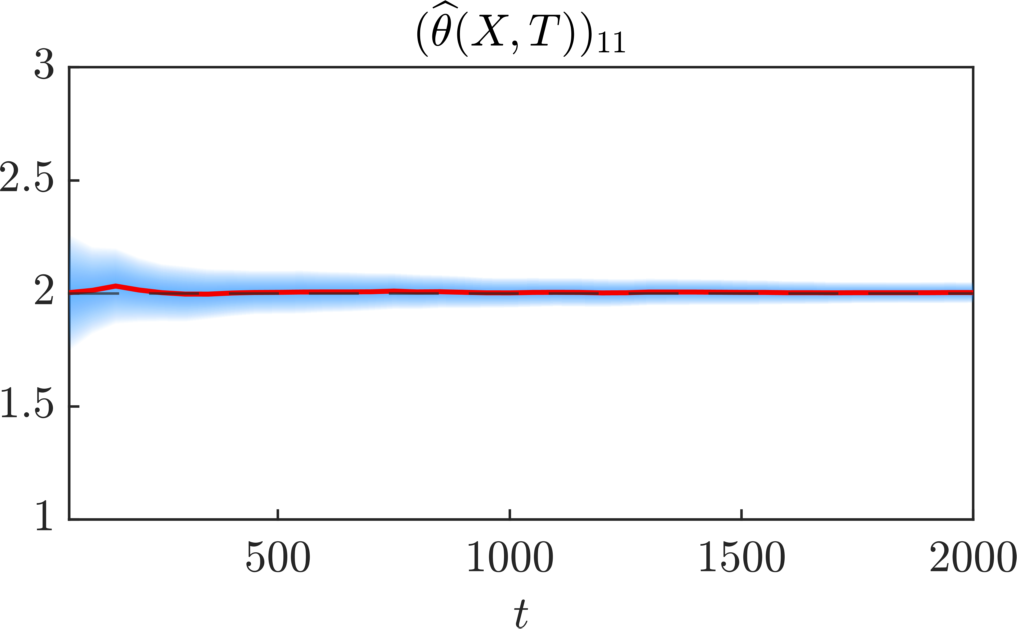}
\includegraphics[]{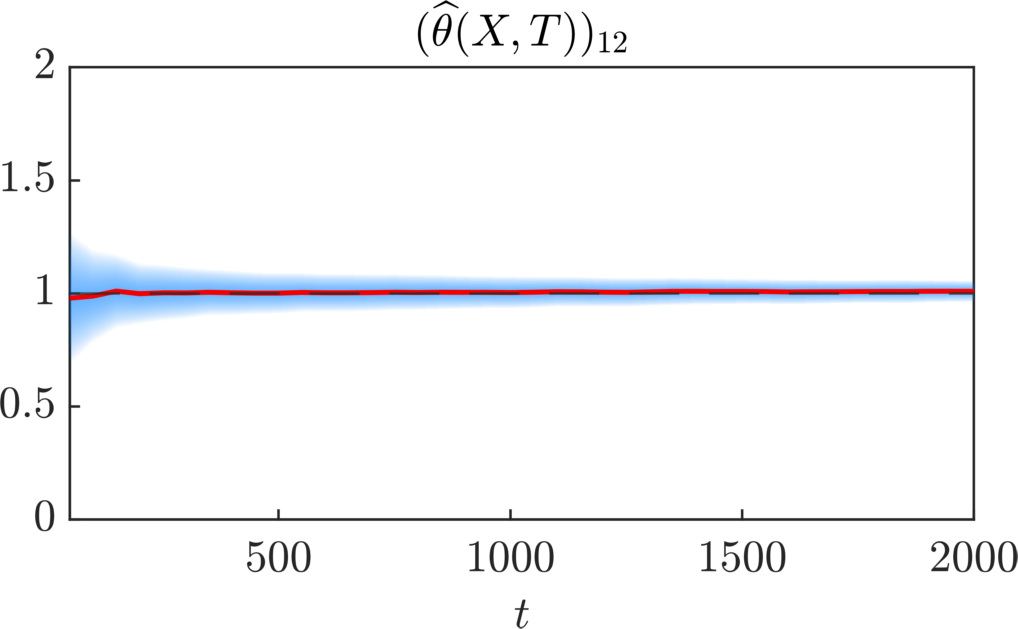}
\includegraphics[]{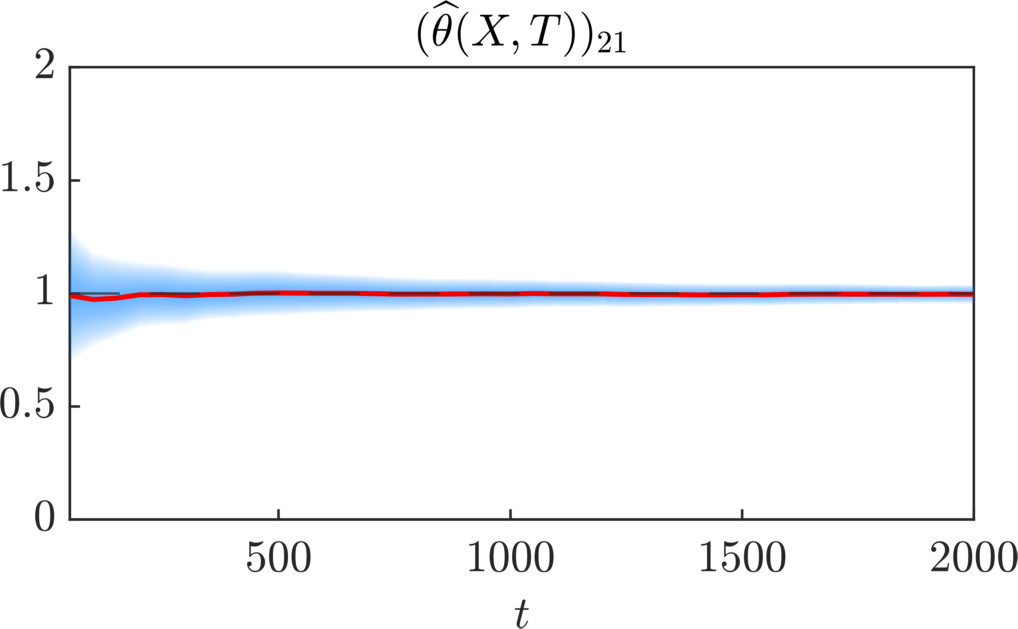}
\includegraphics[]{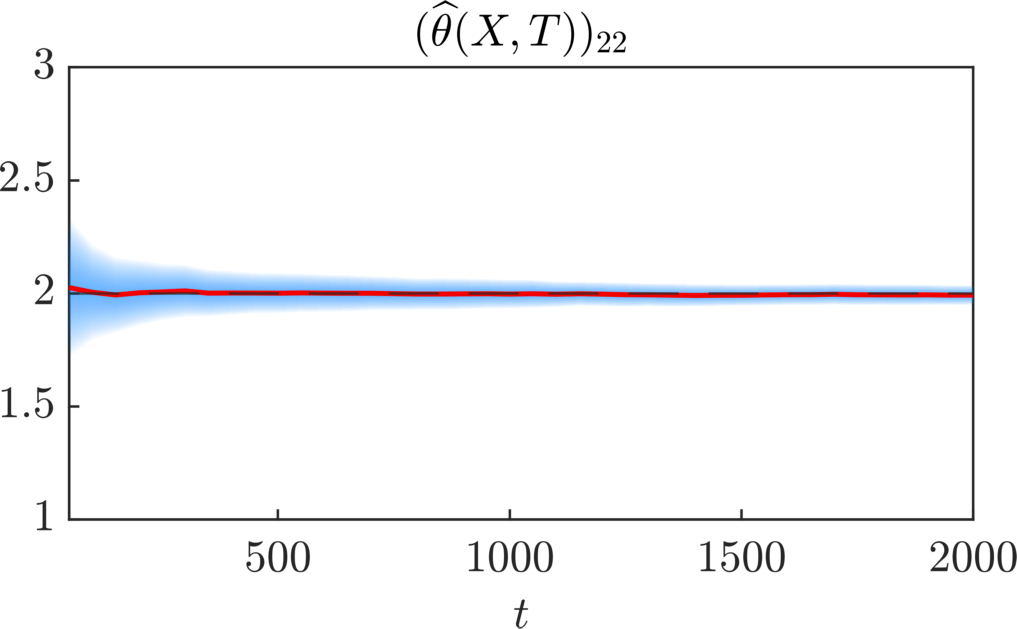}
\caption{Four components of the MLE estimator $\widehat \theta(X,T)$ with data from the limit equation in presence of additive white noise.}
\label{fig:MLE_ok}
\end{figure}

\begin{figure}
\centering
\includegraphics[]{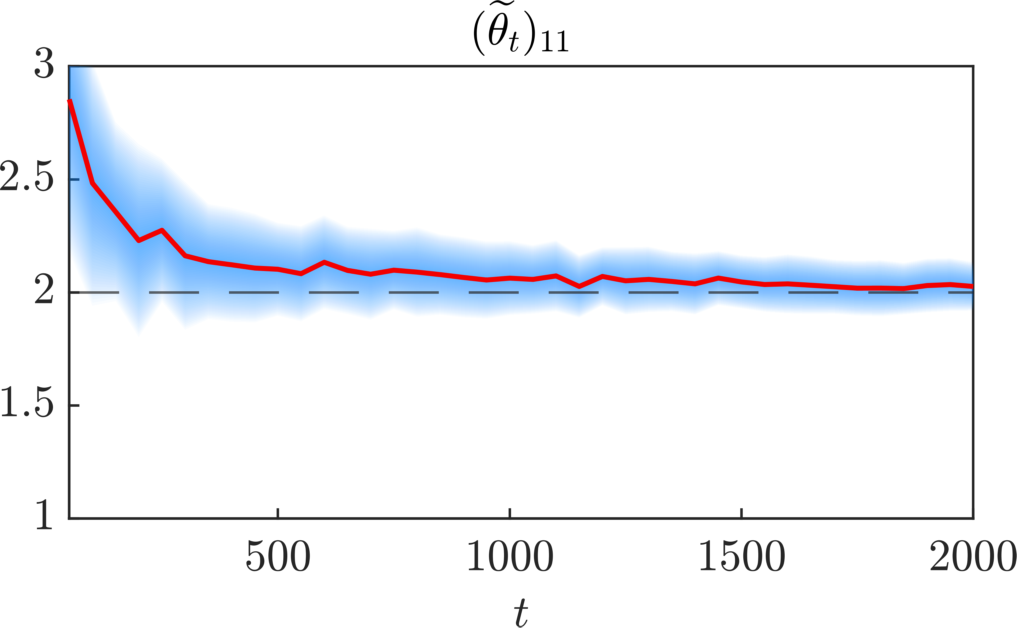}
\includegraphics[]{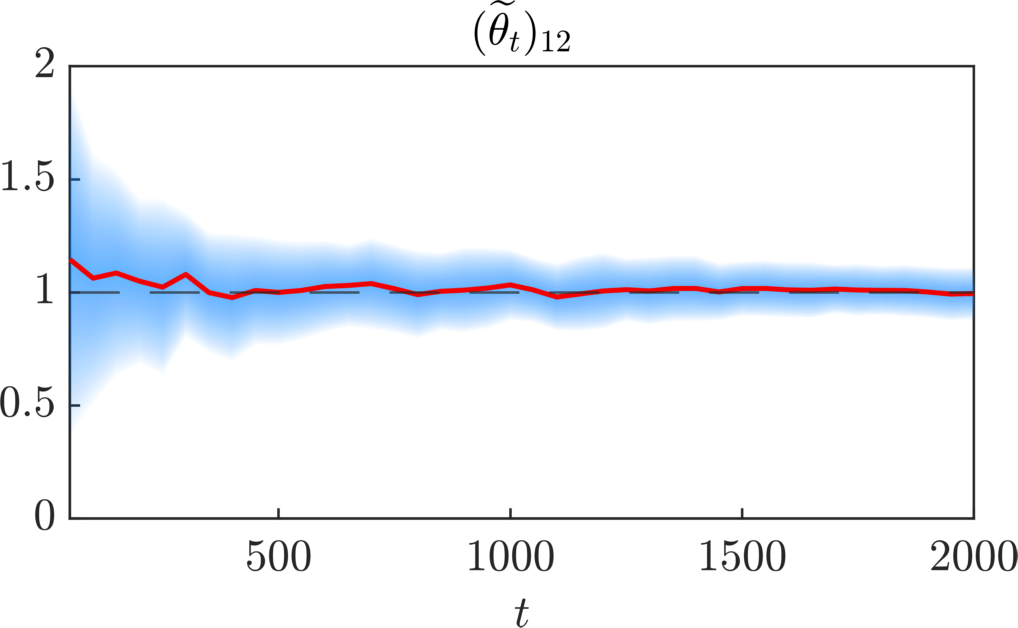}
\includegraphics[]{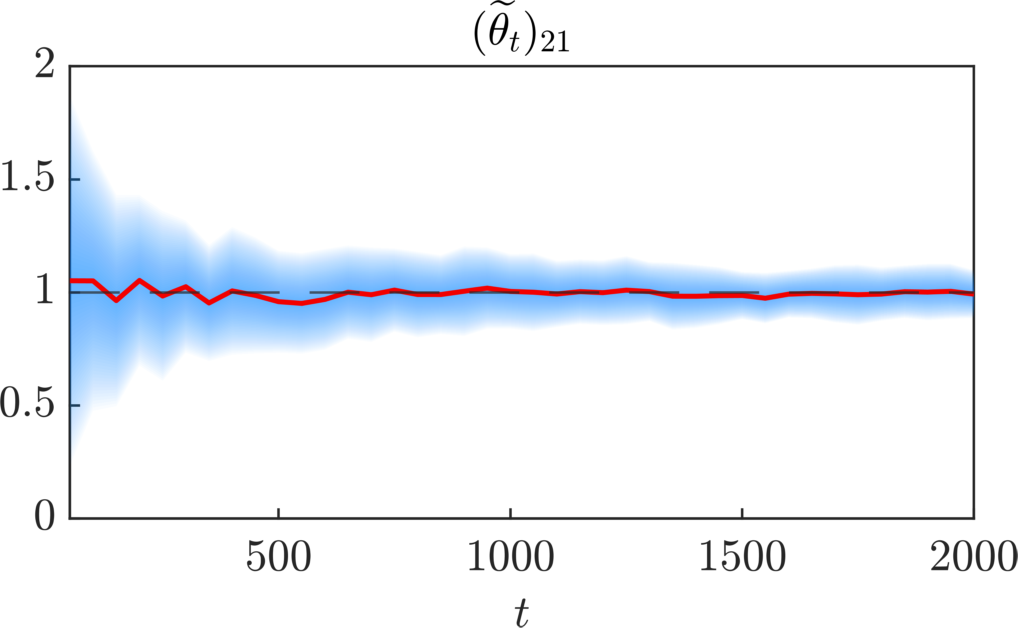} 
\includegraphics[]{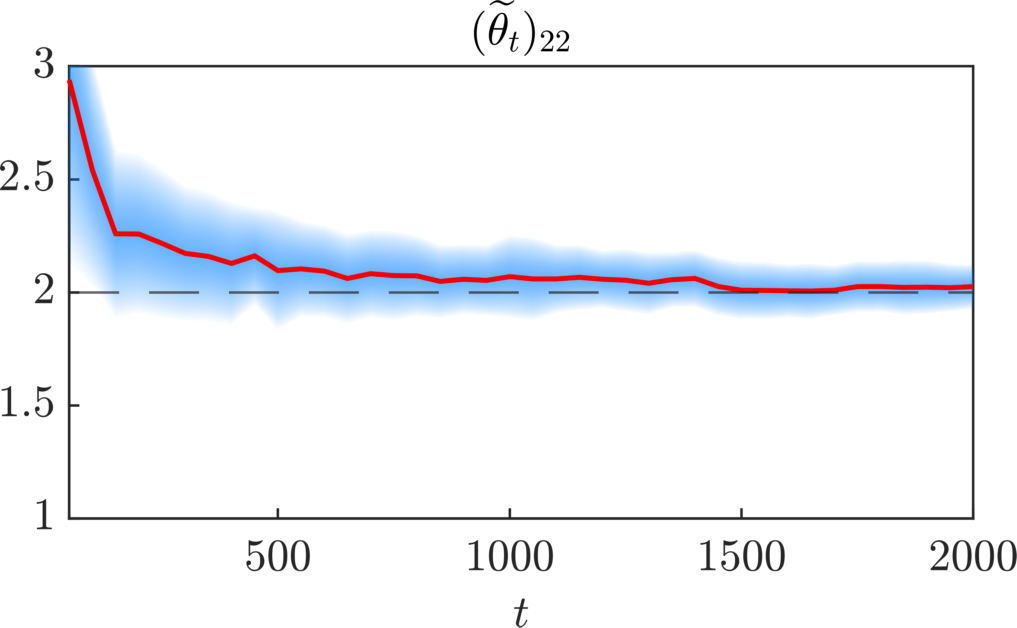}
\caption{Four components of the SGDCT estimator $\widetilde \theta_t$ with data from the limit equation in presence of additive white noise.}
\label{fig:SGDCT_ok}
\end{figure}

\begin{figure}
\centering
\includegraphics[]{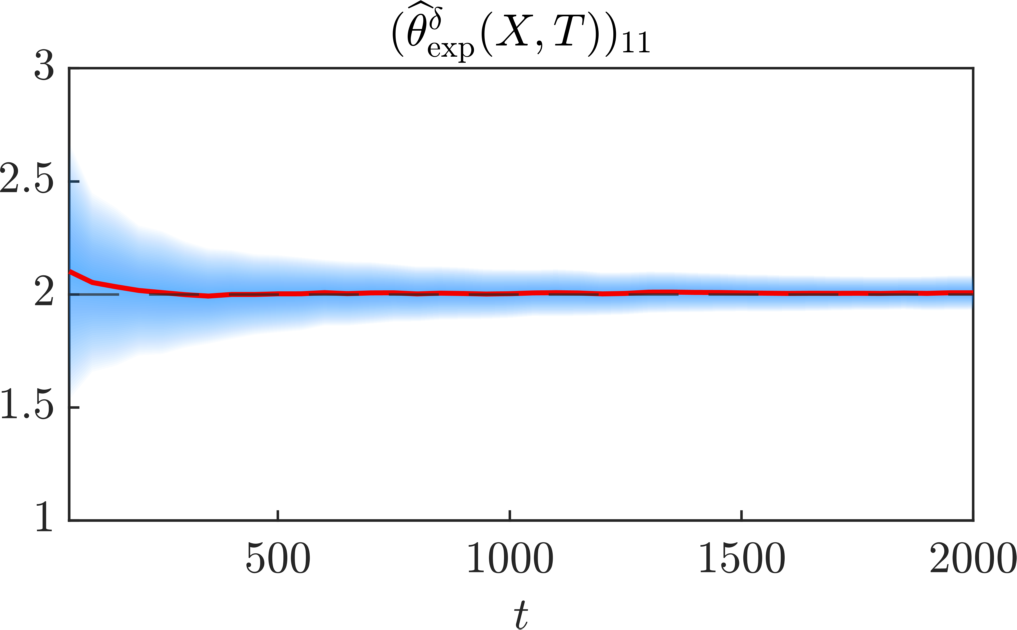}
\includegraphics[]{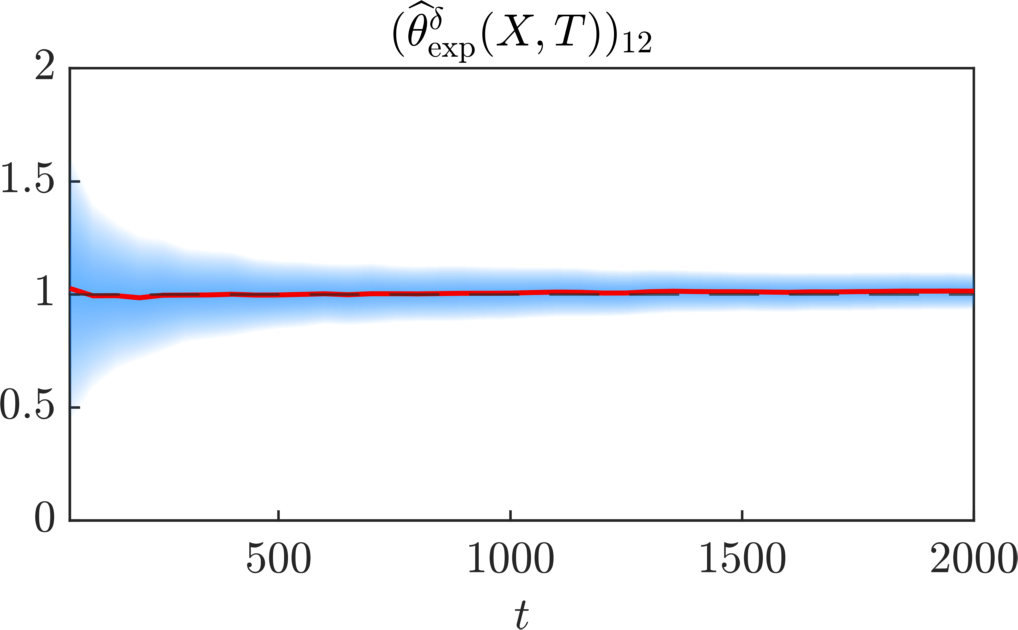}
\includegraphics[]{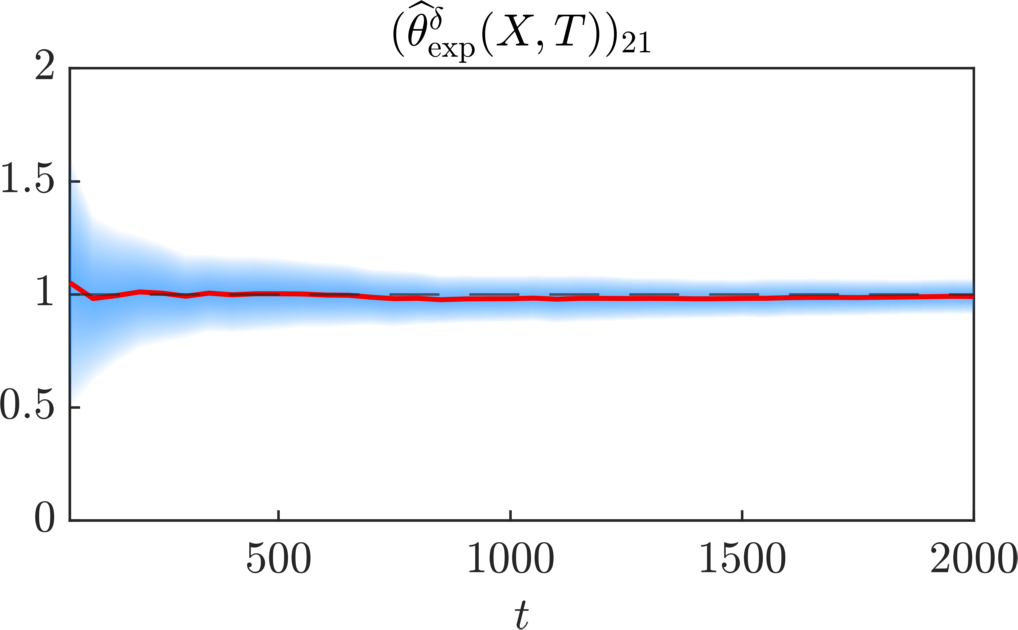} 
\includegraphics[]{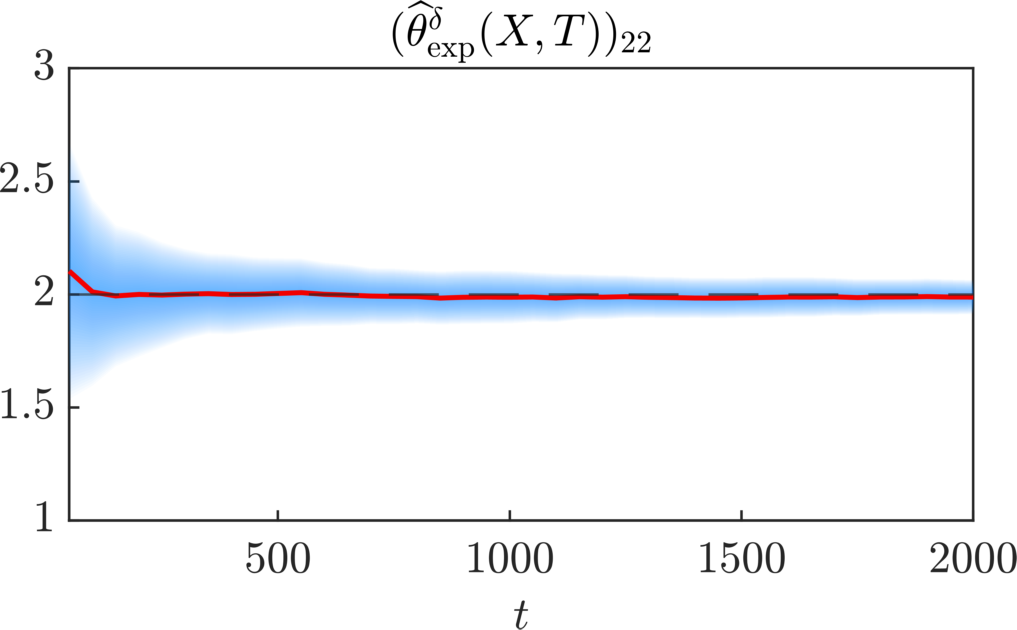}
\caption{Four components of the MLE estimator $\widehat \theta_{\mathrm{exp}}^\delta(X,T)$ with filtered data from the limit equation in presence of additive white noise.}
\label{fig:exp_ok}
\end{figure}

\begin{figure}
\centering
\includegraphics[]{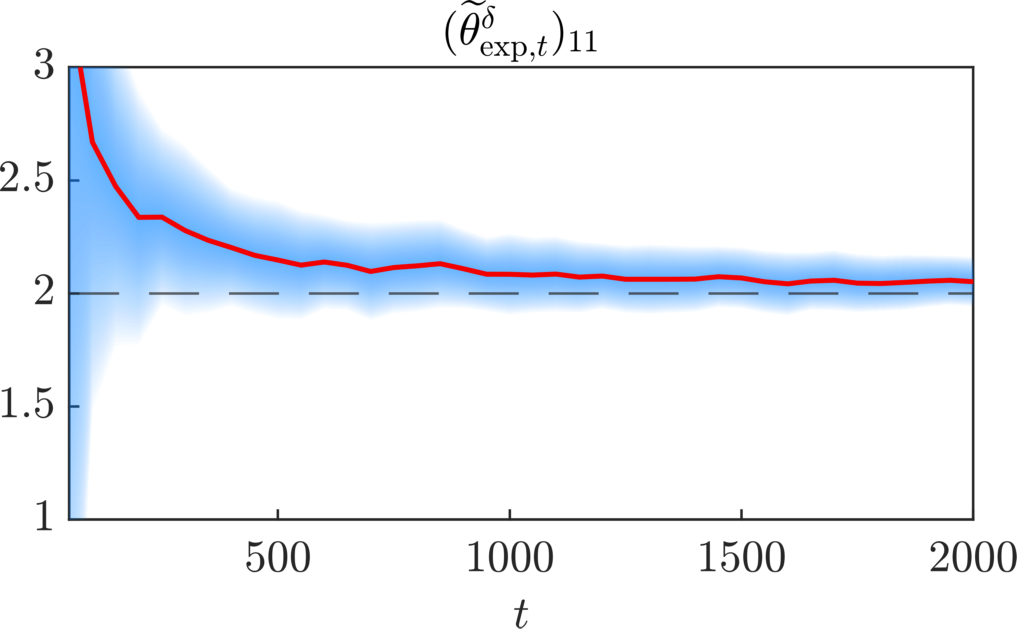}
\includegraphics[]{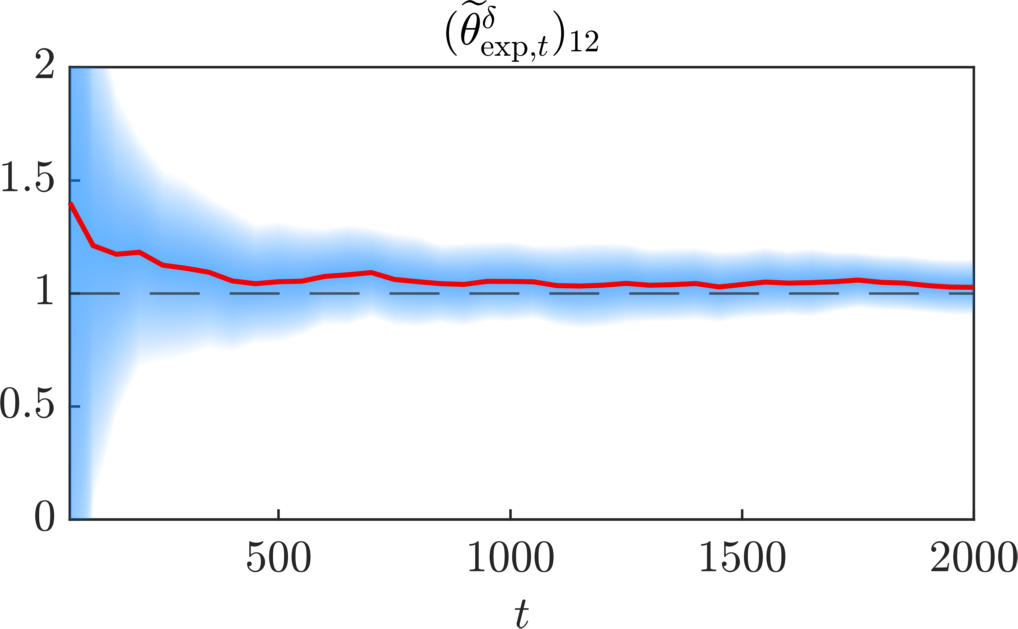}
\includegraphics[]{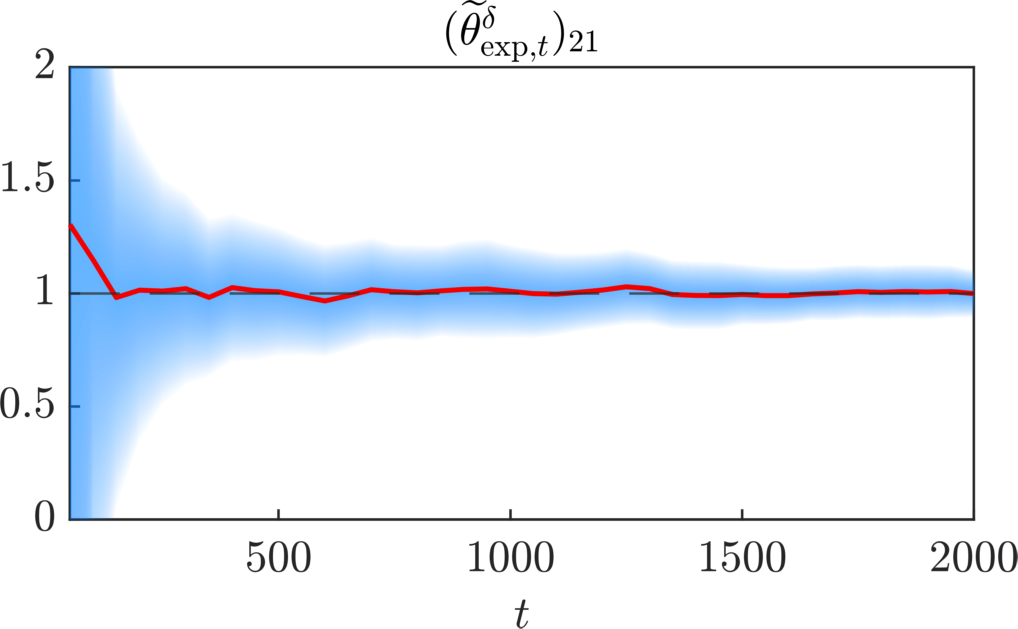} 
\includegraphics[]{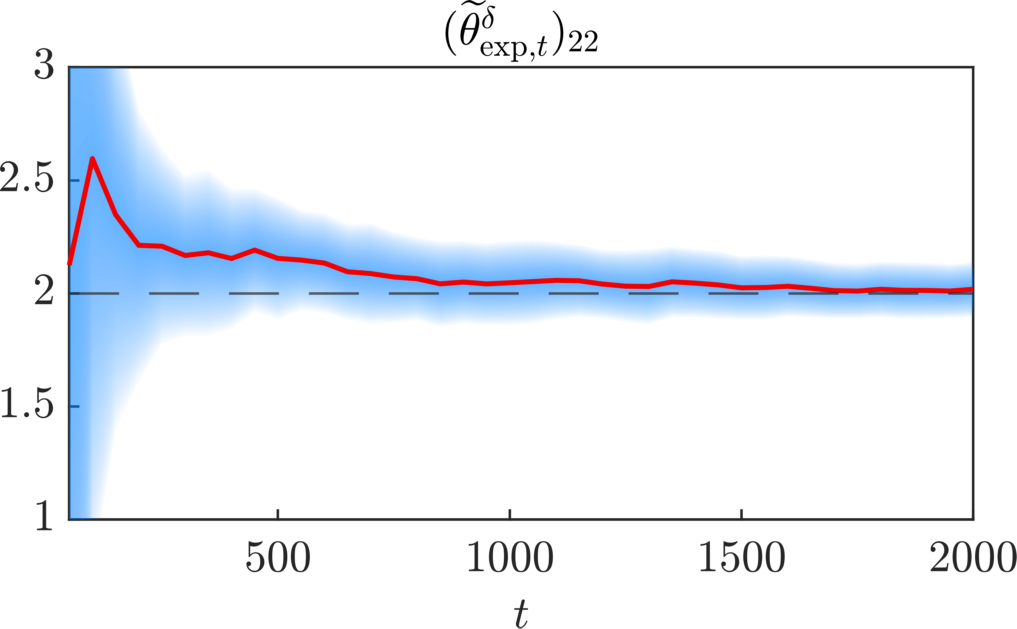}
\caption{Four components of the SGDCT estimator $\widetilde \theta_{\mathrm{exp},t}^\delta$ with filtered data from the limit equation in presence of additive white noise.}
\label{fig:SGDCT_exp_ok}
\end{figure}

\begin{figure}
\centering
\includegraphics[]{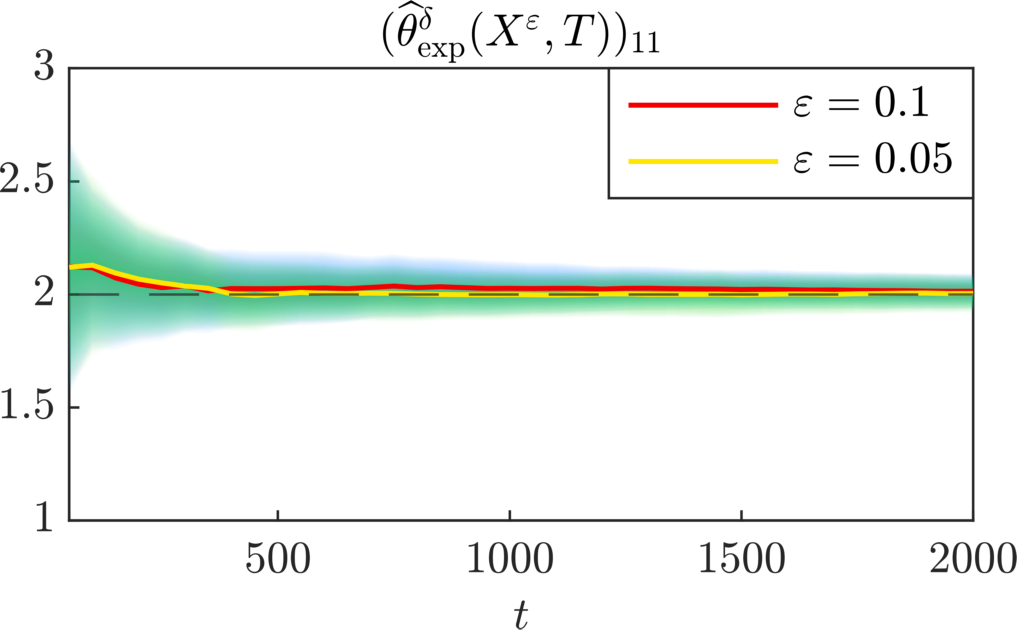}
\includegraphics[]{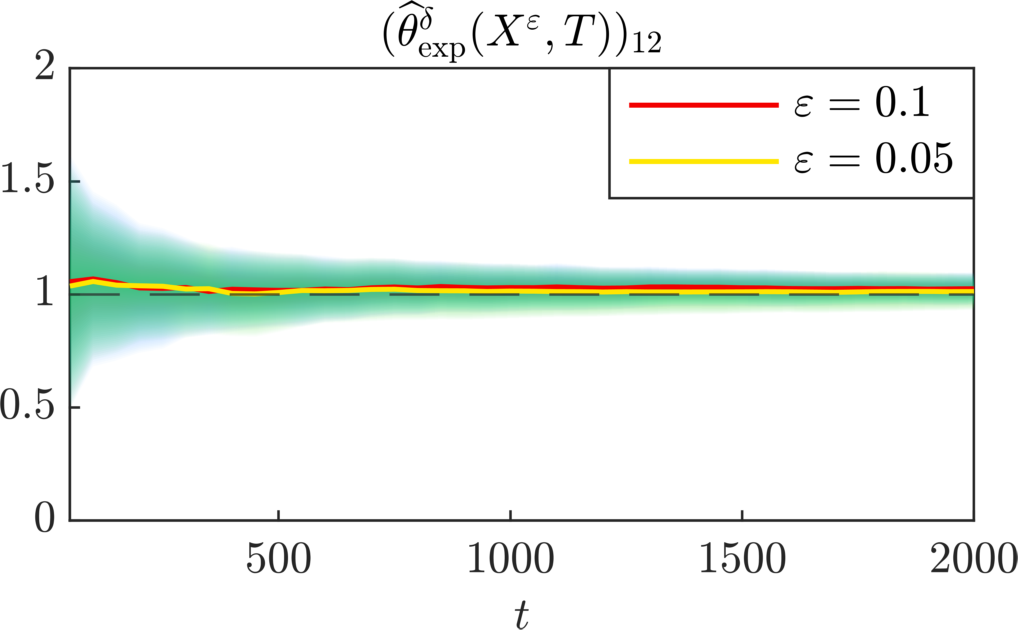}
\includegraphics[]{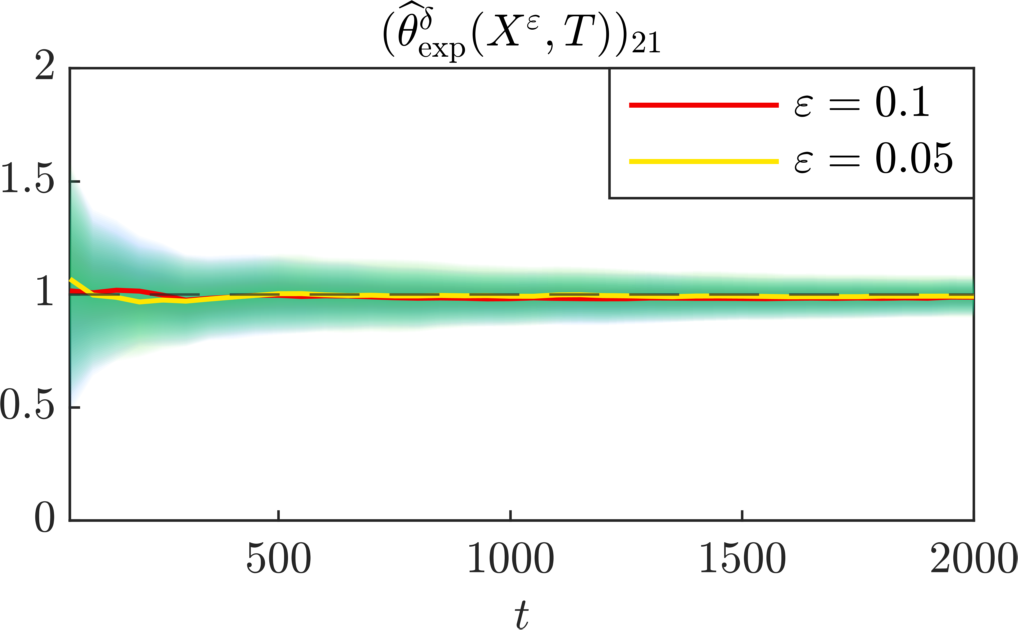} 
\includegraphics[]{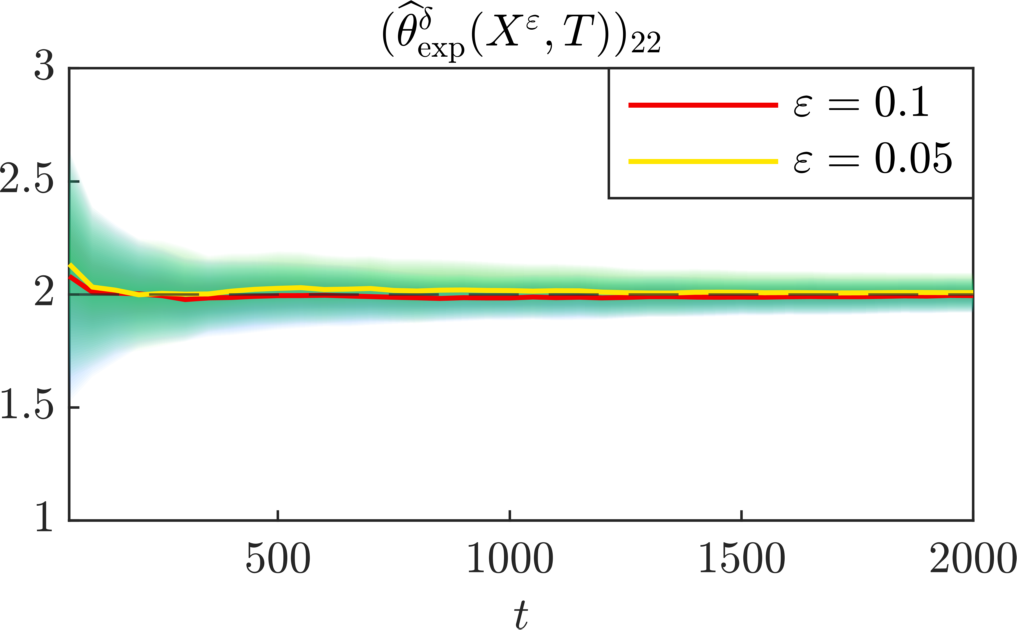}
\caption{Four components of the MLE estimator $\widehat \theta_{\mathrm{exp}}^\delta(X^\epl,T)$ with filtered data from the original equation in presence of additive colored noise.}
\label{fig:exp_ok_eps}
\end{figure}

\begin{figure}
\centering
\includegraphics[]{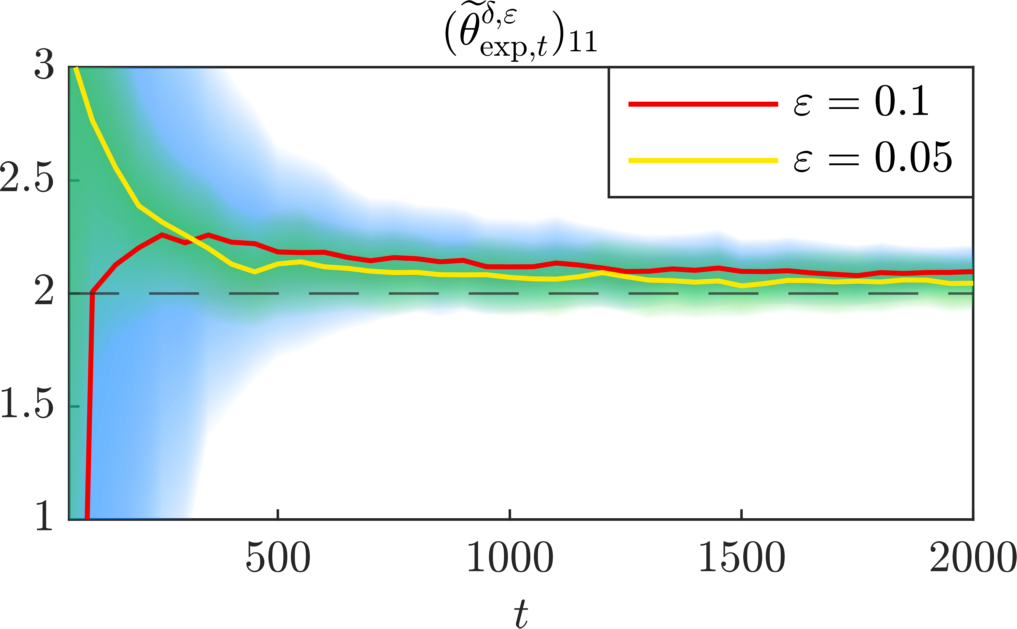}
\includegraphics[]{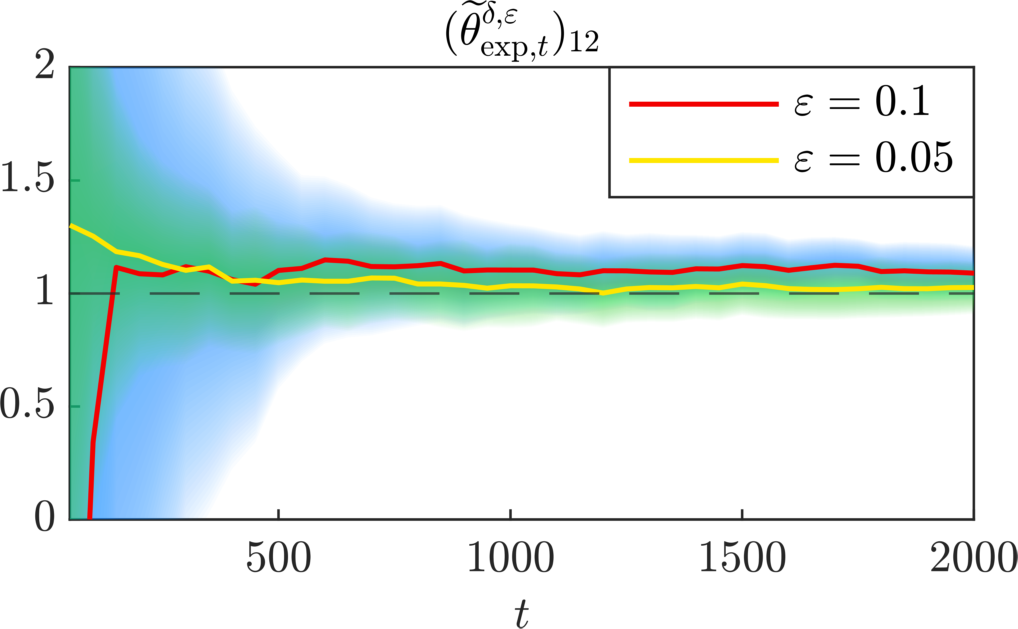}
\includegraphics[]{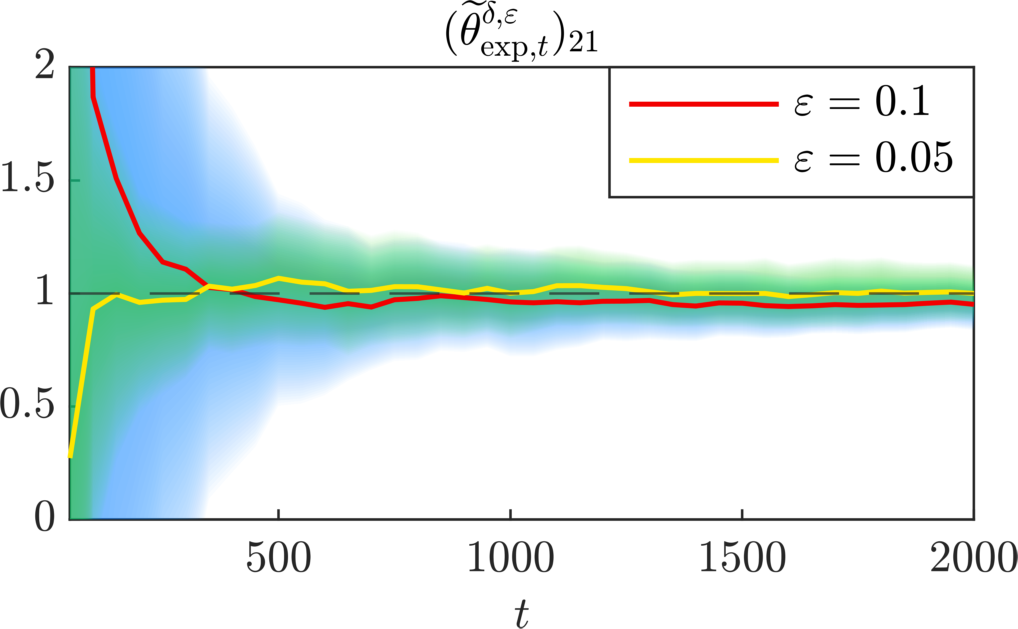} 
\includegraphics[]{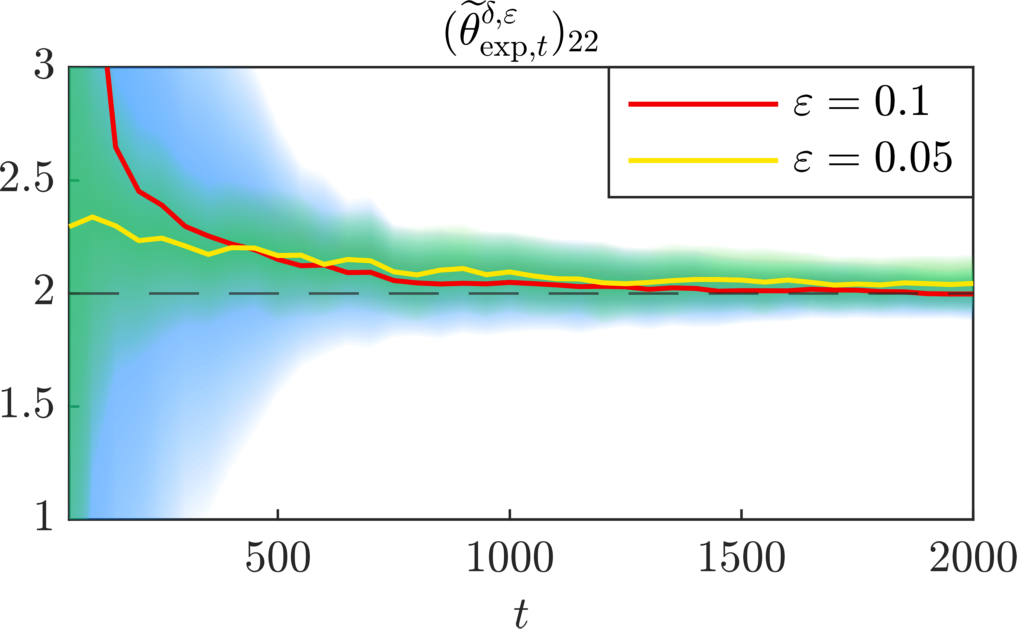}
\caption{Four components of the SGDCT estimator $\widetilde \theta_{\mathrm{exp},t}^{\delta,\epl}$ with filtered data from the original equation in presence of additive colored noise.}
\label{fig:SGDCT_exp_ok_eps}
\end{figure}

Let us now focus on the MLE and SGDCT estimators $\widehat \theta(X,T)$, $\widehat \theta_{\mathrm{exp}}^\delta(X,T)$, $\widehat \theta_{\mathrm{exp}}^\delta(X^\epl,T)$ and $\widetilde \theta_t$, $\widetilde \theta_{\mathrm{exp},t}^\delta$, $\widetilde \theta_{\mathrm{exp},t}^{\delta,\epl}$, for which we rigorously showed their asymptotic unbiasedness in \cref{thm:MLE_ok,thm:exp_ok,thm:exp_ok_eps,thm:SGDCT_ok,thm:SGDCT_exp_ok,thm:SGDCT_exp_ok_eps}. We consider the two-dimensional setting in \cref{ex:multiD}, and we set $\alpha = \gamma = \eta = 1$, and the drift function to be $h(x) = -x$. Moreover, we choose the diffusion function $g(x) = G = I$, where $I$ denotes the identity matrix, so that the noise is additive. The remaining parameters are the coefficients $a=100$ and $b=0.1$ in the learning rate $\xi_t$ of the SGDCT estimators, the final time $T = 2000$, and the unknown matrix $\theta \in \R^{2\times2}$ which is given by
\begin{equation}
\theta =
\begin{pmatrix}
2 & 1 \\
1 & 2
\end{pmatrix}.
\end{equation}
We verify in \cref{fig:MLE_ok,fig:SGDCT_ok,fig:exp_ok,fig:SGDCT_exp_ok,fig:exp_ok_eps,fig:SGDCT_exp_ok_eps} that all the estimators are able to correctly infer the true drift coefficient $\theta$. In particular, we observe that the estimated values tend to stabilize around the correct parameters when time increases, and that the bias and the standard deviation decrease. We remark that, even if the variance of the SGDCT estimators is larger than the variance of the MLE estimators when only few data are collected, the variance of the two approaches is comparable if the final time is sufficiently large. Finally, we also notice that decreasing the values of $\epl$, and therefore considering colored noise which is closer to white noise, provides better approximations both in terms of bias and especially variance of the estimators at finite time.

\subsection{Lévy area correction} \label{sec:num_Levy}

\begin{figure}
\centering
\includegraphics[]{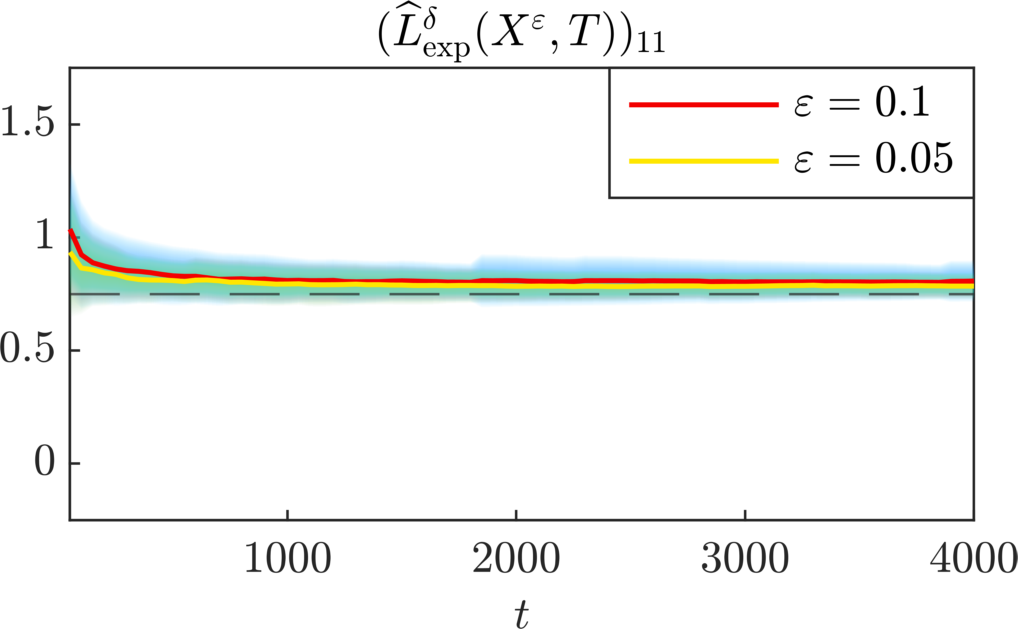}
\includegraphics[]{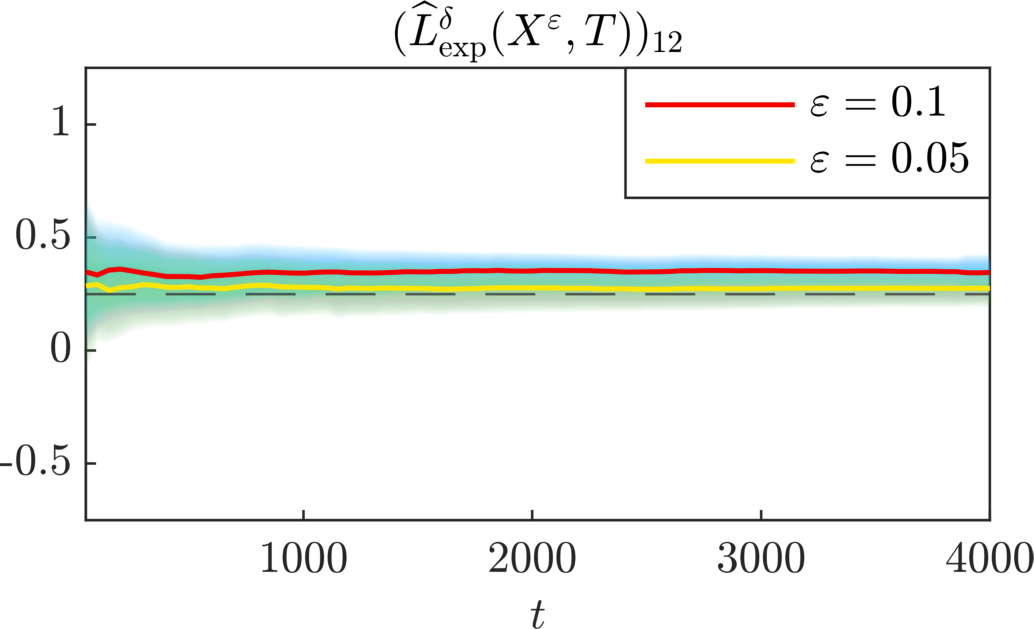}
\includegraphics[]{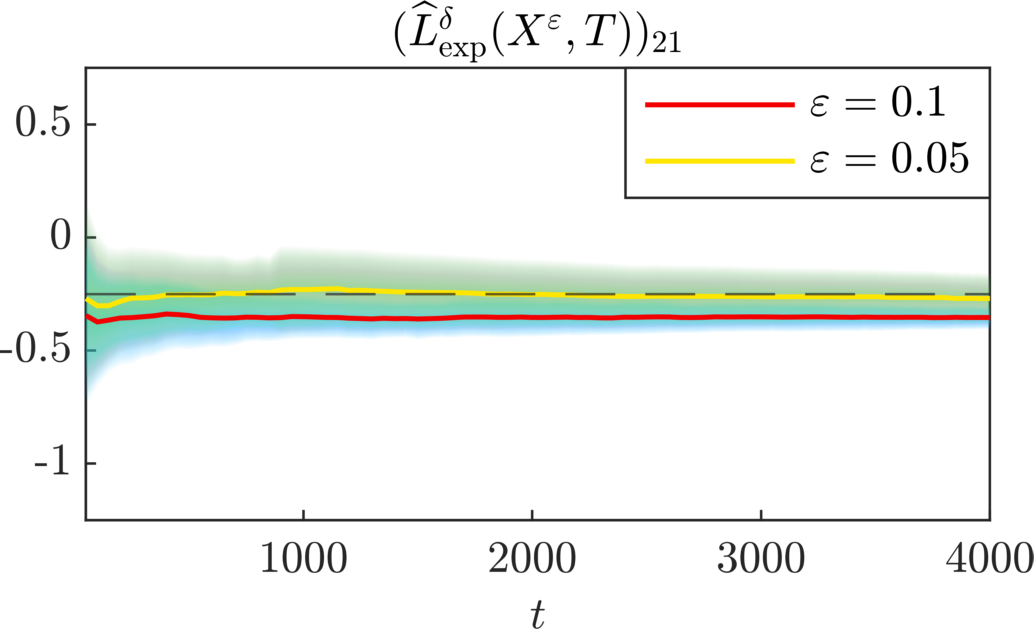} 
\includegraphics[]{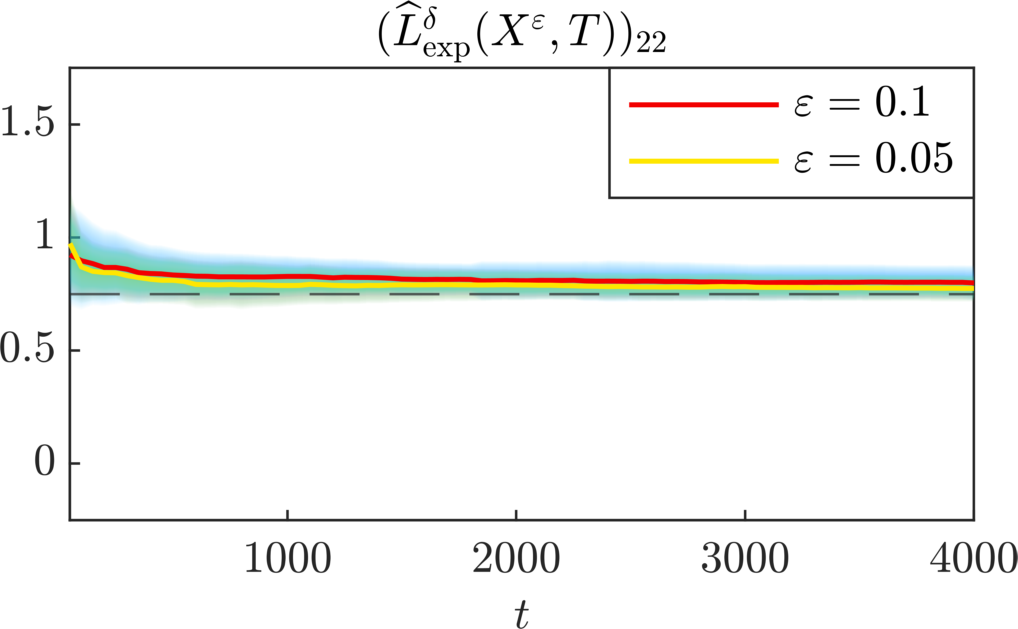}
\caption{Four components of the MLE estimator $\widehat L_{\mathrm{exp}}^\delta(X^\epl,T)$ with filtered data from the original equation driven by colored noise in presence of the Lévy area correction.}
\label{fig:Levy}
\end{figure}

\begin{figure}
\centering
\includegraphics[]{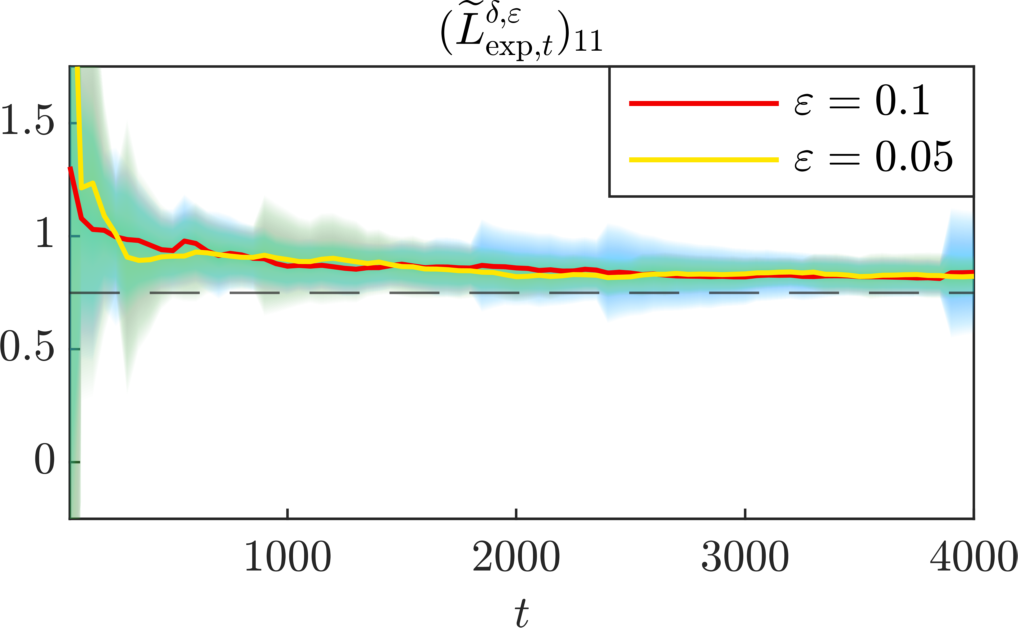}
\includegraphics[]{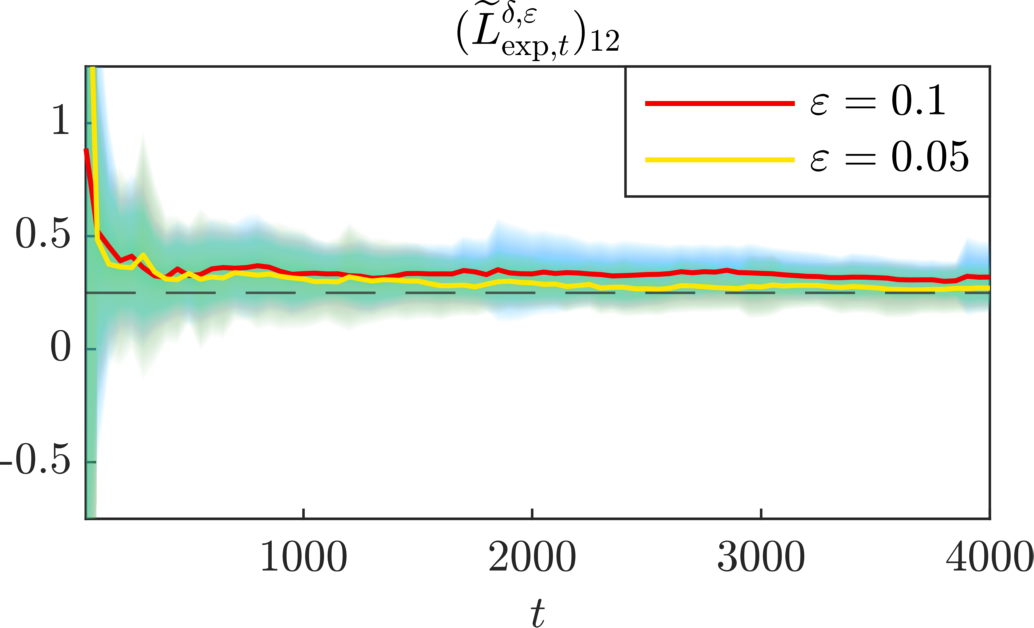}
\includegraphics[]{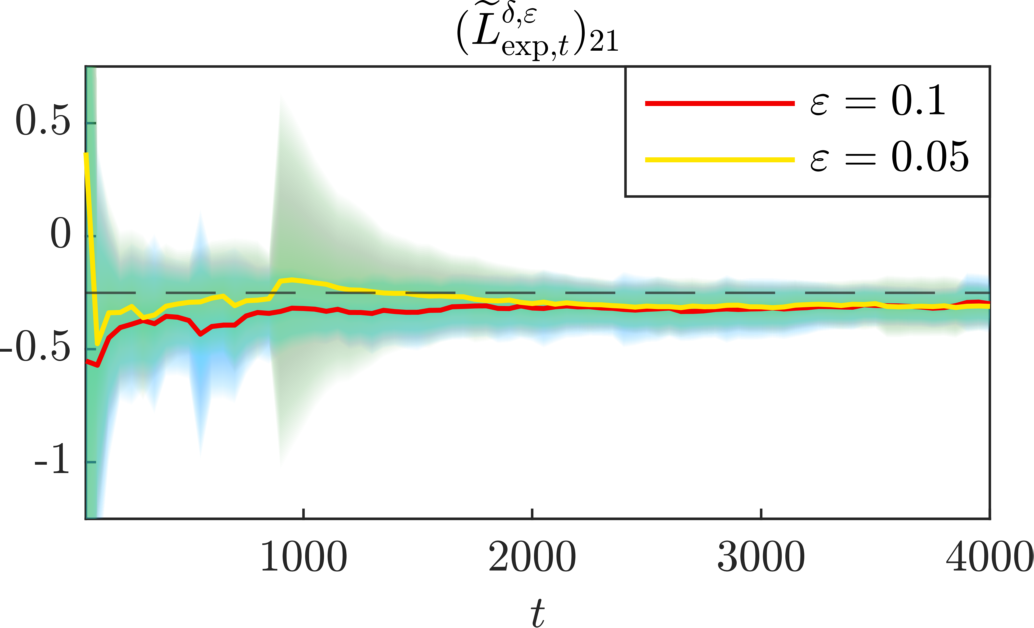} 
\includegraphics[]{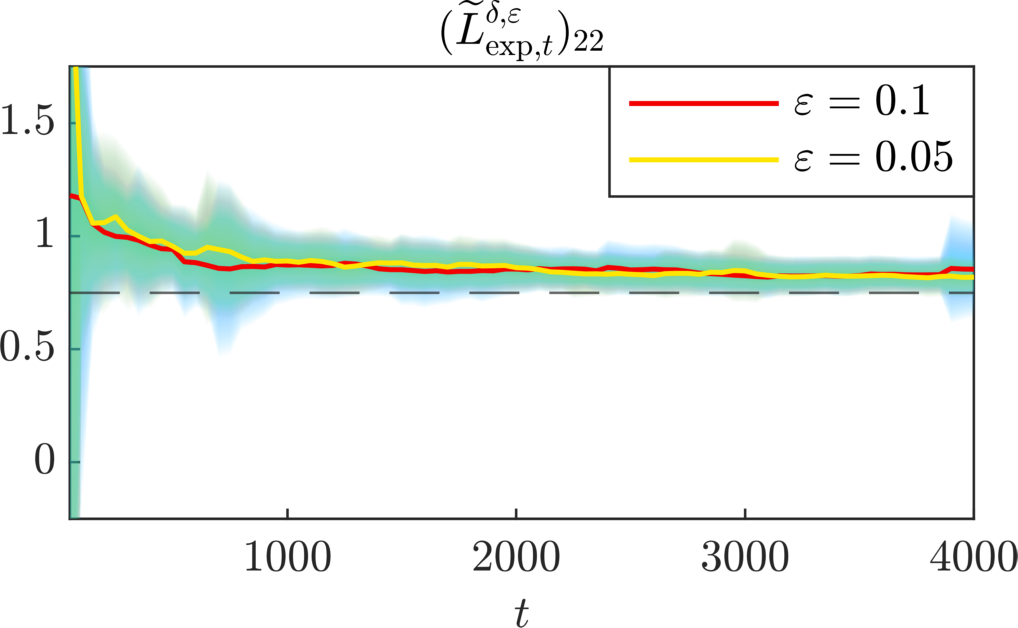}
\caption{Four components of the SGDCT estimator $\widetilde L_{\mathrm{exp},t}^{\delta,\epl}$ with filtered data from the original equation driven by colored noise in presence of the Lévy area correction.}
\label{fig:SGDCT_Levy}
\end{figure}

We consider here the framework of \cref{sec:Levy} and study the case of multiplicative noise that yields a Lévy area correction in the limit equation. We set the final time to $T = 4000$, the coefficients in the equations to $\theta = \alpha = \gamma = \kappa = \beta = 1$, and the parameters $a=10$ and $b=0.1$ in the learning rate $\xi_t$ for the SGDCT estimator. The approximations provided by the estimators $\widehat L^\delta_{\mathrm{exp}}(X^\epl,T)$ and $\widetilde L^{\delta,\epl}_{\mathrm{exp},t}$ are shown in \cref{fig:Levy,fig:SGDCT_Levy}. The results are in line with our findings from the previous test cases. We only remark that, given the more complex stochastic models and consequently the more challenging inference problem, the final estimations are slightly worse. Nevertheless, if the final time and the parameter $\epl$ are sufficiently large and small, respectively, both the MLE and the SGDCT estimator are able to provide a reliable approximation of the unknown drift coefficient $L$. Finally, we notice that in this example the MLE seems to be more robust the SGDCT estimator, which also depends on the learning rate. On the other hand, we believe that the SGDCT estimator with filtered data can be successfully employed in the more general setting when the drift function does not depend linearly on the parameter, and therefore the MLE does not have a closed-form expression. 

\subsection{Asymptotic normality} \label{sec:CLT}

\begin{figure}
\centering
\includegraphics{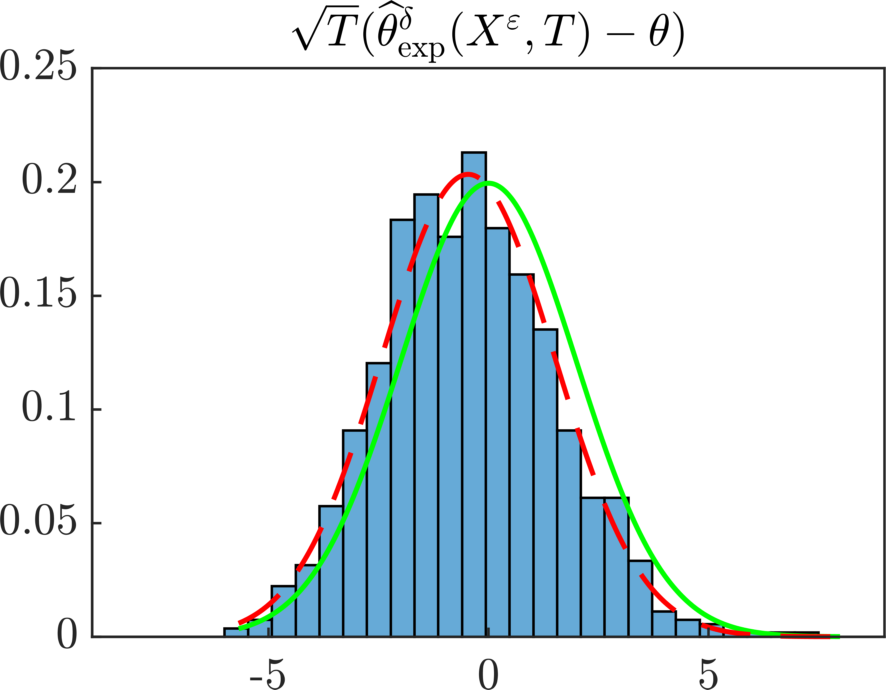} \hspace{1cm}
\includegraphics{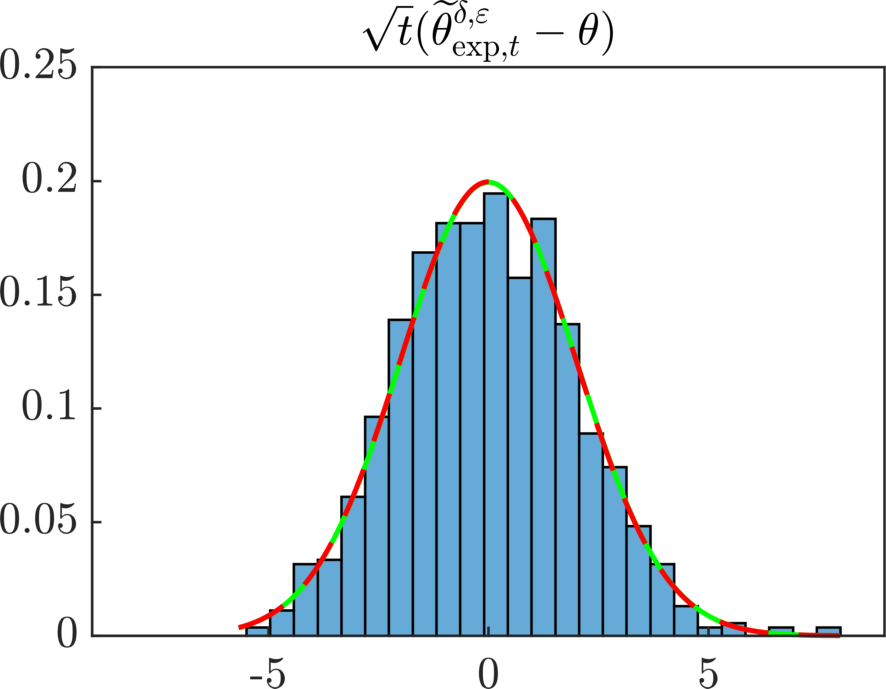}
\caption{Histogram for the numerical study of the central limit theorem for the estimators $\widehat \theta^\delta_{\mathrm{exp}}(X^\epl,T)$ (left) and $\widetilde \theta^{\delta,\epl}_{\mathrm{exp},t}$ (right) in the linear case. The red dashed lines represent the fitted Gaussian distributions, and the green lines represent the asymptotic distribution predicted by the analysis.}
\label{fig:CLT}
\end{figure}

A natural question is whether, in addition to the asymptotic unbiasedness of the estimators, we can also prove a central limit theorem (CLT) and therefore asymptotic normality. This was done for single-scale SDEs for the SGDCT in \cite{SiS20} and it is a classical result for the MLE (see, e.g., \cite{RaR81}). We mention, in particular, that the results from \cite{SiS20} do not directly apply to our problem, since we need to study the asymptotic behaviour of our estimators with respect to both time and $\epl$. We also remark that ideas to prove a CLT for the MLE for multiscale diffusions are sketched in \cite[Section 4.5]{Zan22}. This is a very interesting and challenging problem that we will leave for future study. Nevertheless, in this section we present some formal calculations and numerical experiments that shed some light into how time and the parameter $\epl$ measuring the correlation of the colored noise should scale in order to be able to prove a CLT for the modified SGDCT and MLE estimators. We mention that a similar asymptotic analysis with respect to two parameters is presented in \cite{PaZ22a} for the eigenfunction estimator for interacting particle systems, where one needs to study carefully the asymptotic normality of the estimator with respect to the number of particles and the number of observations of a single particle.

Let us focus for simplicity on the setting of \cref{ex:OU} and set the parameters $\theta = G = A = \sigma = 1$. Let us then consider the estimators $\widehat \theta^\delta_{\mathrm{exp}}(X^\epl,T)$ and $\widetilde \theta^{\delta,\epl}_{\mathrm{exp},t}$ with filtered data. We aim to analyze the convergence of the quantity
\begin{equation}
\sqrt T \left( \widehat \theta^\delta_{\mathrm{exp}}(X^\epl,T) - \theta \right)
\end{equation}
for the MLE and similarly for the SGDCT estimator. Consider the decomposition
\begin{equation} \label{eq:decomposition_CLT_MLE}
\begin{aligned} 
\sqrt T \left( \widehat \theta^\delta_{\mathrm{exp}}(X^\epl,T) - \theta \right) &= \sqrt T \left( \widehat \theta^\delta_{\mathrm{exp}}(X^\epl,T) - \theta - \frac{\E[Z^\epl \frac{Y^\epl}{\epl}]}{\E[Z^\epl X^\epl]} \right) \\
&\quad + \sqrt T \frac{\E[Z^\epl \frac{Y^\epl}{\epl}]}{\E[Z^\epl X^\epl]}, \\
&\eqdef \sqrt T \mathcal Q_T^\epl + \sqrt T \mathcal E^\epl,
\end{aligned}
\end{equation}
and notice that $\mathcal Q_T^\epl \to 0$ as $T \to \infty$ and $\mathcal E^\epl \to 0$ as $\epl \to 0$ due to the proof of \cref{thm:exp_ok_eps}. We therefore need to set $T = \epl^{-\lambda}$ with $\lambda > 0$ and study the two limits simultaneously. In particular, we expect the first part to give the main contribution, i.e., the limit Gaussian distribution, and the second part to vanish. We remark that in this setting all the processes $X_t^\epl, Y_t^\epl, Z_t^\epl$ and $X_t, Z_t$ are Gaussian, so we can perform the computations explicitly. Hence, we have
\begin{equation}
\E \left[Z^\epl \frac{Y^\epl}{\epl} \right] = \lim_{t\to\infty} \frac1{\delta \epl^2} \int_0^t \int_0^s e^{-\frac{t-s}\delta} e^{-(s-r)} \E[Y_r^\epl Y_t^\epl] \dd r \dd s = \frac{\epl^2}{2(1 + \epl^2)(\delta + \epl^2)},
\end{equation}
and
\begin{equation} \label{eq:EZX}
\lim_{\epl\to0} \E[Z^\epl X^\epl] = \E[Z X] = \lim_{t\to\infty} \frac1\delta \int_0^t e^{-\frac{t-s}\delta} \E[X_s X_t] \dd s = \frac1{2(1 + \delta)}.
\end{equation}
Therefore, if $\lambda < 4$ we deduce that
\begin{equation}
\lim_{\epl \to 0} \sqrt T \mathcal E^\epl = \lim_{\epl \to 0} \epl^{-\lambda/2} \mathcal E^\epl = \frac{(1 + \delta) \epl^{2 - \lambda/2}}{(1 + \epl^2)(\delta + \epl^2)} = 0.
\end{equation}
Let us now consider the term $\mathcal Q^\epl_T$, which, by the proof of \cref{thm:exp_ok_eps}, can be rewritten as
\begin{equation}
\mathcal Q_T^\epl = \frac{\int_0^T Z^\epl_t \frac{Y_t^\epl}{\epl} \d t}{\int_0^T Z^\epl_t X^\epl_t \d t} - \frac{\E[Z^\epl \frac{Y^\epl}{\epl}]}{\E[Z^\epl X^\epl]} = \frac1{\int_0^T Z^\epl_t X^\epl_t \d t} \int_0^T \left( Z^\epl_t \frac{Y_t^\epl}{\epl} - Z^\epl_t X^\epl_t \frac{\E[Z^\epl \frac{Y^\epl}{\epl}]}{\E[Z^\epl X^\epl]} \right) \dd t,
\end{equation}
and denote by $\omega^\epl = \omega^\epl(x,y,z)$ the solution of the Poisson problem for the generator of the joint system $(X_t^\epl,Y_t^\epl,Z_t^\epl)$ with right-hand side
\begin{equation}
\upsilon^\epl(x,y,z) = z \frac{y}{\epl} - z x \frac{\E[Z^\epl \frac{Y^\epl}{\epl}]}{\E[Z^\epl X^\epl]}.
\end{equation}
Applying the Itô formula to the process $\Xi_t^\epl = \omega^\epl(X_t^\epl, Y_t^\epl, Z_t^\epl)$ we get
\begin{equation}
\mathcal Q_T^\epl = \frac1{\int_0^T Z^\epl_t X^\epl_t \d t} \left( \omega^\epl(X_T^\epl, Y_T^\epl, Z_T^\epl) - \omega^\epl(X_0^\epl, Y_0^\epl, Z_0^\epl) - \frac1\epl \int_0^T \partial_y \omega^\epl(X_t^\epl, Y_t^\epl, Z_t^\epl) \dd W_t \right),
\end{equation}
which, due to the ergodic theorem and the martingale central limit theorem, formally implies
\begin{equation}
\lim_{\epl \to 0} \sqrt T \mathcal Q_T^\epl = \lim_{\epl \to 0} \epl^{-\lambda/2} \mathcal Q_{\epl^{-\lambda/2}}^\epl = \mathcal N(0, \Lambda), \qquad \text{in law},
\end{equation}
where by the limit \eqref{eq:EZX}
\begin{equation}
\Lambda = \lim_{\epl \to 0} \frac{\E \left[ (\partial_y \omega^\epl(X^\epl, Y^\epl, Z^\epl))^2 \right]}{\epl^2 \E[Z^\epl X^\epl]^2} = \lim_{\epl \to 0} \frac{4(1 + \delta)^2 \E \left[ (\partial_y \omega^\epl(X^\epl, Y^\epl, Z^\epl))^2 \right]}{\epl^2}.
\end{equation}
The next step consists in showing that this limit actually exists and is finite. We notice that the solution $\omega^\epsilon$ of the Poisson problem can be computed analytically and is given by
\begin{equation}
\omega^\epl(x,y,z) = \frac\epl{\delta + (1 + \delta)\epl^2} \left( \frac\epl2 x^2 + \delta\epl xz - \delta yz \right),
\end{equation}
which implies
\begin{equation}
\partial_y \omega(x,y,z) = - \frac{\delta\epl}{\delta + (1 + \delta)\epl^2}z.
\end{equation}
Therefore, since 
\begin{equation} \label{eq:EZ2}
\E[Z^2] = \lim_{t\to\infty} \frac1{\delta^2} \int_0^t \int_0^t e^{-\frac{t-r}\delta} e^{-\frac{t-s}\delta} \E[X_r X_s] \dd r \dd s = \frac1{2(1 + \delta)},
\end{equation}
we obtain
\begin{equation} 
\Lambda = \lim_{\epl\to0} \frac{4(1 + \delta)^2 \delta^2}{(\delta + (1 + \delta)\epl^2)^2} \E[(Z^\epl)^2] = 4(1 + \delta)^2 \E[Z^2] = 2(1 + \delta),
\end{equation}
which is the asymptotic variance of the MLE. This result implies that the filtering width should be chosen as small as possible in the asymptotic regime. However, in concrete applications with $\epl$ small, but finite, $\delta$ must be sufficiently large with respect to $\epl$ in order for the filtering procedure to be effective. 

On the other hand, a different argument should be used to derive the limit distribution for the SGDCT estimator. Even if the term $\mathcal E^\epl$ can be studied analogously, the quantity $\mathcal Q^\epl_T$ should be analyzed separately. In fact, we believe that this analysis can be inspired by \cite{SiS20}, but must differ due to the absence of a clear objective function in our context. In the simplified setting of this section we can take advantage of the explicit expression of the SGDCT estimator given in \eqref{eq:SGDCT_1D_explicit}, which reads
\begin{equation}
\widetilde \theta_{\mathrm{exp},t}^{\delta,\epl} = \theta + (\theta_0 - \theta) e^{-\int_0^t \xi_r X_r^\epl Z_r^\epl \dd r} + \int_0^t \xi_s e^{-\int_s^t \xi_r X_r^\epl Z_r^\epl \dd r} Z_s^\epl \frac{Y_s^\epl}{\epl} \dd s.
\end{equation}
We have the following decomposition
\begin{equation}
\sqrt t \left( \widetilde \theta_{\mathrm{exp},t}^{\delta,\epl} - \theta \right) = \sqrt t \left( \mathcal Q_t^{\epl,1} + \mathcal Q_t^{\epl,2} + \mathcal Q_t^{\epl,3} + \mathcal Q_t^{\epl,4} + \mathcal Q_t^{\epl,5} + \mathcal E^\epl \right),
\end{equation}
where $\mathcal E^\epl$ is defined as in \eqref{eq:decomposition_CLT_MLE} and
\begin{equation}
\begin{aligned}
\mathcal Q_t^{\epl,1} &= (\theta_0 - \theta) \left( e^{-\int_0^t \xi_r X_r^\epl Z_r^\epl \dd r} - e^{-\int_0^t \xi_r \dd r \E [X^\epl Z^\epl]} \right), \\
\mathcal Q_t^{\epl,2} &= (\theta_0 - \theta) e^{-\int_0^t \xi_r \dd r \E [X^\epl Z^\epl]}, \\
\mathcal Q_t^{\epl,3} &= \int_0^t \xi_s e^{-\int_s^t \xi_r X_r^\epl Z_r^\epl \dd r} Z_s^\epl \frac{Y_s^\epl}{\epl} \dd s - \int_0^t \xi_s e^{-\int_s^t \xi_r \dd r \E [X^\epl Z^\epl]} Z_s^\epl \frac{Y_s^\epl}{\epl} \dd s, \\
\mathcal Q_t^{\epl,4} &= \int_0^t \xi_s e^{-\int_s^t \xi_r \dd r \E [X^\epl Z^\epl]} Z_s^\epl \frac{Y_s^\epl}{\epl} \dd s - \int_0^t \xi_s e^{-\int_s^t \xi_r \dd r \E [X^\epl Z^\epl]} \dd s \E \left[ Z^\epl \frac{Y^\epl}{\epl} \right], \\
\mathcal Q_t^{\epl,5} &= \int_0^t \xi_s e^{-\int_s^t \xi_r \dd r \E [X^\epl Z^\epl]} \dd s \E \left[ Z^\epl \frac{Y^\epl}{\epl} \right] - \frac{\E[Z^\epl \frac{Y^\epl}{\epl}]}{\E[Z^\epl X^\epl]}.
\end{aligned}
\end{equation}
Proceeding analogously as for the MLE, we set $t = \epl^{-\lambda}$ with $0 < \lambda < 4$ and deduce that $\sqrt t \mathcal E^\epl \to 0$ as $\epl \to 0$. We then expect the terms $\mathcal Q_t^{\epl,1}, \mathcal Q_t^{\epl,2}, \mathcal Q_t^{\epl,3}, \mathcal Q_t^{\epl,5}$ to vanish faster than $1/\sqrt t$, and therefore not to contribute to the limit distribution, and we focus on $\mathcal Q_t^{\epl,4}$. First, notice that
\begin{equation}
\xi_s e^{-\int_s^t \xi_r \dd r \E [X^\epl Z^\epl]} = \frac{a (b + s)^{a \E [X^\epl Z^\epl] - 1}}{(b + t)^{a\E [X^\epl Z^\epl]}},
\end{equation}
which implies
\begin{equation}
\mathcal Q_t^{\epl,4} = \frac{a}{(b + t)^{a\E [X^\epl Z^\epl]}} \int_0^t (b + s)^{a \E [X^\epl Z^\epl] - 1} \left( Z_s^\epl \frac{Y_s^\epl}{\epl} - \E \left[ Z^\epl \frac{Y^\epl}{\epl} \right] \right) \dd s.
\end{equation}
Let us now denote by $\widetilde \omega^\epl = \widetilde \omega^\epl(x,y,z)$ the solution of the Poisson problem for the generator of the joint system $(X_t^\epl,Y_t^\epl,Z_t^\epl)$ with right-hand side
\begin{equation}
\widetilde \upsilon^\epl(x,y,z) = z \frac{y}{\epl} - \E \left[ Z^\epl \frac{Y^\epl}{\epl} \right],
\end{equation}
which is given by
\begin{equation} \label{eq:solution_Poisson_tilde_CLT}
\widetilde \omega^\epl(x,y,z) = - \frac{\epl}{\delta + \epl^2} \left( \frac{\epl^3}{2(1 + \epl^2)} y^2 + \frac{\epl^2}{1 + \epl^2} xy + \delta yz \right).
\end{equation}
Applying the Itô formula to the process $\widetilde \Xi_s^\epl = (b + s)^{a \E [X^\epl Z^\epl] - 1} \widetilde \omega^\epl(X_s^\epl, Y_s^\epl, Z_s^\epl)$ we get
\begin{equation}
\begin{aligned}
\mathcal Q_t^{\epl,4} &= - \frac{a}{\epl (b + t)^{a \E [X^\epl Z^\epl]}} \int_0^t (b + s)^{a \E [X^\epl Z^\epl] - 1} \partial_y \widetilde \omega^\epl(X_s^\epl, Y_s^\epl, Z_s^\epl) \dd W_s \\
&\quad - \frac{a(a \E [X^\epl Z^\epl] - 1)}{(b + t)^{a \E [X^\epl Z^\epl]}} \int_0^t (b + s)^{a \E [X^\epl Z^\epl] - 2} \widetilde \omega^\epl(X_s^\epl, Y_s^\epl, Z_s^\epl) \dd s \\
&\quad + \frac{a}{(b + t)^{a\E [X^\epl Z^\epl]}} \left( (b + t)^{a \E [X^\epl Z^\epl] - 1} \widetilde \omega^\epl(X_t^\epl, Y_t^\epl, Z_t^\epl) - b^{a \E [X^\epl Z^\epl] - 1} \widetilde \omega^\epl(X_0^\epl, Y_0^\epl, Z_0^\epl) \right),
\end{aligned}
\end{equation}
which, due to the ergodic theorem and the martingale central limit theorem, formally implies
\begin{equation}
\lim_{\epl \to 0} \sqrt t \mathcal Q_t^{\epl,4} = \lim_{\epl \to 0} \epl^{-\lambda/2} \mathcal Q_{\epl^{-\lambda/2}}^{\epl,4} = \mathcal N(0, \widetilde \Lambda), \qquad \text{in law},
\end{equation}
where by the limit \eqref{eq:EZX}
\begin{equation}
\begin{aligned}
\widetilde \Lambda &= \lim_{\epl \to 0} \lim_{t \to \infty} \frac{a^2 t \E \left[ (\partial_y \widetilde \omega^\epl(X^\epl, Y^\epl, Z^\epl))^2 \right]}{\epl^2 (b + t)^{2a\E [X^\epl Z^\epl]}} \int_0^t (b + s)^{2a \E [X^\epl Z^\epl] - 2} \dd s \\
&= \frac{a^2 (1 + \delta)}{a-(1 + \delta)} \lim_{\epl \to 0} \frac1{\epl^2} \E \left[ (\partial_y \widetilde \omega^\epl(X^\epl, Y^\epl, Z^\epl))^2 \right].
\end{aligned}
\end{equation}
Then, from the expression \eqref{eq:solution_Poisson_tilde_CLT} we have
\begin{equation}
\partial_y \widetilde \omega(x,y,z) = - \frac{\epl}{\delta + \epl^2} \left( \frac{\epl^3}{1 + \epl^2} y + \frac{\epl^2}{1 + \epl^2} x + \delta z \right),
\end{equation}
and due to equation \eqref{eq:EZ2}, we obtain
\begin{equation}
\widetilde \Lambda = \frac{a^2 (1 + \delta)}{a-(1 + \delta)} \lim_{\epl \to 0} \frac{\delta^2}{(\delta + \epl^2)^2} \E[(Z^\epl)^2] = \frac{a^2 (1 + \delta)}{a - (1 + \delta)} \E[Z^2] = \frac{a^2}{2(a - (1 + \delta))}.
\end{equation}
We notice that minimizing the expression of $\widetilde \Lambda$ with respect to the parameter $a$ in the learning rate yields the optimal limit variance of the SGDCT estimator $\widetilde \Lambda = 2(1 + \delta)$, which is given by the choice $a = 2(1 + \delta)$ and coincides with the limit variance of the MLE. Moreover, we emphasize that $a$ must be chosen sufficiently large, i.e., $a > 1 + \delta$, in order for the CLT to hold, but this is not necessary for the convergence of the estimator, as already observed in \cite[Section 6]{SiS20} for single-scale SDEs. We remark that this formal computation is only the initial step towards a more rigorous analysis. In \cref{fig:CLT} we plot the histogram of the quantities 
\begin{equation}
\sqrt T \left( \widehat \theta^\delta_{\mathrm{exp}}(X^\epl,T) - \theta \right) \qquad \text{and} \qquad \sqrt t \left( \widetilde \theta^{\delta,\epl}_{\mathrm{exp},t} - \theta \right),
\end{equation}
with $1000$ realizations, and setting the correlation of the colored noise $\epl = 0.1$ and the final time $T = t = \epl^{-3} = 1000$. Moreover, we choose the filtering width $\delta = 1$ and the parameters in the learning rate $a = 4$, which is optimal as we computed above, and $b = 1$. We observe that the empirical distributions are nearly Gaussian, which demonstrates numerically the asymptotic normality of our estimators, and well approximated by the limit distributions predicted by the previous analysis. We finally recall that the limit variances are dependent on the filtering width $\delta$ and the parameter $a$ of the learning rate. Therefore, a thorough study of the asymptotic normality in the general framework would also be fundamental to give insight on the choice of these parameters.

\section{Conclusion} \label{sec:conclusion}

In this work, we considered SDEs driven by colored noise, modelled as Gaussian stationary processes with exponential autocorrelation function, i.e., a stationary Ornstein--Uhlenbeck process. In the limit as the correlation time of the noise goes to zero, the solution of the SDE converges to the solution of an SDE driven by white noise. We studied the problem of inferring unknown drift coefficients in the limit equation given continuous trajectories from the dynamics with colored noise, employing both MLE and SGDCT estimators. This is similar to the problem of inferring parameters in a coarse-grained SDE given observations of the slow variable in the full systems, a problem that has been extensively studied in previous works by our group~\cite{AGP21,APZ22,GaZ23, PaS07, PPS09, KKP15, KPK13}. We first focused on the case of additive noise and noticed that, without preprocessing the data, it is not possible to retrieve the exact drift coefficient, due to the incompatibility between the original observations and the limit equation, as first observed in \cite{PaS07} for multiscale diffusions. We overcame this issue by introducing filtered data, as in~\cite{AGP21}, obtained by convolving the original trajectory with an exponential kernel, in the definition of the estimators. We proved both theoretically and demonstrated through numerical experiments that the estimators developed in this paper are asymptotically unbiased, i.e., that they converge to the exact parameter values in joint limit as the observation time tends to infinity and the correlation time of the noise goes to zero.  Moreover, we applied our estimators to SDEs driven by multiplicative colored noise, for which an additional drift term, due to the Lévy area correction, appears in the limiting SDE. We showed that even in this case our methodology allows us to effectively infer the drift coefficients. We consider this to be an interesting result, since it is not clear at all that an MLE or SGDCT-based inference methodology can identify and learn the L\'{e}vy area correction. 

The results presented in this paper can be improved and extended in many different directions. First, the theoretical analysis for the SGDCT estimator is restricted to stochastic processes on a compact state space, i.e., the multidimensional torus. However, as suggested by the numerical examples, we believe that this restriction is not necessary, and that similar convergence results can be proved for colored noise-driven SDEs in unbounded domains. Since the focus of this paper was the development of the new inference methodologies, we chose to work on the torus in order to avoid technical issues related to the study of hypoelliptic PDEs in unbounded domains.

Second, in the study of the identifiability of the L\'{e}vy area correction we considered the particular example where both the drift functions in the colored multiplicative SDE are linear in the unknown coefficients. In this case simple analytical formulas for the L\'{e}vy area correction exist~\cite[Section 5.1]{Pav14} and it is straightforward to compare between theory and the results of numerical simulations. We believe, however, that the filtered data methodology can be employed in a much more general setting.  In particular, it would be interesting to infer drift functions in colored SDEs with multiplicative noise that depend nonlinearly on the parameters and, possibly, also in a nonparametric form. Furthermore, starting from \cref{sec:CLT} we would also like to obtain rigorous convergence rates and central limit theorems, and therefore asymptotic normality, both for the MLE and for the SGDCT estimators.

Lastly, in the present paper we considered the case of continuous, uncorrupted by noise, observations. It is important to extend our methodology so that it applies to the realistic case of discrete-in-time, noisy observations, both in the low and high frequency regimes. Naturally, the ultimate goal of this project is to apply our inference methodologies to real data. We plan to return to all these issues in future work. 

\subsection*{Acknowledgements} 

The authors would like to thank Markus Reiss for feedback on an earlier version of the paper. GP is partially supported by an ERC-EPSRC Frontier Research Guarantee through Grant No. EP/X038645, ERC Advanced Grant No. 247031, and a Leverhulme Trust Senior Research Fellowship, SRF{\textbackslash}R1{\textbackslash}241055. The work of SR has been partially funded by Deutsche Forschungsgemeinschaft (DFG) through the grant CRC 1114 \lq Scaling Cascades in Complex Systems\rq \,(project number 235221301). The work of AZ was partially supported by the Swiss National Science Foundation, under grant No. 200020\_172710.

\enlargethispage{0.2cm}
 
\bibliographystyle{siamnodash}
\bibliography{biblio}

\end{document}

%% file: ex_shared.tex
\usepackage[T1]{fontenc}
\usepackage{lmodern}
\usepackage[utf8]{inputenc}
\usepackage{microtype}
\usepackage{framed}
\usepackage{listings}
\usepackage{vmargin}
\usepackage{setspace}
\usepackage{mathrsfs, mathenv}
\usepackage{amsmath, amsthm, amssymb, amsfonts, amscd}
\usepackage{graphicx}
\usepackage{epstopdf}
\usepackage[svgnames]{xcolor}
\usepackage{hyperref}
\hypersetup{citecolor=blue, colorlinks=true, linkcolor=black}
\usepackage[ruled, vlined]{algorithm2e}
\usepackage[capitalise]{cleveref}
\usepackage{subcaption}
\usepackage{bbm}
\usepackage{cite}
\usepackage{verbatim}
\usepackage{pgfplots}
\usepackage{tikz}
\usetikzlibrary{arrows,decorations.pathmorphing,backgrounds,positioning,fit,matrix}
\usepackage{etoolbox}
\usepackage{color}
\usepackage{lipsum}
\usepackage{ifthen}
\usepackage[title]{appendix}
\usepackage{autonum}
\usepackage{accents}
\usepackage{xpatch}
\usepackage[]{algpseudocode}
\usepackage{multicol}
\usepackage[abs]{overpic}
\usepackage{eqnarray}

\crefname{assumption}{Assumption}{Assumptions}
\crefname{figure}{Figure}{Figures}
\theoremstyle{plain}
\newtheorem{theorem}{Theorem}[section]

\newtheorem{lemma}[theorem]{Lemma}
\newtheorem{proposition}[theorem]{Proposition}
\numberwithin{equation}{section}

\theoremstyle{definition}

\theoremstyle{remark}
\newtheorem{remark}[theorem]{Remark}
\newtheorem{assumption}[theorem]{Assumption}
\newtheorem{example}[theorem]{Example}

\ifpdf
  \DeclareGraphicsExtensions{.eps,.pdf,.png,.jpg}
\else
  \DeclareGraphicsExtensions{.eps}
\fi

\usepackage{mathtools}
\usepackage{booktabs}

\usepackage{pgfplots}
\usepackage{tikz}
\usetikzlibrary{patterns,arrows,decorations.pathmorphing,backgrounds,positioning,fit,matrix}
\usepackage[labelfont=bf]{caption}
\usepackage{here}
\usepackage[font=normal]{subcaption}

\usepackage{enumitem}
\setlist[itemize]{leftmargin=.5in}
\setlist[enumerate]{leftmargin=.5in,topsep=3pt,itemsep=3pt,label=(\roman*)}

\newcommand{\email}[1]{\href{#1}{#1}}
\newcommand{\TheTitle}{Filtered data based estimators for stochastic processes \\ driven by colored noise}
\newcommand{\TheAuthors}{G. A. Pavliotis, S. Reich, A. Zanoni}
\title{\TheTitle}
\author{Grigorios A. Pavliotis \thanks{Department of Mathematics, Imperial College London, London, UK, \email{g.pavliotis@imperial.ac.uk}.}
\and Sebastian Reich \thanks{Department of Mathematics, University of Potsdam, Potsdam, Germany, \email{sebastian.reich@uni-potsdam.de}.}
\and Andrea Zanoni \thanks{Institute for Computational and Mathematical Engineering, Stanford University, Stanford, CA, USA, \email{andrea.zanoni@stanford.edu}.}
}
\date{}

\usepackage{amsopn}

\newcommand{\abs}[1]{\left\lvert#1\right\rvert}
\newcommand{\norm}[1]{\left\|#1\right\|}
\renewcommand{\Pr}{\mathbb{P}}

\newcommand{\N}{\mathbb{N}}

\newcommand{\R}{\mathbb{R}}
\newcommand{\T}{\mathbb{T}}

\newcommand{\epl}{\varepsilon}

\newcommand{\eqdef}{\eqqcolon}

\newcommand{\E}{\operatorname{\mathbb{E}}}

\renewcommand{\d}{\mathrm{d}}
\newcommand{\dd}{\,\mathrm{d}}
\definecolor{shade}{RGB}{100, 100, 100}
\definecolor{bordeaux}{RGB}{128, 0, 50}

\renewcommand*{\dot}[1]{\accentset{\mbox{\large\bfseries .}}{#1}}

\usepackage[usestackEOL]{stackengine}

\definecolor{leg1}{RGB}{0,114,189}
\definecolor{leg2}{RGB}{217,83,25}
\definecolor{leg3}{RGB}{237,177,32}
\definecolor{leg4}{RGB}{126,47,142}
\definecolor{leg5}{RGB}{119,172,48}

\definecolor{leg21}{RGB}{62,38,169}
\definecolor{leg22}{RGB}{46,135,247}
\definecolor{leg23}{RGB}{55,200,151}
\definecolor{leg24}{RGB}{254,195,56}

\ifpdf
\hypersetup{
	pdftitle={\TheTitle},
	pdfauthor={\TheAuthors}
}
\fi

%% file: main.bbl
\begin{thebibliography}{10}

\bibitem{AGP21}
{\sc A.~Abdulle, G.~Garegnani, G.~A. Pavliotis, A.~M. Stuart, and A.~Zanoni},
  {\em Drift estimation of multiscale diffusions based on filtered data},
  Found. Comput. Math., 23 (2023), pp.~33--84.

\bibitem{APZ22}
{\sc A.~Abdulle, G.~A. Pavliotis, and A.~Zanoni}, {\em Eigenfunction martingale
  estimating functions and filtered data for drift estimation of discretely
  observed multiscale diffusions}, Stat. Comput., 32 (2022), pp.~Paper No. 34,
  33.

\bibitem{ABJ13}
{\sc R.~Azencott, A.~Beri, A.~Jain, and I.~Timofeyev}, {\em Sub-sampling and
  parametric estimation for multiscale dynamics}, Commun. Math. Sci., 11
  (2013), pp.~939--970.

\bibitem{BaP80}
{\sc I.~V. Basawa and B.~L.~S. Prakasa~Rao}, {\em Statistical inference for
  stochastic processes}, Probability and Mathematical Statistics, Academic
  Press, Inc. [Harcourt Brace Jovanovich, Publishers], London-New York, 1980.

\bibitem{BMP90}
{\sc A.~Benveniste, M.~M\'{e}tivier, and P.~Priouret}, {\em Adaptive algorithms
  and stochastic approximations}, vol.~22 of Applications of Mathematics (New
  York), Springer-Verlag, Berlin, 1990.
\newblock Translated from the French by Stephen S. Wilson.

\bibitem{BeT00}
{\sc D.~P. Bertsekas and J.~N. Tsitsiklis}, {\em Gradient convergence in
  gradient methods with errors}, SIAM J. Optim., 10 (2000), pp.~627--642.

\bibitem{Bis08}
{\sc J.~P.~N. Bishwal}, {\em Parameter estimation in stochastic differential
  equations}, vol.~1923 of Lecture Notes in Mathematics, Springer, Berlin,
  2008.

\bibitem{BlP78}
{\sc G.~Blankenship and G.~C. Papanicolaou}, {\em Stability and control of
  stochastic systems with wide-band noise disturbances. {I}}, SIAM J. Appl.
  Math., 34 (1978), pp.~437--476.

\bibitem{BoC13}
{\sc S.~Bo and A.~Celani}, {\em White-noise limit of nonwhite nonequilibrium
  processes}, Phys. Rev. E, 88 (2013), p.~062150.

\bibitem{BrK22}
{\sc S.~L. Brunton and J.~N. Kutz}, {\em Data-driven science and
  engineering---machine learning, dynamical systems, and control}, Cambridge
  University Press, Cambridge, 2022.
\newblock Second edition [of 3930582].

\bibitem{DeH23}
{\sc L.~Della~Maestra and M.~Hoffmann}, {\em The {LAN} property for
  {M}c{K}ean-{V}lasov models in a mean-field regime}, Stochastic Process.
  Appl., 155 (2023), pp.~109--146.

\bibitem{DHP23}
{\sc T.~Diamantakis, D.~D. Holm, and G.~A. Pavliotis}, {\em Variational
  principles on geometric rough paths and the {L}\'{e}vy area correction}, SIAM
  J. Appl. Dyn. Syst., 22 (2023), pp.~1182--1218.

\bibitem{FrH20}
{\sc P.~K. Friz and M.~Hairer}, {\em A course on rough paths}, Universitext,
  Springer, Cham, [2020] \copyright 2020.
\newblock With an introduction to regularity structures, Second edition of [
  3289027].

\bibitem{GaS17}
{\sc S.~Gailus and K.~Spiliopoulos}, {\em Statistical inference for perturbed
  multiscale dynamical systems}, Stochastic Process. Appl., 127 (2017),
  pp.~419--448.

\bibitem{GSM19}
{\sc E.~Garc\'{\i}a-Portugu\'{e}s, M.~S\o~rensen, K.~V. Mardia, and
  T.~Hamelryck}, {\em Langevin diffusions on the torus: estimation and
  applications}, Stat. Comput., 29 (2019), pp.~1--22.

\bibitem{GaZ23}
{\sc G.~Garegnani and A.~Zanoni}, {\em Robust estimation of effective
  diffusions from multiscale data}, Commun. Math. Sci., 21 (2023),
  pp.~405--435.

\bibitem{HoL84}
{\sc W.~Horsthemke and R.~Lefever}, {\em Noise-induced transitions}, vol.~15 of
  Springer Series in Synergetics, Springer-Verlag, Berlin, 1984.
\newblock Theory and applications in physics, chemistry, and biology.

\bibitem{HMW19}
{\sc S.~Hottovy, A.~McDaniel, and J.~Wehr}, {\em A small delay and correlation
  time limit of stochastic differential delay equations with state-dependent
  colored noise}, J. Stat. Phys., 175 (2019), pp.~19--46.

\bibitem{HaJ94}
{\sc P.~Häunggi and P.~Jung}, {\em Colored Noise in Dynamical Systems}, John
  Wiley \& Sons, Ltd, 1994, pp.~239--326.

\bibitem{IkW77}
{\sc N.~Ikeda and S.~Watanabe}, {\em A comparison theorem for solutions of
  stochastic differential equations and its applications}, Osaka Math. J., 14
  (1977), pp.~619--633.

\bibitem{IkW89}
{\sc N.~Ikeda and S.~Watanabe}, {\em Stochastic differential equations and
  diffusion processes}, North-Holland Publishing Co., Amsterdam, second~ed.,
  1989.

\bibitem{KKP15}
{\sc S.~Kalliadasis, S.~Krumscheid, and G.~A. Pavliotis}, {\em A new framework
  for extracting coarse-grained models from time series with multiscale
  structure}, J. Comput. Phys., 296 (2015), pp.~314--328.

\bibitem{KeS99}
{\sc M.~Kessler and M.~S{\o}rensen}, {\em Estimating equations based on
  eigenfunctions for a discretely observed diffusion process}, Bernoulli, 5
  (1999), pp.~299--314.

\bibitem{KLO12}
{\sc T.~Komorowski, C.~Landim, and S.~Olla}, {\em Fluctuations in {M}arkov
  processes}, vol.~345 of Grundlehren der Mathematischen Wissenschaften
  [Fundamental Principles of Mathematical Sciences], Springer, Heidelberg,
  2012.
\newblock Time symmetry and martingale approximation.

\bibitem{KPK13}
{\sc S.~Krumscheid, G.~A. Pavliotis, and S.~Kalliadasis}, {\em Semiparametric
  drift and diffusion estimation for multiscale diffusions}, Multiscale Model.
  Simul., 11 (2013), pp.~442--473.

\bibitem{KPS04}
{\sc R.~Kupferman, G.~A. Pavliotis, and A.~M. Stuart}, {\em It\^{o} versus
  {S}tratonovich white-noise limits for systems with inertia and colored
  multiplicative noise}, Phys. Rev. E (3), 70 (2004), pp.~036120, 9.

\bibitem{KuY03}
{\sc H.~J. Kushner and G.~G. Yin}, {\em Stochastic approximation and recursive
  algorithms and applications}, vol.~35 of Applications of Mathematics (New
  York), Springer-Verlag, New York, second~ed., 2003.
\newblock Stochastic Modelling and Applied Probability.

\bibitem{Kut04}
{\sc Y.~A. Kutoyants}, {\em Statistical inference for ergodic diffusion
  processes}, Springer Series in Statistics, Springer-Verlag London, Ltd.,
  London, 2004.

\bibitem{MSH02}
{\sc J.~C. Mattingly, A.~M. Stuart, and D.~J. Higham}, {\em Ergodicity for
  {SDE}s and approximations: locally {L}ipschitz vector fields and degenerate
  noise}, Stochastic Process. Appl., 101 (2002), pp.~185--232.

\bibitem{MST10}
{\sc J.~C. Mattingly, A.~M. Stuart, and M.~V. Tretyakov}, {\em Convergence of
  numerical time-averaging and stationary measures via {P}oisson equations},
  SIAM J. Numer. Anal., 48 (2010), pp.~552--577.

\bibitem{PPS09}
{\sc A.~Papavasiliou, G.~A. Pavliotis, and A.~M. Stuart}, {\em Maximum
  likelihood drift estimation for multiscale diffusions}, Stochastic Process.
  Appl., 119 (2009), pp.~3173--3210.

\bibitem{PaV01}
{\sc E.~Pardoux and A.~Y. Veretennikov}, {\em On the {P}oisson equation and
  diffusion approximation. {I}}, Ann. Probab., 29 (2001), pp.~1061--1085.

\bibitem{PaV03}
{\sc E.~Pardoux and A.~Y. Veretennikov}, {\em On {P}oisson equation and
  diffusion approximation. {II}}, Ann. Probab., 31 (2003), pp.~1166--1192.

\bibitem{PaV05}
{\sc E.~Pardoux and A.~Y. Veretennikov}, {\em On the {P}oisson equation and
  diffusion approximation. {III}}, Ann. Probab., 33 (2005), pp.~1111--1133.

\bibitem{Pav14}
{\sc G.~A. Pavliotis}, {\em Stochastic processes and applications}, vol.~60 of
  Texts in Applied Mathematics, Springer, New York, 2014.
\newblock Diffusion processes, the Fokker-Planck and Langevin equations.

\bibitem{PaS05}
{\sc G.~A. Pavliotis and A.~M. Stuart}, {\em Analysis of white noise limits for
  stochastic systems with two fast relaxation times}, Multiscale Model. Simul.,
  4 (2005), pp.~1--35 (electronic).

\bibitem{PaS07}
{\sc G.~A. Pavliotis and A.~M. Stuart}, {\em Parameter estimation for
  multiscale diffusions}, J. Stat. Phys., 127 (2007), pp.~741--781.

\bibitem{PaS08}
{\sc G.~A. Pavliotis and A.~M. Stuart}, {\em Multiscale methods}, vol.~53 of
  Texts in Applied Mathematics, Springer, New York, 2008.
\newblock Averaging and homogenization.

\bibitem{PaZ22a}
{\sc G.~A. Pavliotis and A.~Zanoni}, {\em Eigenfunction {M}artingale
  {E}stimators for {I}nteracting {P}article {S}ystems and {T}heir {M}ean
  {F}ield {L}imit}, SIAM J. Appl. Dyn. Syst., 21 (2022), pp.~2338--2370.

\bibitem{PaZ22b}
{\sc G.~A. Pavliotis and A.~Zanoni}, {\em A method of moments estimator for
  interacting particle systems and their mean field limit}, SIAM/ASA J.
  Uncertain. Quantif., 12 (2024), pp.~262--288.

\bibitem{PMH13}
{\sc G.~Pesce, A.~McDaniel, S.~Hottovy, J.~Wehr, and G.~Volpe}, {\em
  {S}tratonovich-to-{I}t\^{o} transition in noisy systems with multiplicative
  feedback}, Nat. Commun., 4 (2013), p.~2733.

\bibitem{RaR81}
{\sc B.~L.~S. Prakasa~Rao and H.~Rubin}, {\em Asymptotic theory of estimation
  in nonlinear stochastic differential equations}, Sankhy\={a} Ser. A, 43
  (1981), pp.~170--189.

\bibitem{Rei22}
{\sc S.~Reich}, {\em Robust parameter estimation using the ensemble {K}alman
  filter}.
\newblock Preprint arXiv:2201.00611, 2023.

\bibitem{SKP23}
{\sc L.~Sharrock, N.~Kantas, P.~Parpas, and G.~A. Pavliotis}, {\em Online
  parameter estimation for the {M}c{K}ean-{V}lasov {S}{D}{E}}, Stoch. Proc.
  Appl., 162 (2023), pp.~481--546.

\bibitem{SiS17}
{\sc J.~Sirignano and K.~Spiliopoulos}, {\em Stochastic gradient descent in
  continuous time}, SIAM J. Financial Math., 8 (2017), pp.~933--961.

\bibitem{SiS20}
{\sc J.~Sirignano and K.~Spiliopoulos}, {\em Stochastic gradient descent in
  continuous time: a central limit theorem}, Stoch. Syst., 10 (2020),
  pp.~124--151.

\bibitem{Sus78}
{\sc H.~J. Sussmann}, {\em On the gap between deterministic and stochastic
  ordinary differential equations}, Ann. Probability, 6 (1978), pp.~19--41.

\bibitem{Sus91}
{\sc H.~J. Sussmann}, {\em Limits of the {W}ong-{Z}akai type with a modified
  drift term}, in Stochastic analysis, Academic Press, Boston, MA, 1991,
  pp.~475--493.

\bibitem{VoW16}
{\sc G.~Volpe and J.~Wehr}, {\em Effective drifts in dynamical systems with
  multiplicative noise: a review of recent progress}, Reports on Progress in
  Physics, 79 (2016), p.~053901.

\bibitem{WoZ65}
{\sc E.~Wong and M.~Zakai}, {\em On the convergence of ordinary integrals to
  stochastic integrals}, Ann. Math. Statist., 36 (1965), pp.~1560--1564.

\bibitem{WoZ69}
{\sc E.~Wong and M.~Zakai}, {\em Riemann-{S}tieltjes approximations of
  stochastic integrals}, Z. Wahrscheinlichkeitstheorie und Verw. Gebiete, 12
  (1969), pp.~87--97.

\bibitem{Wu_09}
{\sc L.~Wu}, {\em Gradient estimates of {P}oisson equations on {R}iemannian
  manifolds and applications}, J. Funct. Anal., 257 (2009), pp.~4015--4033.

\bibitem{Zak67}
{\sc M.~Zakai}, {\em Some moment inequalities for stochastic integrals and for
  solutions of stochastic differential equations}, Israel J. Math., 5 (1967),
  pp.~170--176.

\bibitem{Zan22}
{\sc A.~Zanoni}, {\em Filtered data and eigenfunction estimators for
  statistical inference of multiscale and interacting diffusion processes},
  (2022), p.~231.

\bibitem{Zan23}
{\sc A.~Zanoni}, {\em Homogenization results for the generator of multiscale
  {L}angevin dynamics in weighted {S}obolev spaces}, IMA J. Appl. Math., 88
  (2023), pp.~67--101.

\end{thebibliography}
